\documentclass[reqno]{amsart}
\usepackage{mathrsfs}
\usepackage{amsmath,amsfonts,amssymb,amsthm}
\usepackage{upgreek}
\usepackage[usenames,dvipsnames]{xcolor}
\usepackage{enumitem}
\usepackage{graphicx}
\usepackage{overpic}
\usepackage{cancel}
\usepackage[outdir=./]{epstopdf}
\usepackage[colorlinks=true]{hyperref} 
\usepackage{url}
\usepackage[utf8]{inputenc}
\usepackage{booktabs,multirow}
\usepackage[font=footnotesize,width=\textwidth]{caption}
\hypersetup{linkcolor=Violet,citecolor=PineGreen}
\newtheorem{teor}{Theorem}[section]
\newtheorem{lema}[teor]{Lemma}
\newtheorem{prop}[teor]{Proposition}
\newtheorem{coro}[teor]{Corollary}
\theoremstyle{definition}
\newtheorem{defi}[teor]{Definition}
\newtheorem{exa}[teor]{Example}

\newtheorem{nota}[teor]{Remark}
\newtheorem{notas}[teor]{Remarks}
\numberwithin{equation}{section}

\newcommand{\R}{\mathbb R}

\newcommand{\N}{\mathbb{N}}

\newcommand{\Y}{\mathbb{Y}}

\newcommand{\mA}{\mathcal{A}}
\newcommand{\mB}{\mathcal{B}}
\newcommand{\mC}{\mathcal{C}}
\newcommand{\mD}{\mathcal{D}}

\newcommand{\mK}{\mathcal{K}}
\newcommand{\mM}{\mathcal{M}}
\newcommand{\mN}{\mathcal{N}}
\newcommand{\mO}{\mathcal{O}}
\newcommand{\mP}{\mathcal{P}}

\newcommand{\mS}{\mathcal{S}}
\newcommand{\mU}{\mathcal{U}}
\newcommand{\mV}{\mathcal{V}}
\newcommand{\ep}{\varepsilon}

\newcommand{\mI}{\mathcal{I}}

\newcommand{\W}{\Omega}
\newcommand{\Wf}{{\W_\hf}}

\newcommand{\WRn}{{\Wf\times\R^n}}

\newcommand{\w}{\omega}
\newcommand{\lb}{\lambda}

\newcommand{\ma}{\mathfrak{a}}
\newcommand{\mb}{\mathfrak{b}}
\newcommand{\mc}{\mathfrak{c}}

\newcommand{\mr}{\mathfrak{r}}
\newcommand{\ml}{\mathfrak{l}}

\newcommand{\muk}{\mathfrak{u}}
\newcommand{\mf}{\mathfrak{f}}
\newcommand{\mg}{\mathfrak{g}}

\newcommand{\pb}{\mathfrak{p_b}}
\newcommand{\mpp}{\mathfrak{p}}

\newcommand{\upalfa}{$\upalpha$}
\newcommand{\upomeg}{$\upomega$}

\newcommand{\wit}{\widetilde}

\newcommand{\ws}{\w{\cdot}s}
\newcommand{\wt}{\w{\cdot}t}

\newcommand{\bwt}{\bar\w{\cdot}t}

\newcommand{\hf}{{\hat f}}

\newcommand{\pren}[1]{\left(#1\right)}

\newcommand{\lsm}{\left[\begin{smallmatrix}}
\newcommand{\rsm}{\end{smallmatrix}\right]}

\DeclareMathOperator{\dist}{dist}
\DeclareMathOperator{\di}{d}

\begin{document}
\title[Averaging and tracking of local attractors]
{Averaging and tracking of local attractors in slowly varying systems with two time scales}
\author[C. N\'{u}\~{n}ez]{Carmen N\'{u}\~{n}ez}
\author[R. Obaya]{Rafael Obaya}
\author[J. Rodr\'{\i}guez]{Jorge Rodr\'{\i}guez}
\address{Departamento de Matem\'{a}tica Aplicada, Universidad de Va\-lladolid,
Paseo Prado de la Magdalena 3-5, 47011 Valladolid, Spain.}
\address{{}The three authors are members of IMUVA: Instituto de Investigaci\'{o}n en Matem\'{a}ticas,
Universidad de Valladolid.}
\email[C.~N\'{u}\~{n}ez]{carmen.nunez@uva.es}
\email[R.~Obaya]{rafael.obaya@uva.es}
\email[J.~Rodr\'{\i}guez]{jorge.rodriguez@uva.es}

\thanks{All the authors were supported by Ministerio de Ciencia, Innovaci\'{o}n y
Universidades (Spain) under project PID2024-156691NB-I00. J. Rodr\'{\i}guez was also
supported by Ministerio de Ciencia, Innovaci\'{o}n y
Universidades (Spain) under programme FPU24/00913.}
\date{}
\begin{abstract}
The paper analyzes to what extent the dynamics of a nonautonomous
$n$-dimensional dynamical system with two time scales, formulated
in the slow time as $dx/dt=f(t/\ep, t, x)$, can be approximated for small values
of $\ep$ by the dynamics of the averaged system $dz/dt=\hf(t,z)$.
Assuming that the skewproduct
flow associated with the averaged system admits a local attractor $\mA$, we prove
that the solutions of the original system whose initial data lie in the basin
of attraction of $\mA$ track the fibers of the inflated attractor for all positive
times. If the fiber map of $\mA$ is continuous, inflation is no longer required.
Alternative tracking results with a more classical formulation are also presented,
under assumptions involving uniformly asymptotically stable solutions
or uniform local attractors for the nonautonomous process, rather than for the
skewproduct flow. Several examples illustrate the scope and applicability of
the results. The twofold extension of the classical averaging results 
(to the doubly nonautonomous setting and to the whole positive halfline) 
is expected to be relevant to a broad range of application.
\end{abstract}
\renewcommand{\subjclassname}{\textup{2020} Mathematics Subject Classification}

\maketitle
\section{Introduction}
We consider a nonautonomous $n$-dimensional dynamical system defined on a fast time scale $\tau$ by
\begin{equation}\label{eq:1tau}
\frac{dx}{d\tau}=\ep\,f(\tau,\ep\tau,x),
\end{equation}
where $0<\ep\ll 1$ is a small parameter.
In this formulation, the state vector $x$ evolves slowly, since its derivative is scaled by
the small parameter $\ep$. At the same time, the vector field $f$ depends on two time scales: a fast one,
through $\tau$, generating rapid temporal variations, and a slow one, through
$t=\ep\tau$, which modulates the dynamics.
The system is rewritten in terms of the slow time variable $t=\ep\tau$ as
\begin{equation}\label{eq:1t}
 \frac{dx}{dt}=f(t/\ep,t,x)\,.
\end{equation}
Here, $t/\ep$ represents the fast variable.
Some regularity hypotheses on $f$ with respect to its temporal arguments
imply the existence and regularity of the average function
$\hf(t,z):=\lim_{T\to\infty}(1/T)\int_\tau^{\tau+T} f(s,t,z)\,ds$,
as well as the uniformity of the limit with respect to $\tau$. The map $\hf$
defines the nonautonomous averaged system
\begin{equation}\label{eq:1promt}
 \frac{dz}{dt}=\hf(t,z)\,,
\end{equation}
which describes the effective slow dynamics associated
with the original two-scale system. Our results focus on the tracking of local
nonautonomous attractors of \eqref{eq:1promt} on the whole positive half-line:
roughly speaking, we prove that a local attractor of \eqref{eq:1promt}
determines the asymptotic behavior of the solutions of \eqref{eq:1t} starting
near it, with increasing precision as $\ep$ decreases.

Our tracking results are strongly based on a suitable version
of the averaging principle. Let us put this into context.
Classical averaging theory provides relevant dynamical information on the solutions
of \eqref{eq:1t} in terms of the solutions of \eqref{eq:1promt} when the vector field $f$ does not
depend on the slow variable. Thus, the solutions to be compared are those of
$dx/dt=g(t/\ep,x)$ and $dz/dt=\hat g(z)$, where $\hat g$ is the time-average of $g$.
Foundational contributions are due, among others, to Bogoliubov and Mitropolsky \cite{bomi},
Hale \cite{hale}, Mitropolsky \cite{mitr}, Lochak and Meunier \cite{lomu}, and Arnold \cite{arnol},
while systematic modern accounts can be found in the books by Sanders and Verhulst \cite{save}
and by Sanders {\em et al.}~\cite{savm}.
Some of these works suggest treating $t=\ep\tau$ in \eqref{eq:1tau} as a new state variable
by adding the equation $dt/d\tau=\ep$ in order to deal with the doubly nonautonomous case,
so that the previous results can be applied.
But this idea requires unnecessarily strong conditions on the dependence of $f$ on $t$,
and at the same time leads to the absence of globally bounded solutions and so
precludes the use of many useful tools.

Our approach to averaging in the doubly nonautonomous case is different.
The classical averaging principle provides estimates comparing the solutions of the
original and averaged systems on fast time intervals $[0,c/\ep]$, for fixed $c>0$.
In \cite{arst}, Artstein introduces the notion of averaging pair, an efficient
tool for bounding the difference between the solutions of $dx/dt=g(t/\ep,x)$ and $dz/dt=\hat g(z)$.
Following this line, in \cite{nuor1} we describe certain conditions on $f$ (including
those required for $\hf$ in this work) that provide bounds for the difference
between the solutions of \eqref{eq:1t} and \eqref{eq:1promt}, which are uniform
with respect to the initial time (at which the solutions of both systems coincide)
and with respect to the initial value when it varies over a compact subset of $\R^n$.
These uniform estimates are fundamental for obtaining the main conclusions of this article
regarding tracking.

In parallel with the averaging analysis, the use of tracking methods for slowly varying attracting
objects has a long tradition in the description of the fast state variable of autonomous singularly
perturbed systems of differential equations. The original idea stems from the combined results
of Tikhonov \cite{tikh} and Fenichel \cite{feni}. A detailed description of this theory
can be found in Verhulst \cite{verl}, Berglund and Gentz \cite{bege}, and Kuehn \cite{kueh}.
In \cite{loos}, Longo {\em et al.} develop a nonautonomous version of this tracking theory, applicable
to singularly perturbed systems with explicit dependence on the fast time.

As said before, our results focus on the tracking of the
slow state variable $x$ in \eqref{eq:1tau} or \eqref{eq:1t}.
The ideas underlying the description that we provide differ from those in the mentioned works.
The regularity properties that we assume on $f$ allow the
compactification of the set of $t$-shifts of $\hf$, which in turn permits the use of classic
methods from dynamical systems in the analysis. More precisely,
we embed the averaged dynamics into a skewproduct flow defined on the product of the hull
$\Wf$ of $\hf$ and the phase space $\R^n$. In this extended phase space $\WRn$, local
attractors can be defined in the usual autonomous sense, and their fibers along the orbit of
$\hf$ in its hull provide the corresponding pullback attractors for the process defined
by the averaged system.

Our main tracking theorem, Theorem~\ref{teor:4tracking}, assumes the existence
of a local attractor $\mA$ for the skewproduct flow associated with the averaged
system which projects onto the whole $\Wf$. Under this main hypothesis,
we prove that the solutions of the original system with initial data in the basin
of attraction of $\mA$ track the fibers of the so-called {\em inflated attractor} along the
orbit of $\hf$ in $\Wf$ after a certain initial transient, and then for all
subsequent times. A key feature of this result is that, once this initial
approach period has elapsed, the tracking property holds for all positive times. 
Therefore, our extension of classical averaging results for ordinary 
differential equations is twofold: it provides not only an effective framework for 
analyzing multiscale systems, but also information about the long-term behavior of 
solutions on the whole positive half-line, rather than only on finite time intervals.
In general, the result does not hold without inflating the attractor. However, if
the corresponding fiber map is continuous with respect to the Hausdorff metric,
inflation is no longer necessary: the solutions track the fiber of the local
attractor itself, which leads to a much more precise description
of the evolution of the solutions. These remarks show that the analysis of the
continuity of the fiber map and, more generally, of the shape of the attractor fibers
over each point of the base, are relevant issues in this work,
which is why significant effort is also devoted to it.
Another relevant related property is that, even in the absence of this continuity,
when the average map $\hf$ is recurrent with respect to the time, the set
of times for which the distance from the values of the solution to the fibers of
the attractor is as uniformly small as desired is large, since it is relatively
dense and has positive lower density in $[0,\infty)$.

The paper includes two additional tracking results,
Theorem \ref{teor:6trackAtractor} and Corollary \ref{coro:6trackSol},
with a more classical formulation: in their statements,
neither the existence nor the properties of the skewproduct flow play a role. Instead, they require
the existence of a uniformly asymptotically stable solution of the averaged system or
of a bounded uniform local attractor for the corresponding nonautonomous process.
These hypotheses may hold even in the absence of a local attractor for the skewproduct flow
projecting onto $\Wf$.

The use of averaging results in tracking problems under attraction assumptions
goes back to the works of Eckhaus and S\'{a}nchez-Palencia (described in Verhulst
{\em et al.}~\cite{savm}), who considered ODEs of the form
$dx/dt=g(t/\ep,x)$ for which the corresponding averaged system $dz/dt=\hat g(z)$
admits an exponentially stable equilibrium.
Averaging theory has also been extensively applied to deterministic and stochastic
partial differential equations. For example, Dymnikov and Filatov~\cite{dyfi},
Vishik and Chepyzhov~\cite{vich}, and Li {\em et al.} \cite{liyt} analyze, in some dissipative
models, the continuous variation of the
global uniform attractor as the parameter $\ep$ tends to zero.

We emphasize that our theory is local in nature. Even in dissipative ordinary differential
equations, the
global attractor frequently contains several local attractors, which in our approach
capture the long-time behavior of solutions and hence constitute the relevant invariant
objects. Other notions of attractor sets with weaker attraction
properties have been considered in the literature (see, e.g., Milnor~\cite{miln},
Olja\u{c}a {\em et al.}~\cite{olar} and Newman {\em et al.}~\cite{near}), and they
allow a natural extension of the results of this article to interesting dynamical scenarios.

The versatility of the tracking framework presented in this paper
also suggests possible applications in several areas
where multiscale nonautonomous dynamics arise. For example, in celestial mechanics,
short orbital periods may be averaged out to reveal a slowly evolving secular dynamics.
In population dynamics, fast seasonal
forcing may interact with slower environmental changes and give rise to time-dependent ecological
regimes. More generally, the same point of view may be relevant in systems with rapidly varying
external forcing combined with slow temporal modulation, where the averaged equation captures the
effective attracting structures and the original dynamics evolves nearby.

Finally we point out that one of the key points of interest in tracking analysis
is its applicability to the study of critical transitions in
complex systems: see, e.g., Ashwin {\em et al.} \cite{aspw} and Longo {\em et al.} \cite{lnor}.
The results of this work may contribute to this line of research: by monitoring the
evolution of the nonautonomous local attractors
as the input of the system varies, one may detect bifurcation phenomena driven by
the slow time scale, where the loss of stability of an averaged attractor signals a
qualitative change in the dynamics induced by \eqref{eq:1t}.

The paper is organized as follows. Section~\ref{sec:2} briefly reviews basic
properties of the semidistance and the Hausdorff distance, as well as some
elementary facts on continuous flows on metric spaces. Section~\ref{sec:3}
introduces the hypotheses on the function $f$ defining the original system and its relation
with the averaged vector field, and describes the construction of the
hull and the associated skewproduct flow. In this framework, we introduce the
notions of local attractor and fiber map and discuss their main properties.
Section~\ref{sec:4} contains the main tracking results, a simple application, and
a detailed analysis of the hypotheses. Section~\ref{sec:5} is devoted to
examples illustrating the scope and limitations of the previous results and the
optimality of their hypotheses. Finally, Section~\ref{sec:6} presents tracking
results that admit a more classical formulation (since they do not depend on the
existence or properties of the skewproduct flow), at the price of assumptions
which may be stronger or harder to verify.

\section{Preliminaries}\label{sec:2}
This short section recalls concepts and properties of distances and flows on metric spaces which
will be required throughout the paper.

Let $\Y$ be a metric space with distance $\di_\Y$. We represent
\begin{itemize}[leftmargin=25pt]
\item[-] by $\dist_\Y(y,\mC):=\inf_{\bar y\in\mC}\di_\Y(y,\bar y)$ the distance from
    a point $y$ to a nonempty set $\mC\subset\Y$;
\item[-] by $\dist_\Y(\mC_1,\mC_2):=\sup_{y_1\in\mC_1}\big(\dist_{\Y}\big(y_1,\,\mC_2\big)\big)$
    the Hausdorff semidistance between two nonempty subsets $\mC_1$ and $\mC_2$ of $\Y$;
\item[-] by $\di^H_\Y\big(\mC_1,\,\mC_2\big):=\max\big(\dist_\Y\big(\mC_1,\,\mC_2\big),
    \,\dist_\Y\big(\mC_2,\,\mC_1\big)\big)$ the Hausdorff distance between two compact
    subsets $\mC_1$ and $\mC_2$ of $\Y$;
\item[-] $\mB_\mu(\mC):=\{y\in\Y\,|\;\dist_\Y\big(y,\,\mC\big)<\mu\}$ for $\mC\subset\Y$ and $\mu>0$,
    and $\bar\mB_\mu(\mC):=\text{closure}_{\Y}\,\mB_\mu(\mC)$.
\end{itemize}
We will use the following basic properties throughout the paper, without
making reference to this list:
\begin{itemize}[leftmargin=25pt]
\item[-] in general, $\dist_\Y\big(\mC_1,\,\mC_2\big)\ne\dist_\Y\big(\mC_2,\,\mC_1\big)$;
\item[-] if $\mC_1\subseteq\mC_2$, then $\dist_\Y\big(\mC_1,\,\mC_2\big)=0$ and
$\di^H_\Y\big(\mC_1,\,\mC_2\big)=\dist_\Y\big(\mC_2,\,\mC_1\big)$;
\item[-] if $\dist_\Y\big(\mC_1,\,\mC_2\big)<\mu$, then $\mC_1\subseteq\mB_\mu(\mC_2)$;
\item[-] if $\mC\subseteq\mD$, then $\dist_\Y\big(y,\,\mD\big)\le\dist_\Y\big(y,\,\mC\big)$;
\item[-] if $C_2\subseteq\mD_2$, then $\dist_\Y\big(\mC_1,\,\mD_2\big)\le
    \dist_\Y\big(\mC_1,\,\mC_2\big)$;
\item[-] if $C_1\subseteq\mD_1$, then $\dist_\Y\big(\mC_1,\,\mC_2\big)\le
    \dist_\Y\big(\mD_1,\,\mC_2\big)$;
\item[-] $\dist_\Y\big(\mC_1,\,\mC_2\big)\le \dist_\Y\big(\mC_1,\,\mC_3\big)+
    \dist_\Y\big(\mC_3,\,\mC_2\big)$ if $\mC_3$ is compact.
\end{itemize}

We will mainly work with three types of metric spaces: $\R^n$, subsets $\W$ of
$C(\R,\R^n)$ which are compact and metric for the compact-open topology (i.e., for the
topology of uniform convergence on compact subsets of $\R$), and products $\W\times\R^n$.
In the case of $\R^n$, we choose the Euclidean distance, $\di_{\R^n}(x_1,x_2):=|x_1-x_2|$,
and we will represent by $\mB_\mu$ the open ball centered at $0\in\R^n$ of radius $\mu>0$
and by $\bar\mB_\mu$ its closure. Note that $\W\times\R^n$ is also a metric space for
the distance $\di_{\W\times\R^n}$ defined by
\[
 \di_{\W\times\R^n}\!\big((\w_1,x_1),\,(\w_2,x_2)\big):=\max(\di_\W(\w_1,\w_2),|x_1-x_2|)\,,
\]
where $\di_\W$ is the distance on $\W$.

A continuous map $\phi\colon\mO_\phi\subseteq\R\times\Y\to\Y$, $(t,y)\mapsto\phi(t,y)$ defined on an open subset
$\mO_\phi\supset\{0\}\times\Y$ is a ({\em local\/}) {\em continuous flow\/} on $\Y$ if $\phi(0,y)=y$
for all $y\in\Y$ and $\phi(t+s,y)=\phi(t,\,\phi(s,y))$ whenever the right-hand term exists.
The flow $\phi$ is {\em global\/} if $\mO_\phi=\R\times\Y$. The $\phi$-{\em orbit of $y$}
is the set $\{\phi(t,y)\,|\;(t,y)\in\mO_\phi\}$, and $\{\phi(t,y)\,|\;(t,y)\in\mO_\phi,\;t\ge 0\}$
and $\{\phi(t,y)\,|\;(t,y)\in\mO_\phi,\;t\le 0\}$ are the
{\em forward} and {\em backward $\phi$-semiorbits of $y$}, respectively. The $\phi$-orbit
(resp.~a $\phi$-semiorbit) of $y$ is {\em globally defined\/} if $\{(t,y)\,|\;t\in\R\}\subset\mO_\phi$
(resp.~$\{(t,y)\,|\,\pm t\in\R_+\}\subset\mO_\phi$, where $\R_+:=[0,\infty)$).
The prefix $\phi$ will be often omitted.

Given a subset $\mC\subseteq\Y$, we call
$\phi_t(\mC):=\{\phi(t,y)\,|\;y\in\mC\text{\;\and\;}(t,y)\in\mO_\phi\}$.
The set $\mC\subseteq\Y$ is {\em $\phi$-invariant\/} if $\,\R\times\mC\subseteq\mO_\phi$
and $\phi_t(\mC)=\mC$ for all $t\in\R$.
The set $\mC$ is {\em positively $\phi$-invariant\/} if $\R_+\!\times\mC\subseteq\mO_\phi$
and $\phi_t(\mC)\subseteq\mC$ for all $t\in\R_+$.

The {\em \upomeg-limit} (resp. {\em \upalfa- limit}) {\em set} of a point $y\in\Y$ with globally defined
forward (resp.~backward) semiorbit is the set of all the possible limits of sequences of the form
$(\phi(t_k,y))$, where the sequence $(t_k)$ has limit $\infty$ (resp.~$-\infty$). It is a closed
$\phi$-invariant set which, in addition, is nonempty, compact and connected if the semiorbit is relatively
compact (see, e.g., \cite[Chapter I, Theorem 8.1]{hale}). A nonempty subset $\mM\subseteq\Y$ is {\em minimal}
if it is compact and $\phi$-invariant and it does not contain properly any other compact $\phi$-invariant
subset. It is easy to check that a compact set $\mM\subseteq\Y$ is $\phi$-minimal if and only if the
$\phi$-orbit of every $y\in\mM$ is dense in $\mM$, and to deduce that a $\phi$-minimal set coincides with
the \upalfa-limit and \upomeg-limit sets of any of its elements. In addition, every compact $\phi$-invariant
set contains some $\phi$-minimal set (see, e.g., \cite[Chapter I, Lemma 8.1]{hale}).

Assume now that $(\Y,\di_\Y)$ is a compact metric space and that the flow $\phi$
is globally defined. A (complete regular Borel)
normalized measure $m$ on $\Y$ is {\em $\phi$-invariant} if $m(\mC)=m(\phi(t,\mC))$
for all Borel subset $\mC\subseteq\Y$ and all $t\in\R$, and it is {\em $\phi$-ergodic}
if, in addition, $m(\mC)$ is 0 or 1 if $\mC$ is $\phi$-invariant.
The continuous flow $(\Y,\phi)$ is {\em minimal} if $\Y$ itself is $\phi$-minimal.
The flow $(\Y,\phi)$ is {\em almost periodic\/} or {\em equicontinuous} if,
for any $\ep>0$ there exists $\delta_\ep>0$ such that $\di_\Y(\phi(t,y_1),\phi(t,y_2))<\ep$ for all $t\in\R$
whenever $\di_\Y(y_1,y_2)<\delta_\ep$. An almost periodic and minimal flow is {\em uniquely ergodic}
(see, e.g., \cite[Corollary 2.10 of Part I]{shyi4} and \cite[Section 1.1]{rudi2}),
which means that there exists a unique $\phi$-invariant normalized measure $m$. In this case, $m$
is $\phi$-ergodic (see, e.g., \cite[Theorem 6.10]{walt}).
\section{Hypotheses, the hull construction, and local attractors}\label{sec:3}
This section is divided into the three subsections indicated in its title,
with each part building on the previous one. The third section contains
a proof of certain properties of the {\em local\/} attractors of skewproduct flows which, although
they are predictable, do not appear in the standard references
on this topic.
\subsection{The hypotheses on $f$}\label{subsec:31}
A nondecreasing map $\theta\colon\R_+\to\R_+\cup\{\infty\}$
with $\lim_{r\to 0^+}\theta(r)=0$ is a {\em modulus of continuity}.
The map $f\colon\R_+\!\times\R\times\R^n\to\R^n$ giving rise to the equation to be analyzed,
\eqref{eq:4initau}, is assumed to satisfy (all or part of) the following conditions:
\begin{enumerate}[leftmargin=25pt,label=\rm{\bf{f\arabic*}}]
\item\label{f1} for any $\mu>0$, $f$ is bounded and continuous when restricted to
    $\R_+\!\times\R\times\bar\mB_\mu$;
\item\label{f2} for any $\mu>0$, there exist a modulus of continuity $\theta_\mu$ and
    a positive constant $L_\mu$ such that $|f(\tau,t_1,x_1)-f(\tau,t_2,x_2)|\le
    \theta_\mu(|t_1-t_2|)+L_\mu|x_1-x_2|$ for all $\tau\in\R_+$, $t_1,t_2\in\R$,
    and $x_1,x_2\in\bar\mB_\mu$;
\item\label{f3} for any $(t,x)\in\R\times\R^n$, there exists the limit
    \begin{equation}\label{def:3promedio}
     \hf(t,x):=\lim_{T\to\infty} \frac{1}{T}\int_0^T f(s,t,x)\,ds\,;
    \end{equation}
\item\label{f4} for any $\mu>0$,
    \[
     \hf(t,x)=\lim_{T\to\infty} \frac{1}{T}\int_\tau^{T+\tau} f(s,t,x)\,ds
    \]
    uniformly on $(\tau,t,x)\in\R_+\!\times\R\times\bar\mB_\mu$.
\end{enumerate}
\begin{notas}
1.~It is easy to check that conditions \ref{f1} and \ref{f3}
ensure that the limit in \ref{f4} exists
for all $\tau\in\R$ and takes the same value as for $\tau=0$. What condition \ref{f4} provides
is the uniformity of the limiting behavior, which will be a fundamental tool in the
main results of this paper.

2.~Throughout the paper (more precisely, in Section \ref{sec:5})
we will also work with maps $f\colon\R_+\!\times\R_+\!\times\R^n\to\R^n$.
In this case, when we say that $f$ satisfies one of the above conditions,
we mean that this condition is satisfied
when we consider only nonnegative values of the second variable $t$.
The results stated in the next proposition remain valid,
now for the map $\hf$ on $\R_+\!\times\R^n$.
\end{notas}

\begin{prop}\label{prop:3prophf}
Let $\mu$ be any positive constant, and let $f$ satisfy \ref{f3}.
\begin{itemize}
\item[\rm(i)] If $f$ satisfies \ref{f1}, then $\hf$ is
bounded on $\R\times\bar\mB_\mu$;
\item[\rm(ii)] if $f$ satisfies \ref{f2}, then
$|\hf(t_1,x_1)-\hf(t_2,x_2)|\le \theta_\mu(|t_1-t_2|)+L_\mu|x_1-x_2|$
for all $t_1,t_2\in\R$, and $x_1,x_2\in\bar\mB_\mu$; in particular,
$\hf$ is uniformly continuous on $\R\times\bar\mB_\mu$.
\end{itemize}
\end{prop}
\begin{proof}
Property (i) is trivial, and (ii) is an easy consequence of
\[
 |\hf(t_1,x_1)-\hf(t_2,x_2)|\le
 \lim_{T\to\infty}\frac{1}{T}\int_0^{T} |f(s,t_1,x_1)-f(s,t_2,x_2)|\,ds
\]
for all $(t_1,x_1),\,(t_2,x_2)\in\R\times\R^n$, which is ensured by \ref{f3}.
\end{proof}
\subsection{The hull framework}\label{subsec:32}
Throughout Section \ref{subsec:32},
we assume \ref{f1}, \ref{f2} and \ref{f3}.
Given a continuous map $h\colon\R\times\R^n\to\R^n$, we define
$h{\cdot}s\colon\R\times\R^n\to\R^n$
by $h{\cdot}s(t,x):=h(t+s,x)$ for $s\in\R$. The properties of $\hf$ established in Proposition
\ref{prop:3prophf} show that the hull $\Wf$ of $\hf$, defined by
\begin{equation}\label{def:3hull}
 \Wf:=\text{closure}\{\hf{\cdot}s\,|\;s\in\R\}\subset C(\R\times\R^n,\R^n)
\end{equation}
in the compact-open topology of the set of continuous maps, is a compact metric space
for a distance $\di_\Wf$ (whose particular definition does not play a role in this paper),
as well as the continuity of the global flow
\[
 \sigma_\hf\colon\R\times\Wf\to\Wf\,,\quad (t,\w)\mapsto\wt
\]
(see, e.g., \cite[Section III.I]{sell2}). We define
\[
 \mf\colon\WRn\to\R^n\,,\quad (\w,x)\mapsto \w(0,x)\,,
\]
and observe that it is a continuous map and that $\mf(\wt,x)=\wt(0,x)=\w(t,x)$. So,
\[
 \frac{dz}{dt}=\hf(t,z)
\]
is one of the equations of the family
\begin{equation}\label{eq:3hull}
 \frac{dz}{dt}=\mf(\wt,z)\,,\quad\w\in\Wf:
\end{equation}
it corresponds to $\w=\hf$.
It is easy to check that any $\w\colon\R\times\R^n\to\R^n$
satisfies the properties formulated for $\hf$ in Proposition \ref{prop:3prophf}.
The family \eqref{eq:3hull} induces the (possibly local) continuous flow
\begin{equation}\label{def:3pi}
 \pi\colon\mO_\pi\subseteq\R\times\WRn\to\WRn\,,
\quad(t,\w,x)\mapsto(\wt,v(t;\w,x))\,,
\end{equation}
where $t\mapsto v(t;\w,x)$ is the maximal solution of the equation \eqref{eq:3hull}
corresponding to $\w$ satisfying $v(0;\w,x)=x$ (see, e.g., \cite[Theorem IV.3]{sell2}).
A flow like $\pi$ is called a {\em skewproduct} flow,
since it is defined on a bundle $\WRn$ and preserves the {\em base flow\/} $\sigma_\hf$
(that is, the flow on the {\em base\/} $\Wf$ of the bundle).
\begin{nota}\label{nota:3union}
For further purposes we observe that
$\Wf=\upalpha(\hf)\cup\{\hf{\cdot}s\,|\;s\in\R\}\cup\upomega(\hf)$, where
$\upalpha(\hf)$ and $\upomega(\hf)$ are respectively the \upalfa-limit and \upomeg-limit sets
of $\hf$ for the flow $\sigma_\hf$ (see Section \ref{sec:2}).
A proof of this basic property can be found in \cite[Lemma 2.4]{duno4}.
Note that the three sets
of the union are not necessarily disjoint. For example, the three of them coincide
if $\hf$ is periodic in $t$ (which is the case if $f$ is periodic in $t$);
and $\upalpha(\hf)=\upomega(\hf)=\{0\}$ if $\hf(t,x)=g(t)$ for a uniformly continuous
map $g\colon\R\to\R$ with $\lim_{t\to\pm\infty} g(t)=0$.
\end{nota}
\subsection{Local attractors}\label{subsec:33}
The results of this section concern the skewproduct flow $(\WRn,\pi)$ projecting onto
$(\Wf,\sigma_\hf)$, defined in Section \ref{subsec:32}. So, conditions
\ref{f1}, \ref{f2} and \ref{f3} are again in force.

Let us consider a set $\mC\subseteq\WRn$. The definitions of $\pi(t;\mC)$ for $t\in\R$ and
$\pi$-invariance and positive $\pi$-invariance of $\mC$ are given in Section \ref{sec:2}.
Recall also that, given $\mu>0$, we represent
$\mB_\mu(\mC):=\{(\w,x)\in\WRn\,|\;\dist_\WRn\!\big((\w,x),\,\mC\big)<\mu\}$
for $\mu>0$, and denote by $\bar\mB_\mu(\mC)$ its closure. Finally, we define
$\mC_\w:=\{x\in\R^n\,|\;(\w,x)\in\mC\}$ and call it the {\em fiber of $\mC$ over $\w$}.

The most fundamental invariant object to understand the tracking results that
we will present in Section \ref{sec:4} is a local attractor, which is defined in the
context of semiflows in \cite[Section 1.2]{klra}.
It is valid for any local flow, although we formulate it for
that given by \eqref{def:3pi}.

\begin{defi}\label{def:3atractor}
A nonempty compact $\pi$-invariant set $\mA\subset\WRn$ is a
{\em local attractor\/} for the family
\eqref{eq:3hull} or for the flow \eqref{def:3pi} if there exists $\mu>0$
such that $\R_+\!\times\bar\mB_\mu(\mA)\subseteq\mO_\pi$ and,
in addition, $\lim_{t\to\infty}\dist_\WRn\!\big(\pi_t(\bar\mB_\mu(\mA)),\,\mA\big)=0$.
\end{defi}
\begin{nota}\label{rm:3existen}
The existence of a local attractor is not a priori guaranteed. However, there are well-known
situations in which local attractors exist. In fact, a {\em global attractor}
(a nonempty compact $\pi$-invariant set $\mA\subset\WRn$
satisfying the condition in Definition \ref{def:3atractor} for all $\mu>0$)
exists in the case of dissipative systems, and dissipativity can be guaranteed by
the coercivity condition $\limsup_{x\to\pm\infty}(\pm\mf(\w,x))<0$ uniformly in
$\Wf$ (see, e.g., \cite[Theorem 2.13]{duen} for a self-contained proof).
A global attractor can contain properly several local attractors.
In fact, we will be interested in local attractors projecting onto the whole base
whose fibers are ``as small as possible'': the smaller a local attractor is, the
more precise the tracking information that we will obtain in Section \ref{sec:4}.
In addition, there are many examples of
local attractors for nondissipative systems, as in the case of some concave
in $x$ maps $\mf$ giving rise to a global dynamics determined by the existence
of an attractor-repeller pair of $\pi$-invariant compact sets (see, e.g.,
\cite{dno5}). A general condition ensuring the existence of a
local attractor is given in \cite[Theorem 1.23]{klra}.
\end{nota}
Whenever a local attractor $\mA\subset\WRn$ exists,
we represent its projection onto the base of the bundle as
\begin{equation}\label{def:3proy}
 \W_\mA:=\{\w\in\Wf\,|\;\exists\,(\w,x)\in\mA\}\subseteq\Wf\,.
\end{equation}
\begin{teor}\label{teor:3posibilidades}
Assume the existence of a local attractor $\mA$ for \eqref{def:3pi}.
Then, either $\W_\mA=\Wf$ or $\W_\mA=\upomega(\hf)$.
In addition, $\W_\mA=\Wf$ if and only if $\hf\in\W_\mA$.
\end{teor}
\begin{proof}
Note that $\W_\mA$ is a compact $\sigma_\hf$-invariant subset of $\Wf$.
First, we assume the existence of $s\in\R$ with $\hf{\cdot}s\in\W_\mA$. The $\sigma_\hf$-invariance
of $\W_\mA$ yields $\{\hf{\cdot}s\,|\;s\in\R\}\subseteq\W_\mA$, and its compactness
ensures $\Wf=\W_\mA$: see definition \eqref{def:3hull}.
Note that this proves the {\em if\/} part of the last assertion, and that the {\em only if\/} is
obvious.
\par
Now, let us assume the existence of $\w\in\W_\mA\cap\upalpha(\hf)$. Our goal
is to check that $\hf\in\W_\mA$ which, as just seen, ensures $\Wf=\W_\mA$.
We take a sequence $(s_k)\downarrow-\infty$ with
$\lim_{k\to\infty}\hf{\cdot}s_k=\w$, and $x\in\R^n$ with $(\w,x)\in\mA$.
Let $\mu>0$ be the constant of Definition \ref{def:3atractor}. Since
$\lim_{k\to\infty}\di_\WRn\big((\hf{\cdot}s_k,x),\,(\w,x)\big)=0$, we find
$k_0\in\N$ with $(\hf{\cdot}s_k,x)\in\bar\mB_\mu(\mA)$ if $k\ge k_0$. For any $m\in\N$,
we take $t_m>0$ such that $\dist_\WRn\!\big(\pi_t(\bar\mB_\mu(\mA)),\,\mA\big)<1/m$
for all $t\ge t_m$, and choose $k(m)\ge k_0$ such that
$-\bar s_m:=-s_{k(m)}>t_m$. Then, $(\hf{\cdot}\bar s_m,x)\in\bar\mB_\mu(\mA)$ and hence
\[
 \dist_\WRn\!\big(\pi_{-\bar s_m}(\hf{\cdot}\bar s_m,x),\,\mA\big)\le
\dist_\WRn\!\big(\pi_{-\bar s_m}(\bar\mB_\mu(\mA)),\,\mA)<1/m\,.
\]
Thus, there exists $(\w_m,x_m)\in\mA$ such that
$\di_\WRn\big(\pi_{-\bar s_m}(\hf{\cdot}\bar s_m,x),\,(\w_m,x_m)\big)<1/m$.
We assume without loss of generality the existence of the limit
$(\bar\w,\bar x):=\lim_{m\to\infty}(\w_m,x_m)$, in which case $(\bar\w,\bar x)\in \mA$.
Then, $\bar\w=\lim_{m\to\infty}\sigma_\hf(-\bar s_m,\hf{\cdot}\bar s_m)=
\lim_{m\to\infty}\hf{\cdot}(\bar s_m-\bar s_m)=\lim_{m\to\infty}\hf=\hf$,
and hence $\hf\in\W_\mA$, as asserted.
\par
From now on, we assume that $\hf\notin\W_\mA$ and hence, according to
the previous paragraphs and Remark \ref{nota:3union}, that
$\W_\mA\subseteq\upomega(\hf)$. The proof of the first assertion will be complete
once checked that $\upomega(\hf)\subseteq \W_\mA$. We fix $(\w_1,x_1)\in\mA$
and take $(s_k)\uparrow\infty$
with $\w_1=\lim_{k\to\infty}\hf{\cdot}s_k$. As before, we find a $k_0$ large enough to get
$(\hf{\cdot}s_{k_0},x_1)\in\bar\mB_\mu(\mA)$. Now we take $\w_2\in\upomega(f)$
and look for $(t_k)\uparrow\infty$ such that
$\w_2=\lim_{k\to\infty}\hf{\cdot}(s_{k_0}+t_k)$. Since
$\lim_{k\to\infty}\dist_\WRn\!\big(\pi_{t_k}(\bar\mB_\mu(\mA)),\,\mA\big)=0$,
for any $m\in\N$ there exists a time $\bar t_m:=t_{k(m)}$ large enough and a point
$(\w_m,x_m)\in\mA$ such that
$\di_\WRn\big(\pi_{\bar t_m}(\hf{\cdot}s_{k_0},x_1),\,(\w_m,x_m)\big)<1/m$.
We assume without restriction that $(\bar t_m)\uparrow\infty$ and the existence
of $(\bar\w,\bar x):=\lim_{m\to\infty}(\w_m,x_m)\in\mA$.
Clearly, $\bar\w=\lim_{m\to\infty}\hf{\cdot}(s_{k_0}+\bar t_m)=\w_2$,
and hence $\w_2\in\W_\mA$. So,
$\upomega(\hf)\subseteq \W_\mA$, as asserted.
\end{proof}
Our following result shows that some of the properties of a local attractor reproduce the most
basic ones of a global attractor. Recall that, for each point $\w\in\W$,
\[
 \mA_\w:=\{x\in\R^n\,|\;(\w,x)\in\mA\}\,,
\]
so $\mA_\w$ is nonempty if and only if $\w\in\W_\mA$. We denote by $\mP_c(\R^n)$ the
set of nonempty compact subsets of $\R^n$ and endow it
with the topology given by the Hausdorff distance $\di_{\R^n}^H$,
which makes $\mP_c(\R^n)$ a metric space. Recall that
the set-valued map $\W_\mA\to\mP_c(\R^n)\,,\w\mapsto\mA_\w$ is
{\em upper semicontinuous\/} if, for any $\bar\w\in\W_\mA$ and any
open set $\mU$  of $\R^n$ containing $\mA_{\bar\w}$,
there exists $\delta>0$ such that $\mA_\w\subset\mU$ whenever $\dist_\Wf(\w,\bar\w)<\delta$.
The characterization appearing in the following statement (i) is an easy exercise on topology.
\begin{prop}\label{prop:3semicont}
Assume the existence of a local attractor $\mA$ for \eqref{def:3pi},
and let $\mu>0$ be the constant of Definition {\em\ref{def:3atractor}}. Then,
\begin{itemize}
\item[\rm(i)] the map $\W_\mA\to\mP_c(\R^n)\,,\;\w\mapsto\mA_\w$ is upper semicontinuous;
i.e., for any $\bar\w\in\W_\mA$,
$\lim_{\di_\Wf(\w,\bar\w)\to 0}\dist_{\R^n}\!\big(\mA_\w,\,\mA_{\bar\w}\big)=0$.
\item[\rm(ii)] The point $(\w,x)$ belongs to $\mA$ if and only if the corresponding
$\pi$-orbit is globally defined and contained in $\bar\mB_\mu(\mA)$.
\end{itemize}
\end{prop}
\begin{proof}
(i) We include a proof of (i), which is standard. For contradiction,
we assume the existence of $\bar\w\in\W_\mA$, $\sigma>0$,
and a sequence $(\w_k)$ in $\W_\mA$
with limit $\bar\w$ such that, for each $k$, there exists $x_k\in\mA_{\w_k}$ with
$\dist_{\R^n}\!\big(x_k,\,\mA_{\bar\w}\big)>\sigma$.
We assume without restriction the existence of $\bar x:=\lim_{k\to\infty}x_k$.
The contradiction arises from the fact that $(\bar\w,\bar x)\in\mA$;
i.e., $\bar x\in\mA_{\bar\w}$.
\smallskip\par
(ii) The {\em only if\/} assertion is trivial, since $\mA$ is $\pi$-invariant.
To check the {\em if\/}, we reason as in the proof of
\cite[Theorem 1.7]{calr}: if the set $\mO(\w,x):=\{\pi(t;\w,x)\,|\;t\in\R\}$ is well-defined
and contained in $\bar\mB_\mu(\mA)$, we have
\[
\begin{split}
 &\dist_\WRn\big((\w,x),\,\mA\big)
 =\dist_\WRn\!\big(\pi(t;\pi(-t;\w,x)),\,\mA\big)\\
 &\qquad\qquad\le\dist_\WRn\!\big(\pi(t;\mO(\w,x)),\,\mA\big)\le
 \dist_\WRn\!\big(\pi(t;\bar\mB_\mu(\mA)),\,\mA\big)\,.
\end{split}
\]
Since $\lim_{t\to\infty}\dist_\WRn\!\big(\pi(t;\bar\mB_\mu(\mA)),\,\mA\big)=0$,
we have $\dist_\WRn\big((\w,x),\,\mA\big)=0$, which means
that $(\w,x)$ is in the closed set $\mA$.
\end{proof}
\begin{nota}\label{rm:3upper}
The upper semicontinuity of the map
$\W_\mA\to\mP_c(\R^n),\,\w\mapsto\mA_\w$ ensures the existence of a $\sigma_\hf$-invariant residual
subset $\W_\mA^{\text{\rm cont}}\subseteq\W_\mA$ of continuity points: see, e.g., \cite[Theorem 1.4.13]{aufr}.
\end{nota}
For our main result in Section \ref{sec:4}, it is interesting to determine the maximal open set
of points that are asymptotically attracted to $\mA$; i.e., the {\em basin of attraction of
$\mA$}. Again, the following definition can be formulated for any local flow.
\begin{defi}\label{def:3attractionset}
Let $\mA$ be a local attractor for \eqref{def:3pi}.
An open set $\mU\subseteq\WRn$ is an {\em attraction open set for $\mA$}
if $\R_+\!\times\mU\subset\mO_\pi$ and, for any compact subset $\mK\subset\mU$,
$\lim_{t\to\infty}\dist_\WRn\!\big(\pi_t(\mK),\,\mA\big)=0$.
\end{defi}
Note that any open subset of the set $\mB_\mu(\mA)$ of Definition \ref{def:3atractor}
is an attraction open set for $\mA$. In addition,
\begin{prop}
Any union of attraction open sets for a local attractor $\mA$ is an
attraction open set for $\mA$.
\end{prop}
\begin{proof}
Let $\mI$ be any index set, let $\mU_i$ be an attraction open set for
$\mA$ for any $i\in\mI$, and let $\mU:=\text{{\Large$\cup$}}_{i\in\mI}\,\mU_i$.
Then, $\R_+\!\times\mU=\text{{\Large$\cup$}}_{i\in\mI}\,(\R_+\!\times\mU_i)\subset\mO_\pi$.
Given a compact subset $\mK\subset\mU$ and $(\w,x)\in\mK$, we look for $i_{\w,x}\in\mI$ and
an open set $\mV_{\w,x}\ni(\w,x)$ with compact closure $\bar \mV_{\w,x}\subseteq\mU_{i_{\w,x}}$.
Since $\mK\subseteq\text{{\Large$\cup$}}_{(\w,x)\in\mK}\mV_{\w,x}$
and $\mK$ is compact, we can find $(\w_1,x_1),\ldots,(\w_m,x_m)\in\mK$ such that
$\mK\subseteq \mV_{\w_1,x_1}\cup\ldots\cup\mV_{\w_m,x_m}$.
Then $\mK_j:=\mK\cap\bar\mV_{\w_j,x_j}$ is a compact subset of $\,\mU_{i_{\w_j,x_j}}$
for $j=1,\ldots,m$,
and $\mK=\mK_1\cup\ldots\cup\mK_m$. Finally, $\dist_\WRn\!\big(\pi_t(\mK),\,\mA\big)=
\max_{j=i,\ldots,m}\dist_\WRn\!\big(\pi_t(\mK_j),\,\mA\big)$, and
hence $\lim_{t\to\infty}\dist_\WRn\!\big(\pi_t(\mK),\,\mA\big)=0$, since
$\lim_{t\to\infty}\dist_\WRn\!\big(\pi_t(\mK_j),\,\mA\big)=0$ for
$j=i,\ldots,m$.
\end{proof}
\begin{defi}\label{def:3maximal}
Let $\mA\subset\WRn$ be a local attractor. The {\em basin of attraction of $\mA$},
$\mU_\mA$, is the union of all the attraction open sets for $\mA$.
\end{defi}
Observe that $\mA$ is a global attractor if and only if
$\mU_\mA=\WRn$.

Let us assume the existence of a local attractor $\mA$ for \eqref{def:3pi}
satisfying $\W_\mA=\Wf$ (as we will do in the main results of Section \ref{sec:4}).
Let us take $\delta\ge 0$. In consonance with the idea
introduced in \cite{walk}, we call {\em $\delta$-inflated local attractor} to the set
$\mA^{[\delta]}:=\{(\w,x)\in\WRn\,|\;x\in\mA_\w^{[\delta]}\}$, with
\[
 \mA_\w^{[\delta]}:=\bigcup_{\di_\Wf(\w,\bar\w)\le\delta}\mA_{\bar\w}\,.
\]
Note that $\mA^{[0]}=\mA$ and that $\mA_\w\subseteq\mA_\w^{[\delta]}$.
We complete this section establishing some technical
properties which will simplify the proofs of the main results in Section \ref{sec:4}.
\begin{lema}\label{lema:3inflado}
Assume the existence of a local attractor $\mA$ for the
flow \eqref{def:3pi} satisfying $\W_\mA=\Wf$.
Then,
\begin{itemize}
\item[\rm(i)] $\dist_\WRn\!\big((\w,x),\,\mA\big)\le\delta$
if and only if $\dist_{\R^n}\!\big(x,\,\mA^{[\delta]}_\w\big)\le\delta$.
\item[\rm(ii)] $\lim_{\delta\to 0^+}\dist_{\R^n}\!\big(\mA_\w^{[\delta]},\,\mA_\w\big)=0$
for all $\w\in\Wf$.
\item[\rm(iii)] If $\Wf\to\mP_c(\R^n),\,\w\mapsto\mA_\w$ is continuous
for the Hausdorff metric, then
$\lim_{\delta\to 0^+}\dist_{\R^n}\!\big(\mA_\w^{[\delta]},\,\mA_\w\big)=0$ uniformly in $\Wf$.
\end{itemize}
\end{lema}
\begin{proof}
(i) To prove the {\em only if\/} assertion, we assume that
$\dist_\WRn\!\big((\w,x),\,\mA\big)\le\delta$ and look for $(\bar\w,\bar x)\in\mA$ such that
\[
\begin{split}
 \max(\di_\Wf(\w,\bar\w),|x-\bar x|)=\di_{\WRn}\big((\w,x),\,(\bar\w,\bar x)\big)=
 \dist_\WRn\!\big((\w,x),\,\mA\big)\le\delta\,.
\end{split}
\]
Hence $\dist_{\R^n}\!\big(x,\,\mA_{\bar\w}\big)\le|x-\bar x|\le \delta$, with $\bar\w$
satisfying $\di_\Wf(\w,\bar\w)\le\delta$.
Therefore, $\dist_{\R^n}\!\big(x,\,\mA^{[\delta]}_\w\big)\le
\dist_{\R^n}\!\big(x,\,\mA_{\bar\w}\big)\le\delta$.

The proof of the {\em if\/} assertion is similar: since
$\mA^{[\delta]}_\w$ is compact, we can take
$\bar x\in \mA^{[\delta]}_\w$ such that $|x-\bar x|=
\dist_{\R^n}\!\big(x,\,\mA^{[\delta]}_\w\big)\le\delta$, look for $\bar\w$ with $\bar x\in
\mA_{\bar\w}$ and $\di_{\Wf}(\w,\bar\w)\le\delta$, and conclude that
$\dist_\WRn\!\big((\w,x),\,\mA\big)\le\max(\di_\Wf(\w,\bar\w),|x-\bar x|)\le\delta$.
\smallskip\par
(ii)\&(iii) Property (ii) is an immediate consequence of Proposition
\ref{prop:3semicont}(i), and (iii) is almost trivial.
\end{proof}
\begin{nota}\label{rm:3past}
A map $f$ initially defined on $\R_+\!\times\R_+\!\times\R^n$ can be often extended to
$\R_+\!\times\R\times\R^n$ in different ways.
It is interesting to point out that the choice of this ``past in $t$'' of $\hf$ affects
the existence and type of projection of a local attractor, as well as its properties.
So, a good choice may mean a much more accurate description of the
tracking behavior of solutions, which do not depend for positive times
on that choice of extension, as one can deduce from Theorem \ref{teor:4tracking}
and Proposition \ref{prop:4density}.
\end{nota}
\section{Tracking an inflated local attractor}\label{sec:4}
Formulating and proving the strongest tracking results of the article,
Theorem \ref{teor:4tracking} and Proposition \ref{prop:4density},
is the main objective of this section, which also contains
an analysis of the hypotheses and a description of the
properties and shape of the local attractor whose existence is assumed,
in order to provide a broader view of the scope of the results. This task
will be completed with the examples described in Section \ref{sec:5}.

Let us assume \ref{f1}, \ref{f2} and \ref{f3}.
Four $\ep$-dependent families of systems of ordinary differential equations
(from now on, just {\em equations})
will be the subject of our study, two of which (those in which $\hf$ appears) are auxiliary:
the equations
\begin{align}
\frac{dx}{d\tau}&=\ep\,f(\tau,\ep\tau,x)\,,\label{eq:4initau}\\
\frac{dz}{d\tau}&=\ep\,\hf(\ep\tau,z)\,,\label{eq:4promtau}
\end{align}
written with respect to the {\em fast time\/} $\tau$,
whose respective (unique) maximal solutions with value $x^*$ at $\tau^*\ge 0$
we represent by $x_\ep(\tau;\tau^*,x^*)$ and
$z_\ep(\tau;\tau^*,x^*)$; and those obtained from them
by the change of variables $t=\ep\tau$ for $\ep>0$,
i.e. written with respect to the {\em slow time} $t$,
\begin{align}
 \frac{dx}{dt}&=f(t/\ep,t,x)\,,\label{eq:4init}\\
 \frac{dz}{dt}&=\hf(t,z)\,.\label{eq:4promt}
\end{align}
Note that the respective maximal solutions $\tilde x_\ep(t;t^*,x^*)$ and $\tilde z(t;t^*,x^*)$
of the two last equations with value $x^*$ at $t^*\ge 0$ satisfy
\begin{equation}\label{eq:4relacion}
 \tilde x_\ep(\ep\tau;\ep\tau^*,x^*)=x_\ep(\tau;\tau^*,x^*)
  \qquad\text{and}\qquad
  \tilde z(\ep\tau;\ep\tau^*,x^*)=z_\ep(\tau;\tau^*,x^*)\,,
\end{equation}
and that (see \eqref{def:3pi})
\begin{equation}\label{eq:4solu}
 \tilde z(t;t^*,x^*)=v(t-t^*;\hf{\cdot}t^*,x^*):
\end{equation}
$(d/dt)\,v(t-t^*;\hf{\cdot}t^*,x^*)=
\mf((\hf{\cdot}t^*){\cdot}(t-t^*),v(t-t^*;\hf{\cdot}t^*,x^*))
=f(t,v(t-t^*;\hf{\cdot}t^*,x^*))$, so
that both maps solve $dz/dt=\hf(t,z)$ and take the value $x^*$ at $t=t^*$.

Note also that the maximal intervals of definition of
$x_\ep(\tau;\tau^*,x^*)$ and $\tilde x_\ep(t;t^*,x^*)$ are
open intervals of $\R_+$ containing $\tau^*\ge 0$ and $t^*\ge 0$,
respectively; and those of $z_\ep(\tau;\tau^*,x^*)$ and
$\tilde z(t;t^*,x^*)$ are open intervals of $\R$. In Section \ref{sec:6},
we will assume that the domain of $f$ is $\R_+\!\times\R_+\!\times\R^n$,
so that of $\hf$ is $\R_+\!\times\R^n$,
and then the maximal intervals of the solutions of the second and
four equations are open subsets of $\R_+$.

A first result comparing solutions of \eqref{eq:4initau} and \eqref{eq:4promtau}
sharing value $x^*$ at a given time $t^*$ is proved in
\cite[Proposition 2.6 and Theorem 2.8]{nuor1}.
Observe that, if $f$ is required to be globally bounded (as in Theorem
\ref{teor:4acotada}), so is $\hf$, and hence the solutions of the equations
\eqref{eq:4initau}, \eqref{eq:4promtau}, \eqref{eq:4init} and \eqref{eq:4promt}
are globally defined on $\R_+$.
Relation  \eqref{eq:4relacion} ensures that the two definitions appearing
in the following statement give rise to the same value.
\begin{teor}\label{teor:4acotada}
Let $f\colon\R_+\!\times\R_+\!\times\R^n\to\R^n$ be a globally bounded and continuous map
satisfying \ref{f2}, \ref{f3} and \ref{f4}. Let us fix $t_0>0$ and $k>0$, and define
\begin{equation}\label{def:4Deltat}
 D_{t_0,k}(\ep):=\sup_{(t^*,x^*)\in\R_+\!\times\bar\mB_k,\,t\in[t^*,t^*+t_0]}
 |\tilde x_\ep(t;t^*,x^*)-\tilde z(t;t^*,x^*)|
\end{equation}
for $\ep>0$; or, equivalently,
\[
 D_{t_0,k}(\ep):=\sup_{(\tau^*,x^*)\in\R_+\!\times\bar\mB_k,\,
 \tau\in[\tau^*,\tau^*+t_0/\ep]}
 |x_\ep(\tau;\tau^*,x^*)-z_\ep(\tau;\tau^*,x^*)|\,.
\]
Then, $\lim_{\ep\to 0^+}D_{t_0,k}(\ep)=0$.
\end{teor}
The uniformity of the previous limit with respect to $\tau^*\in\R_+$
is strongly based on property \ref{f4} (see \cite{nuor1}), and is a key
point in the proof of the main tracking result of this section, which
now we can formulate and prove. Recall that
$\W_\mA$ represents the projection of a local attractor $\mA$ onto the base
(see \eqref{def:3proy}), and that $\mU_\mA$ denotes its basin of attraction
(see Definition \ref{def:3maximal}). Recall also that the set $\mP_c(\R^n)$ of nonempty
compact subsets of $\R^n$ is a metric space for the topology induced by the Hausdorff
distance $\di_{\R^n}^H$.
\begin{teor}\label{teor:4tracking}
Let $f\colon\R_+\!\times\R\times\R^n\to\R^n$ satisfy \ref{f1}, \ref{f2}, \ref{f3} and \ref{f4}.
Assume the existence of a local attractor $\mA$ for the flow \eqref{def:3pi} with $\W_\mA=\Wf$,
and that $(\hf,\hat x)\in\mU_\mA$. Then,
\begin{itemize}
\item[\rm(i)] for all $\sigma>0$ and $\delta>0$, there exist
$t_0=t_0(\sigma,\delta)>0$ and
$\ep_0=\ep_0(\sigma,\delta)>0$ such that, if $0<\ep<\ep_0$, then
$\tilde x_{\ep}(t;0,\hat x)$ is defined on $[0,\infty)$ and
\begin{equation}\label{3.tesist}
 \dist_{\R^n}\!\big(\tilde x_{\ep}(t;0,\hat x),\,
 \mA_{\hf{\cdot}t}^{[\delta]}\big)\le\sigma
 \quad\text{for all $t\ge t_0$}\,.
\end{equation}
Or, equivalently, $x_{\ep}(\tau;0,\hat x)$ is defined on $[0,\infty)$ and
\[
 \dist_{\R^n}\!\big(x_{\ep}(\tau;0,\hat x),\,
 \mA_{\hf{\cdot}\ep\tau}^{[\delta]}\big)\le\sigma
 \quad\text{for all $\tau\ge t_0/\ep$}\,.
\]
\item[\rm(ii)] If, in addition, the map
$\Wf\to\mP_c(\R^n),\;\w\mapsto\mA_\w$ is continuous for the Hausdorff metric,
then {\rm(i)} also holds for $\delta=0$.
\end{itemize}
\end{teor}
\begin{proof}
(i) The equivalence mentioned in (i) is an easy consequence of
relation \eqref{eq:4relacion}. Let us fix $\mu>0$ with
$\bar\mB_\mu(\mA)\subset\mU_\mA$. Clearly, it suffices to check (i) for
$\sigma\in(0,\mu]$. So, we fix $\sigma\in(0,\mu]$ and $\delta>0$, and call
$\rho:=\min(\sigma,2\delta)$, so that $\rho\in(0,\mu]$.
It is easy to check that the inequality in \eqref{3.tesist} holds for a fixed
time $t>0$ if
\begin{equation}\label{3.tesis2t}
 \text{$\tilde x_{\ep}(t;0,\hat x)\;$ exists \;and\;\;}
 \dist_{\R^n}\!\big(\tilde x_{\ep}(t;0,\hat x),\,
 \mA_{\hf{\cdot}t}^{[\rho/2]}\big)\le\rho\,.
\end{equation}
Therefore, (i) will be proven once $t_0=t_0(\rho)$ and $\ep_0=\ep_0(\rho)$ have
been found and \eqref{3.tesis2t}
has been verified for any fixed $\ep\in(0,\ep_0)$ and for all $t\ge t_0$.
We will proceed by induction. Before it, we will choose $t_0$ and $\ep_0$, and
establish some preliminary properties that will be needed later.

Since $\mK:=\{(\hf,\hat x)\}\cup\bar\mB_\mu(\mA)\subset\mU_\mA$ is compact,
there exists $t_0(\rho)$ such that
\begin{equation}\label{eq:3limt}
 \dist_\WRn\!\big(\pi_t(\mK),\,\mA\big)\le\rho/2\quad\text{if $\;t\ge t_0(\rho)$}\,.
\end{equation}
(Clearly, by making the optimal choice for $t_0(\rho)$, we get a decreasing function
$t_0\colon(0,\mu]\to\R$.)
The value $t_0$ of (i) is $t_0:=t_0(\rho)$, which depends on $\sigma$ and $\delta$.

The choice of $\ep_0$, based on Theorem \ref{teor:4acotada}, requires some more work.
Since $\R_+\!\times\mK\subset\R_+\!\times\mU_\mA\subset\mO_\pi$
(see Definitions \ref{def:3maximal} and \ref{def:3attractionset}), we can define
\[
 \wit\mK:=\bigcup_{0\le t\le t_0}\pi_t(\mK)\,,
\]
which coincides with $\pi([0,t_0]\times\mK)$ and hence is compact. Let us check
that $\wit\mK=\text{{\Large$\cup$}}_{t\ge 0}\pi_t(\mK)$, and hence it
is positively $\pi$-invariant and independent of $\rho$: if $t\ge t_0(\rho)$,
then \eqref{eq:3limt} yields $\pi_t(\mK)\subseteq\bar\mB_{\rho/2}(\mA)\subset\mB_\mu(\mA)
\subset\mK\subseteq\wit\mK$.

We fix $k>0$ such that $\wit\mK\subset\Wf\times\mB_{k}$,
and look for a
$C^1$-map $h\colon\R^n\to[0,1]$ with value $1$ on
$\bar\mB_{k+\mu}$ and $0$ outside $\mB_{k+2\mu}$.
Let us define $g\colon\R_+\!\times\R\times\R^n$ by $g(\tau,t,x):=h(x)\,f(\tau,t,x)$,
which is globally bounded and continuous. It is immediate to check that
$g$ satisfies \ref{f3} and \ref{f4}, with $\hat g(t,x)=h(x)\,\hat f(t,x)$.
It is also easy to deduce from \ref{f1} and
the local Lipschitz character of $h$ that also \ref{f2} holds for $g$.
So, $g$ satisfies all the conditions assumed on $f$ in Theorem \ref{teor:4acotada}.
We will apply this theorem to compare the solutions of
\begin{equation}\label{eq:4porht}
 \frac{dx}{dt}=g(t/\ep,t,x)
 \qquad\text{and}\qquad  \frac{dz}{dt}=\hat g(t,z)\,,
\end{equation}
which we respectively call $\bar x_\ep(t;t^*,x^*)$ and
$\bar z(t;t^*,x^*)$, and which are globally defined on $\R_+$
since $g$ and $\hat g$ are globally bounded.
Clearly, if $x^*\in\mB_k$, then $\bar x_\ep(t;t^*,x^*)=
\tilde x_\ep(t;t^*,x^*)$ and $\bar z(t;t^*,x^*)=\tilde z(t;t^*,x^*)$ as
long as the solutions take values in $\bar\mB_{k+\mu}$, which will be the case
for the solutions we deal with due to the positive $\pi$-invariance of
$\wit\mK$ and the application of Theorem \ref{teor:4acotada} itself.

We define $D_{2t_0,k}(\ep)$ by \eqref{def:4Deltat} but for equations
\eqref{eq:4porht}. Theorem \ref{teor:4acotada} ensures the existence of $\ep_0$ such
that $D_{2t_0,k}(\ep)<\rho/2$ if $0<\ep<\ep_0$. The value $\ep_0$ of (i) is
precisely this one: note that it depends on $t_0$ and $\rho$, and hence on $\sigma$ and
$\delta$. For the rest of the proof, we fix $\ep\in(0,\ep_0)$.

Our induction argument consists in checking that, for the fixed $\ep>0$,
\[
\tag{$\mP_m$}
 \text{$\;(\hf{\cdot}t,\tilde x_\ep(t;0,\hat x))\in\mK$\quad
 and \quad\eqref{3.tesis2t} holds
 \quad if $\;t\in [mt_0,\,(m+1)\,t_0]$}
\]
for all $m\in\N$ with $m\ge 1$. It is implicit in the above expression
that $\tilde x_{\ep}(t;0,\hat x)$ is defined for $t\in[0,t_0]$, a fact that will be
established explicitly below. Therefore, proving $(\mP_m)$ for all $m\ge 1$
ensures that
$\tilde x_{\ep}(t;0,\hat x)$ is defined on $[0,\infty)$ as well as \eqref{3.tesis2t}
for $t\ge t_0$, and hence proves (i).

Let us prove $(\mP_1)$. Since $(\hf,\hat x)\in\mK$,
\eqref{eq:3limt} ensures that
\begin{equation}\label{3.distt}
 \dist_\WRn\!\big((\hf{\cdot}t,\tilde z(t;0,\hat x)),\,\mA\big)\le\rho/2
 \qquad\text{if $\;t\ge t_0$}\,,
\end{equation}
which combined with Lemma \ref{lema:3inflado}(i) ensures that
\[
 \dist_{\R^n}\!\big(\tilde z(t;0,\hat x),\,
 \mA_{\hf{\cdot}t}^{[\rho/2]}\big)\le\rho/2
 \quad\text{\;if $t\ge t_0$}\,.
\]
In addition, for all $t\ge 0$, $\pi(t;\hf,\hat x)=(\hf{\cdot}t,v(t;\hf,\hat x))=
(\hf{\cdot}t,\tilde z(t;0,\hat x))\in\wit\mK\subset\Wf\times\mB_k$
(see \eqref{eq:4solu}).
That is, $\tilde z(t;0,\hat x)$ takes values in $\mB_k$ for all $t\ge 0$ and hence
it coincides with the solution $\bar z(t;0,\hat x)$ of the second
equation in \eqref{eq:4porht}.

The definition \eqref{def:4Deltat} of $D_{2t_0,k}(\ep)$ and the choice of $\ep_0$
ensure that
\[
 |\bar x_\ep(t;0,\hat x)- \tilde z(t;0,\hat x)|=
 |\bar x_\ep(t;0,\hat x)-\bar z(t;0,\hat x)|\le
 D_{2t_0,k}(\ep)<\rho/2\,,
\]
for every $t\in[0,2t_0]$, so that $\bar x_\ep(t;0,\hat x)
\in\mB_{k+\mu}$. Hence, $\tilde x_\ep(t;0,\hat x)$ exists and satisfies
$\tilde x_\ep(t;0,\hat x)=
\bar x_\ep(t;0,\hat x)$ for all $t\in[0,2t_0]$.
This fact and the last inequalities yield
\begin{equation}\label{3.dist2t}
\begin{split}
 &\dist_{\R^n}\!\big(\tilde x_\ep(t;0,\hat x),\,\mA_{\hf{\cdot}t}^{[\rho/2]}\big)
 \le |\tilde x_\ep(t;0,\hat x)-\tilde z(t;0,\hat x)|\\
 &\qquad\qquad \quad+\dist_{\R^n}\!\big(\tilde z(t;0,\hat x),\,
 \mA_{\hf{\cdot}t}^{[\rho/2]}\big)
 <\rho
 \quad\text{for all $t\in[t_0,\,2t_0]$}\,.
\end{split}
\end{equation}
Note also that \eqref{3.dist2t} guarantees that
$\dist_\WRn\!\big((\hf{\cdot}t,\tilde x_\ep(t;0,\hat x)),\,\mA\big)\le\rho\le\mu$
if $t\in[t_0,\,2t_0]$: see Lemma \ref{lema:3inflado}(i). Hence,
$(\hf{\cdot}t,\tilde x_\ep(t;0,\hat x))\in\bar\mB_\mu(\mA)
\subseteq\mK$ if $t\in[t_0,\,2t_0]$. Property $(\mP_1)$ is proved.

Now, we assume $(\mP_m)$ and deduce $(\mP_{m+1})$.
We change the initial data (time, value) from
$(0,\hat x)$ to $(mt_0, \hat x_m)$ for $\hat x_m:=
\tilde x_\ep(mt_0;0,\hat x)$. Clearly, as long as
$\tilde x_\ep(t;0,\hat x)$ exists, $\tilde x_\ep(t;mt_0,\hat x_m)=
\tilde x_\ep(t;0,\hat x)$.
Let us consider the solution $\tilde z(t;mt_0,\hat x_m)$ of \eqref{eq:4promt}.
According to $(\mP_m)$,
\[
 (\hf{\cdot}mt_0,\tilde z(mt_0;mt_0,\hat x_m))=(\hf{\cdot}mt_0,\hat x_m)\in\mK\,,
\]
and so $(\hf{\cdot}t,\tilde z(t;mt_0,\hat x_m))=
(\hf{\cdot}t,v(t-mt_0;\hf{\cdot}mt_0,\hat x_m))
=\pi(t-mt_0;\hf{\cdot}mt_0,\hat x_m) \in\pi_{t-mt_0}(\mK)$ for all $t\ge mt_0$
(see again \eqref{eq:4solu}).
Hence, we apply property \eqref{eq:3limt} to conclude that
$\dist_\WRn\!\big((\hf{\cdot}t,\tilde z(t;mt_0,\hat x_m)),\,\mA\big)\le\rho/2$
if $t-mt_0\ge t_0$; i.e., if $t\ge(m+1)\,t_0$, and Lemma \ref{lema:3inflado}(i)
to conclude that
$\dist_{\R^n}\!\big(\tilde z(t;mt_0,\hat x_m),\,\mA_{\hf{\cdot}t}^{[\rho/2]}\big)
\le\rho/2$. The positive $\pi$-invariance of $\wit\mK$ shows
that $\tilde z(t;mt_0,\hat x_m)\in\mB_k$ for
$t\ge mt_0$. Thus, $\tilde z(t;mt_0,\hat x_m)=\bar z(t;mt_0,\hat x_m)$
for $t\ge mt_0$, and hence the definition \eqref{def:4Deltat}
of $D_{2t_0,k}(\ep)$ guarantees that
\[
 |\bar x_\ep(t;mt_0,\hat x_m)-\tilde z(t;mt_0,\hat x_m)|=
 |\bar x_\ep(t;mt_0,\hat x_m)-\bar z(t;mt_0,\hat x_m)|
 <\rho/2
\]
for every $t\in[mt_0,\,(m+2)\,t_0]$. So,
$\tilde x_\ep(t;0,\hat x)=\tilde x_\ep(t;mt_0,\hat x_m)$ exists and satisfies
$\tilde x_\ep(t;0,\hat x)=\bar x_\ep(t;mt_0,\hat x_m)$
for all $t\in[mt_0,(m+2)\,t_0]$. Altogether, we get
\[
\begin{split}
 \dist_{\R^n}\!\big(\tilde x_\ep(t;0,\hat x),\,\mA_{\hf{\cdot}t}^{[\rho/2]}\big)
 &\le |\tilde x_\ep(t;0,\hat x)-\tilde z(t;mt_0,\hat x_m)|\\
 &\quad+\dist_{\R^n}\!\big(\tilde z(t;mt_0,\hat x_m),\,
 \mA_{\hf{\cdot}t}^{[\rho/2]}\big)
 <\rho
\end{split}
\]
if $t\in[(m+1)\,t_0,\,(m+2)\,t_0]$,
which in turn implies that $(\hf{\cdot}t,\tilde x_\ep(t;0,\hat x))\in\mK$
for these values of $t$. This completes the proofs of $(\mP_{m+1})$ and (i).
\smallskip\par
(ii) Let us assume that the map $\w\mapsto\mA_\w$ is continuous.
We fix $\sigma>0$ and apply Lemma \ref{lema:3inflado}(iii) to find
$\rho=\rho(\sigma)>0$ such that, if $\di_\Wf(\w_1,\w_2)\le\rho$,
then $\dist_{\R^n}\!\big(\mA_\w^{[\rho]},\,\mA_\w\big)\le\sigma/2$ for all $\w\in\Wf$.
Applying (i) to $\sigma/2$ and $\rho$, we conclude that there exist $t_0=t_0(\sigma)>0$
and $\ep_0=\ep_0(\sigma)>0$ such that, if $0<\ep<\ep_0$, then
\[
\begin{split}
 &\dist_{\R^n}\!\big(\tilde x_\ep(t;0,\hat x),\,\mA_{\hf{\cdot}t}\big)\le
 \dist_{\R^n}\!\big(\tilde x_\ep(t;0,\hat x),\,\mA_{\hf{\cdot}t}^{[\rho]}\big)\\
 &\qquad\qquad\qquad\qquad+
 \dist_{\R^n}\!\big(\mA_{\hf{\cdot}t}^{[\rho]},\,
 \mA_{\hf{\cdot}t}\big)\le\sigma/2+\sigma/2=\sigma
 \qquad\text{for all $t\ge t_0$}\,.
\end{split}
\]
This completes the proofs of (ii) and the theorem.
\end{proof}
The following result refines the previous one. It shows that, under additional
hypotheses on $\Wf$ and the local attractor $\mA$ (which are fairly standard,
as we will explain below), even if the fiber map $\w\mapsto\mA_\w$ is not
continuous, the set of times $t$ for which
$\tilde x_\ep(t;0,\hat x)$ is as close as desired to $\mA_{\hf{\cdot}t}$
for $\ep>0$ small enough, is ``large'', in the sense that it is relatively
dense and has lower density close to $1$. Recall that the set $\W_\mA^{\text{\rm cont}}\subseteq\Wf$
of continuity points of the map $\Wf\to\mP_c(\R^n),\;\w\mapsto\mA_\w$
is $\sigma_\hf$-invariant and residual (see Remark {\rm\ref{rm:3upper}}).
A set $\mS\subseteq\Wf$ has {\em complete measure\/}
if $m(\mS)=1$ for every $\sigma_\hf$-invariant measure on $\Wf$.
Recall also that $\mC\subset[0,\infty)$ is
{\em relatively dense} if there exists an {\em inclusion length}
$d_\mC>0$ such that $\mC\cap[b,b+d_\mC]$ is
nonempty for all $b\ge 0$; and that its {\em lower density} is
$d_l(\mC):=\liminf_{t\to\infty} (1/t)\int_0^t \chi|_{_{\mC}}(s)\,ds$, where $\chi|_{_{\mC}}$
is the characteristic function of~$\mC$. As usual,
$\mB_{\Wf}(\bar\w,\delta):=\{\w\in\Wf\,|\;\di_\Wf(\w,\bar\w)<\delta\}$,
with closure $\bar\mB_{\Wf}(\bar\w,\delta)$.
\begin{prop}\label{prop:4density}
Let $f\colon\R_+\!\times\R\times\R^n\to\R^n$
satisfy \ref{f1}, \ref{f2}, \ref{f3} and \ref{f4}.
Assume the existence of a local attractor $\mA$ for the
flow \eqref{def:3pi} with $\W_\mA=\Wf$, and that
$(\hf,\hat x)\in\mU_\mA$. Assume also that $\Wf$ is minimal.
\begin{itemize}
\item[\rm(i)] Given $\bar\w\in\W_\mA^{\text{\rm cont}}$ and $\sigma>0$, there exists
$\bar\ep=\bar\ep(\bar\w,\sigma)>0$ such that, if $0<\ep<\bar\ep$, then
$\dist_{\R^n}\!\big(\tilde x_{\ep}(t;0,\hat x),\,\mA_{\bar\w}\big)\le2\sigma/3$ and
$\dist_{\R^n}\!\big(\tilde x_{\ep}(t;0,\hat x),\,\mA_{\hf{\cdot}t}\big)\le\sigma$
for all $t$ in a set $\mD=\mD(\bar\w,\sigma)\subseteq[0,\infty)$ which is relatively dense
and has positive lower density.
\item[\rm(ii)] Assume that, in addition, $\W_\mA^{\text{\rm cont}}$ has complete measure.
Given $\sigma>0$ and $\rho\in(0,1)$, there exist $\ep^*=\ep^*(\sigma,\rho)>0$
such that, if $0<\ep<\ep^*$, then
$\dist_{\R^n}\!\big(\tilde x_{\ep}(t;0,\hat x),\,\mA_{\hf{\cdot}t}\big)\le\sigma$
for all $t$ in a set $\mD^*=\mD^*(\sigma,\rho)\subseteq[0,\infty)$ which is
relatively dense and has lower density $d_l(\mD^*)\ge 1-\rho$.
\end{itemize}
\end{prop}
\begin{proof}
We look for $\delta=\delta(\bar\w,\sigma)>0$ such that
$\di^H_{\R^n}\big(\mA_\w,\,\mA_{\bar\w}\big)\le\sigma/3$ if $\di_\Wf(\w,\bar\w)\le\delta$. It follows
easily that $\di^H_{\R^n}\big(\mA_\w^{[\delta/2]},\,\mA_{\bar\w}\big)\le\sigma/3$ if
$\di_\Wf(\bar\w,\w)\le\delta/2$. Let $\bar t=\bar t(\bar\w,\sigma):=t_0(\sigma/3,\delta/2)$ and
$\bar\ep=\bar\ep(\bar\w,\sigma):=\ep_0(\sigma/3,\delta/2)$ be provided by Theorem \ref{teor:4tracking}(i),
so that $\dist_{\R^n}\!\big(\tilde x_{\ep}(t;0,\hat x),\,\mA_{\hf{\cdot}t}^{[\delta/2]}\big)\le\sigma/3$
if $0<\ep<\bar\ep$ and $t\ge \bar t$. We define
\begin{equation}\label{def:4D}
 \bar\mD=\bar\mD(\bar\w,\sigma):=\{t\in[\bar t(\bar\w,\sigma),\infty)\,|\;
 \hf{\cdot}t\in\bar\mB_\Wf(\bar\w,\delta(\bar\w,\sigma)/2)\}\,,
\end{equation}
so that if $t\in\bar\mD$ and $0<\ep<\bar\ep$, then
\[
 \dist_{\R^n}\!\big(\tilde x_{\ep}(t;0,\hat x),\,\mA_{\bar\w}\big)
 \le \dist_{\R^n}\!\big(\tilde x_{\ep}(t;0,\hat x),\,\mA_{\hf{\cdot}t}^{[\delta/2]}\big)+
 \di^H_{\R^n}\!\big(\mA_{\hf{\cdot}t}^{[\delta/2]},\,\mA_{\bar\w}\big)\le2\sigma/3
\]
and, consequently,
\[
 \dist_{\R^n}\!\big(\tilde x_{\ep}(t;0,\hat x),\,\mA_{\hf{\cdot}t}\big)\le
 \dist_{\R^n}\!\big(\tilde x_{\ep}(t;0,\hat x),\,\mA_{\bar\w}\big)+
 \di^H_{\R^n}\!\big(\mA_{\bar\w},\,\mA_{\hf{\cdot}t}\big)\le \sigma  \,.
\]
This means that $\bar\mD$ is contained in the set $\mD$ of statement (i). Clearly, it contains
the nonempty set $\tilde\mD$ of points on $[\bar t,\infty)$ such that
$\di_\Wf(\hf{\cdot}t,\bar\w)<\delta/4$.

Let us check $\tilde \mD$ is relatively dense in $[0,\infty)$ (and hence so is $\mD$).
The minimality of the flow on $\Wf$ and the open character of $\mB:=\mB_\Wf(\bar\w,\delta/4)$
ensure the existence of positive values of time $s_1<\cdots<s_p$ with $s_1\ge\bar t$
such that $\Wf\subseteq\sigma_\hf(-s_1,\mB)\cup\ldots\cup\sigma_\hf(-s_p,\mB)$; hence,
given $b\ge 0$, $\hf{\cdot}b\in\sigma_\hf(-s_i,\mB)$ for an $i\in[1,\ldots,p]$;
and this means that $\hf{\cdot}(b+s_i)\in\mB$, so $\tilde\mD\cap[b,b+s_p]$ is
nonempty. This means that $d_{\tilde\mD}:=s_p$ is an inclusion length of $\tilde\mD$.

Now, let us check that $\liminf_{t\to\infty} l(\mD\cap[0,t])/t>0$ where $l$ is the Lebesgue measure on $\R$
(i.e., that $\mD$ has positive lower density).
We look for $\rho\in(0,d_{\tilde\mD}]$ such that $\sigma_\hf\big([0,\rho]\times\bar\mB_\Wf(\bar\w,\delta/4)\big)
\subseteq\bar\mB_\Wf(\bar\w,\delta/2)$, and deduce that $[t,t+\rho]\subset\bar\mD\subseteq\mD$ for all
$t\in\tilde\mD$. For each integer
$k\ge 0$, we take a point $t_k\in [\bar t+2\,k\,d_{\tilde\mD},\,\bar t+(2\,k+1)d_{\tilde\mD}]
\cap\tilde\mD\subset\mD$. As seen above, $[t_k,t_k+\rho]\subset\mD$, and hence
$[t_k,t_k+\rho]\subset\mD\cap [\bar t+2\,k\,d_{\tilde\mD},\,\bar t+(2\,k+2)d_{\tilde\mD}]$.
For each $t>\bar t$, we denote by $k_t\ge 0$ the (lower) integer part of
$(t-\bar t)/(2d_{\tilde\mD})$, and
observe that $\lim_{t\to\infty}k_t=\infty$ and that $t-\bar t<(2\,k_t+2)\,d_{\tilde\mD}$.
Since $\mD\cap[\bar t,t]\supseteq\bigcup_{k=0}^{k_t-1}
\big(\mD\cap [\bar t+2\,k\,d_{\tilde\mD},\,\bar t+(2\,k+2)d_{\tilde\mD}]\big)$, we conclude that
$l(\mD\cap[0,t])\ge k_t\,\rho$, and hence $l(\mD\cap[0,t])/t\ge
k_t\,\rho/(\bar t+(2\,k_t+2)\,d_{\tilde\mD})$. This implies that
$\liminf_{t\to\infty} l(\mD\cap[0,t])/t\ge \rho/(2\,d_{\tilde\mD})>0$, as asserted.
%
\smallskip

(ii) We have proved in (i) that, if $\bar\w\in\W_\mA^{\text{\rm cont}}$ and
$0<\ep<\bar\ep(\bar\w,\sigma)$, then
$\dist_{\R^n}\!\big(\tilde x_{\ep}(t;0,\hat x),\,\mA_{\hf{\cdot}t}\big)\le\sigma$
for all $t$ in the set $\bar\mD(\bar\w,\sigma)$ defined by \eqref{def:4D}.
We denote $\delta_{\bar\w}:=\delta(\bar\w,\sigma)/2$
for each $\bar\w\in\W_\mA^{\text{\rm cont}}$.
By hypothesis, the open set
$\mU:=\bigcup_{\bar\w\in\W_\mA^{\text{\rm cont}}}\mB_\Wf(\bar\w,\delta_{\bar\w})\supseteq\W_\mA^{\text{\rm cont}}$
has complete measure. We will prove below that, for the fixed $\rho\in(0,1)$,
there exists a compact set $\mK\subseteq\mU$ with $m(\mK)\ge 1-\rho$ for all $\sigma_\hf$-invariant
measure $m$ on $\Wf$. The compactness of $\mK$
ensures that $\mK\subseteq \mO:=\mB_\Wf(\bar\w_1,\delta_{\bar\w_1})\cup\cdots
\cup\mB_\Wf(\bar\w_p,\delta_{\bar\w_p})$ for some points $\bar\w_1,\ldots,\bar\w_p\in\W_\mA^{\text{\rm cont}}$,
as well as the existence of a continuous map $\ml\colon\Wf\to[0,1]$ such that $\ml|_{\mK}\equiv 1$ and
$\ml|_{\Wf\setminus\mO}\equiv 0$. So, $\int_\Wf\ml(\w)\,dm\ge m(\mK)\ge1-\rho$ for all
$\sigma_\hf$-invariant measure $m$.

We define the time $t^*=t^*(\sigma,\rho):=\max(\bar t(\bar\w_1,\sigma),\ldots,\bar t(\bar\w_p,\sigma))>0$,
the set $\mD_*=\mD_*(\sigma,\rho):=\{t\in[t^*,\infty)\,|\;
\hf{\cdot}t\in\mO\}
\subseteq[t^*,\infty)$, and the value $\ep^*=\ep^*(\sigma,\rho):=
\min(\bar\ep(\bar\w_1,\sigma),\ldots,\bar\ep(\bar\w_p,\sigma))>0$. Then,
$\dist_{\R^n}\!\big(\tilde x_{\ep}(t;0,\hat x),\,\mA_{\hf{\cdot}t}\big)\le\sigma$ whenever
$t\in\mD_*$ and $0<\ep<\ep^*$, so that $\mD_*$ is contained in the set $\mD^*$ of (ii).
Let us check that $\mD_*$ (and hence $\mD^*$) satisfies the
remaining properties of (ii). To prove that $\mD_*$ is relatively dense in $[0,\infty)$,
we can reason as in (i) for the set $\tilde\mD$.
Now, let us check that its lower density $d_l(\mD_*)$ is larger than $1-\rho$. Since $s\in\mD_*$
if and only if $s\ge t^*$ and $\hf{\cdot}s\in\mO$, the definition of lower density yields
\[
 d_l(\mD_*)=\liminf_{t\to\infty}\frac{1}{t}\int_{t^*}^t\chi|_{_{\mO}}(\hf{\cdot}s)\,ds\ge
 \liminf_{t\to\infty}\frac{1}{t}\int_0^t\ml(\hf{\cdot}s)\,ds=
 \lim_{k\to\infty}\frac{1}{T_k}\int_0^{T_k}\ml(\hf{\cdot}s)\,ds
\]
for a suitable sequence $(T_k)\uparrow\infty$. Riesz's Representation Theorem
associates to the bounded linear functional defined by $C(\Wf,\R)\to\R,\;
\mg\mapsto (1/T_k)\int_0^{T_k}\mg(\hf{\cdot}s)\,ds$, whose norm is 1,
a (complete regular Borel) normalized measure $\tilde m_k$, which satisfies
$\int_\Wf \mg(\w)\,d\tilde m_k=(1/T_k)\int_0^{T_k}\mg(\hf{\cdot}s)\,ds$
for $\mg\in C(\Wf,\R)$. The set $\mathfrak{M}(\Wf)$
of (complete regular Borel) normalized measures on $\Wf$
endowed with the weak$^*$ topology (which means that $(m_k)$
converges to $m$ if and only if $\lim_{k\to\infty}\int_\Wf\mg(\w)\,dm_k=
\int_\Wf\mg(\w)\,dm$ for every $\mg\in C(\Wf,\R)$) is a metrizable compact space
(see e.g.~Theorems 6.4 and 6.5 of \cite{walt}). So,
the sequence $(\tilde m_k)$ has a subsequence $(\tilde m_{k_j})$
which weakly$^*$ converges to a measure $\tilde m$ on $\Wf$: $\tilde m$ satisfies
$\int_{\Wf}\mg(\w)\,d\tilde m=\lim_{j\to\infty}(1/T_{k_j})\int_0^{T_{k_j}}\mg(\hf{\cdot}s)\,ds$
for all $\mg\in C(\Wf,\R)$. It is easy to check that $\tilde m$ is $\sigma_\hf$-invariant, and hence
\[
 d_l(\mD_*)\ge \lim_{j\to\infty}\frac{1}{T_{k_j}}\int_0^{T_{k_j}}\ml(\hf{\cdot}s)\,ds
 =\int_\Wf\ml(\w)\,d\tilde m\ge 1-\rho\,,
\]
as asserted.

It remains to check the existence of the set $\mK\subseteq\mU$ with $m(\mK)\ge 1-\rho$ for all
$m$ in the subset $\mathfrak{M}_{\text{inv}}(\Wf,\sigma_\hf)\subset\mathfrak{M}(\Wf)$ of
$\sigma_\hf$-invariant measures.
For an increasing sequence $(\mg_k)$ in $C(\Wf,[0,1])$ with $\lim_{k\to\infty}\mg_k(\w)=
\chi|_{_{\mU}}(\w)$ for all $\w\in\Wf$ (as $\mg_k(\w):=\min(1,k\dist_\Wf(\w,\Wf\setminus\mU))$, for example),
we define $\Phi_k\colon\mathfrak{M}_{\text{inv}}(\Wf,\sigma_\hf)\to\R,\,m\mapsto\int_\Wf\mg_k(\w)\,dm$.
It is easy to check that $(\Phi_k)$ is an increasing sequence of real weakly$^*$ continuous maps,
and that its limit is the constant map $1$. Dini's theorem ensures that the convergence
is uniform, and hence there exists $\bar k$ such that $\int_\Wf\mg_{\bar k}(\w)\,dm\ge 1-\rho^2$ for all
$m\in\mathfrak{M}_{\text{inv}}(\Wf,\sigma_\hf)$. Now we define
$\mK:=\{\w\in\mU\,|\;\mg_{\bar k}(\w)\ge1-\rho\}$, and observe that
$\int_\Wf\mg_{\bar k}(\w)\,dm=\int_\mK\mg_{\bar k}(\w)\,dm+\int_{\mU-\mK}
\mg_{\bar k}(\w)\,dm\le m(\mK)+(1-\rho)\,(1-m(\mK))$. So, $1-\rho^2\le 1-\rho+\rho\,m(\mK)$,
from where the assertion follows.
\end{proof}
Regarding the extra conditions in Proposition \ref{prop:4density}:
as explained at the end of Section \ref{sec:2}, if the initial map $f$
(and hence its average $\hf$) is almost periodic in the
slow variable $t$, then the flow on $\Wf$ is minimal and uniquely ergodic. Hence,
the set $\W_\mA^{\text{\rm cont}}$ has complete measure if and only if
it is not negligible with respect to the unique $\sigma_\hf$-invariant measure.
This last situation is expected to be highly frequent from a topological point of
view when the initial map is not periodic. Remark \ref{rm:5.meager} explains this assertion
for a special type of maps $f$.

We will present in Section \ref{sec:5} two nontrivial examples devoted
to delve into the applicability and
optimality of Theorems \ref{teor:4tracking}(i) and \ref{teor:4acotada},
one of which also illustrates the information
provided by Proposition \ref{prop:4density}.
Example \ref{ex:5fibracont} illustrates the ``good'' tracking
behavior in the case of continuity of the fibers of the attractor, and
shows that, even in this case, it is not true in general that
$\lim_{\ep\to 0^+}(\tilde x_\ep(t;t^*,x^*)-\tilde z(t;t^*,x^*))=0$
uniformly in any positive half-line.
That is, the classical averaging principle does not stand on positive half-lines: the
solution of the original equation does not track the solution of the averaged equation but
the fibers of the (perhaps inflated) attractor.
And Example \ref{ex:5fibranocont2} shows that,
even in the case of minimality of $\Wf$,
we cannot in general take $\delta=0$ in Theorem \ref{teor:4tracking}(i) unless the fiber map
is continuous; and at the same time, as Proposition \ref{prop:4density} establishes,
the set of times at which $\tilde x_\ep(t;t^*,x^*)$ is close to $\mA_{\hf{\cdot}t}$
is large and with positive density.

Now, we give a quite simpler example of application of Theorem \ref{teor:4tracking}(ii),
which shows its scope for scalar ODEs when there exists a so-called {\em attractive
hyperbolic copy of the base}; i.e., the uniformly exponentially stable (at $+\infty$) graph of
an invariant continuous map $\ma\colon\Wf\to\R$. When the stability property holds at $-\infty$,
we have a {\em repulsive hyperbolic copy of the base}.
\begin{exa}\label{ex:4parAR}
Let us define $f(\tau,t,x):=p_1(\tau)+p_2(t)-x^2$ for
$p_1(\tau):=\sin^2(\tau)$ and $p_2(t):=\sin(2t)+\sin(\pi t)$. The function
$p_1$ is $\pi$-periodic with average $\hat{p}_1=1/2$. So, the average function
$\hf$ defined by \eqref{def:3promedio} is $\hf(t,x):=1/2+p_2(t)-x^2$.
Conditions \ref{f1}, \ref{f2}, \ref{f3} and \ref{f4}
are satisfied: to check the first two ones is very easy (since $p_2$ is
uniformly continuous), and \ref{f3} and \ref{f4}
follow easily from the periodicity of $p_1$ or can be deduced, for example,
from Proposition \ref{prop:5ejemplos}(i)
below. The main question, posed with respect to the slow time $t$,
is to describe the long-time behavior of the solution
$\tilde x_\ep(t;0,\hat x)$ of $dx/dt=f(t/\ep,t,x)$, determining a local
attractor of $dz/dt=\hf(t,z)$ whose fibers are tracked by this solution. In this case,
the local attractor will be simply defined from a particular solution of the averaged equation.

We define $\W_{p_2}:=\text{closure}\{p_2{\cdot}s\,|\;s\in\R\}$
in the compact open-topology of $C(\R,\R)$, where $p_2{\cdot}s(t):=p_2(s+t)$,
and note that the elements of the hull of $\hf_\lambda(t,x):=p_2(t)-x^2+\lb$
are the maps $(t,x)\mapsto \mpp_2(\wt)-x^2+\lb$ for $\w\in\W_{p_2}$, where
$\mpp_2(\w):=\w(0)$. So, $\W_{\hf_\lb}=\{\w-x^2+\lb\,|\;\w\in\W_{p_2}\}$, and
the $\lambda$-parametric family of equations
\begin{equation}\label{eq:4exjemplohull}
 \frac{dz}{dt}=\mpp_2(\wt)-z^2+\lb\,,\quad\w\in\W_{p_2}
\end{equation}
coincides with \eqref{eq:3hull} when $\lb=1/2$.
Thanks to the strict concavity on $x$ of the map $(\w,x)\mapsto\mpp_2(\w)-x^2$,
the global dynamics and its variation with respect to $\lb$ are well-known:
either proofs or suitable references of the properties that we summarize now are given,
e.g., in \cite[Theorems 3.3, 3.4 and 3.5]{dno5}.
The most relevant dynamical characteristic is the existence of $\bar\lb\in\R$ such that
each equation has no bounded solutions if and only if
$\lb<\bar\lb$, has no uniformly separated
bounded solutions for $\lb=\bar\lb$, and has two hyperbolic solutions $a^\w_\lb(t)>r^\w_\lb(t)$
for $\lb>\bar\lb$, with $a_\lb^\w$ attractive and $r_\lb^\w$ repulsive. More precisely,
for $\lb>\bar\lb$, there exist an attractive hyperbolic copy of the base and a repulsive one,
given by the graphs of the continuous maps $\ma_\lb,\,\mr_\lb\colon\W_{p_2}\to\R$ defined by
$\ma_\lb(\w):=a^\w_\lb(0)$ and $\mr_\lb(\w):=r^\w_\lb(0)$, which satisfy
$\ma_\lb(\wt)=a^\w_\lb(t)$ and $\mr_\lb(\wt)=r^\w_\lb(t)$ for all $\w\in\W_{p_2}$.
In addition, the maps $(\bar\lb,\infty)\to C(\W_{p_2},\R),\,\lb\mapsto \ma_\lb,\,\mr_\lb$
are continuous in the uniform topology of $C(\W_{p_2},\R)$; and,
if $\lb>\bar\lb$, then $\lim_{t\to\infty}(\ma_\lb(\wt)-v_\lb(t;\w,x^*))=0$
for all $x^*>\mr_\lb(\w)$, where $t\mapsto v_\lb(t;\w,x^*)$ solves \eqref{eq:4exjemplohull}
with value $x^*$ at $t=0$.

From the definition of attractive hyperbolic copy of the base,
it is easy to deduce that the graph of $\ma_\lb$ is a local attractor for
the skewproduct flow induced by \eqref{eq:4exjemplohull} which projects onto $\W_{p_2}$,
and hence $\mA_\lb:=\{(\w-x^2+\lb,\ma_\lb(\w))\,|\;\w\in\W_{p_2}\}$
is a local attractor for \eqref{eq:3hull}
which projects onto $\W_{\hf_\lb}$. Clearly,
the fiber map $\W_{\hf_\lb}\to\mP_c(\R),\;\w-x^2+\lb\mapsto\{\ma_\lb(\w)\}$ is continuous.
So, all the hypotheses of Theorem \ref{teor:4tracking}(ii) are fulfilled whenever
$(\hf_\lb,\hat x)$ belongs to the basin of attraction of $\mA_\lb$, i.e.,
whenever $\hat x>\mr_\lb(p_2)$.
\begin{figure}[ht]
 \includegraphics[width=0.95\textwidth]{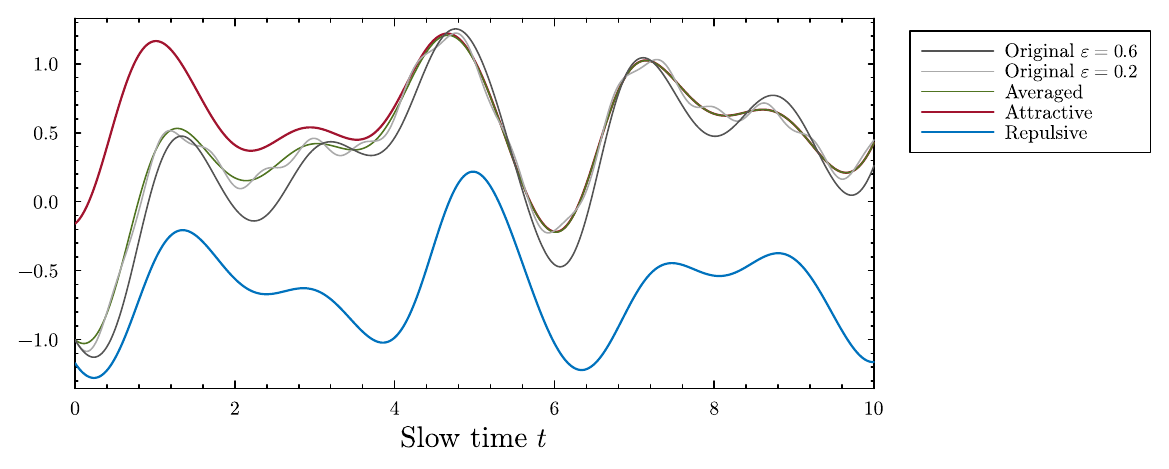}
 \caption{This figure corresponds to Examples \ref{ex:4parAR} and \ref{ex:6parAR}.
 It depicts the hyperbolic solutions of the equation $dz/dt=
 1/2+\sin(2t)+\sin(\pi t)-z^2$:
 $a_{1/2}$, attractive, in red; and $r_{1/2}$, repulsive, in blue.
 The fiber of the local attractor $\mA_{1/2}$ over $\hf{\cdot}t$ is
 $\{a_{1/2}(t)\}$. The figure also shows $\tilde z(t;0,-1)$, in green;
 and $\tilde x_\ep(t;0,-1)$ for $\ep=0.6$ and $\ep=0.2$ in darker and lighter
 shades of gray.}
 \label{fig:4exparAR}
\end{figure}

Simple numerical computations indicate that $\bar\lb\sim 0.1844<1/2$.
So, we can assume that $\bar\lb<1/2$ and hence that $\mA_{1/2}$ exists.
We take $\w=p_2$: the two hyperbolic solutions of the averaged equation $dz/dt=1/2+p_2(t)-z^2$
can be numerically obtained and are depicted in Figure \ref{fig:4exparAR}.
As just seen: the upper and lower ones are $a_{1/2}(t):=\ma_{1/2}(p_2{\cdot}t)$ and
$r_{1/2}(t):=\mr_{1/2}(p_2{\cdot}t)$; $(\mA_{1/2})_{\hf{\cdot}t}
=\{a_{1/2}(t)\}$ for all $t\in\R$; and $\lim_{t\to\infty}(a_{1/2}(t)-\tilde z(t;0,\hat x))=
\lim_{t\to\infty}(\ma_{1/2}(p_2{\cdot}t)-v_{1/2}(t;p_2,\hat x))=0$ if $\hat x>r_{1/2}(0)$.
Theorem \ref{teor:4tracking}(ii) ensures that, if $\ep>0$ is small enough, then the solution
$\tilde x_\ep(t;0,\hat x)$ stays at a distance less than $\sigma>0$ from $a_{1/2}(t)$ on a
certain half-line $[t_\sigma,\infty)$, being $t_\sigma$ larger if $\sigma$ is lower.
\end{exa}

The approaching of $\tilde z(t;0,\hat x)$ and $\tilde x_\ep(t;0,\hat x)$
to $a_{1/2}(t)$ is reflected in Figure \ref{fig:4exparAR} for an $\hat x>r_{1/2}(0)$
and two values of $\ep$.
It is important to observe that, whereas $\tilde z(t;0,\hat x)$ and $a_{1/2}(t)$ are
numerically indistinguishable after a certain time, this is not the case
for $\tilde x_\ep(t;0,\hat x)$ and $a_{1/2}(t)$, whose distance does not tend
to zero: the graph of $\tilde x_\ep(t;0,\hat x)$ oscillates around that of $a_{1/2}(t)$.

\begin{nota}\label{rm:4F}
So far, we have started with a single map $f$ and use the hull construction
as a fundamental tool to get the tracking results of Theorem \ref{teor:4tracking}:
without the skewproduct flow $\pi$, we cannot talk of a local attractor $\mA$.

But the initial situation could be already that of a family of equations varying over
a compact metric space. More precisely,
let $(\W_1,\sigma_1)$ and $(\W_2,\sigma_2)$ be two continuous flows on
two compact metric spaces.
For simplicity, we represent $\w_1{\cdot}\tau:=\sigma_1(\tau,\w_1)$ and
$\w_2{\cdot}t:=\sigma_2(t,\w_2)$.
Let $F\colon\W_1\times\W_2\times\R^n\to\R^n$ be a continuous map,
and let us consider the family of equations
\begin{equation}\label{eq:4Ffam}
 \frac{dx}{d\tau}=\ep\,F(\w_1{\cdot}\tau,\w_2{\cdot}\ep\tau,x)\,,
 \qquad (\w_1,\w_2)\in\W_1\times\W_2\,.
\end{equation}
So, the equation corresponding to a fixed element $(\bar\w_1,\bar\w_2)\in\W_1\times\W_2$
takes the form
\begin{equation}\label{eq:4notaF}
 \frac{dx}{d\tau}=\ep\,f(\tau,\ep\tau,x)
\end{equation}
for $f\colon\R_+\!\times\R\times\R^n\to\R^n,\,(\tau,t,x)\mapsto
F(\bar\w_1{\cdot}\tau,\bar\w_2{\cdot}t,x)$.
It is obvious that $f$ satisfies \ref{f1}. How can we guarantee the
remaining conditions initially required for $f$ in Theorem \ref{teor:4tracking}?

Let us see that $f$ satisfies \ref{f2} if $F$ is locally Lipschitz with respect to $x$
uniformly with respect to $\w_1,\w_2$;
i.e., if for any $\mu>0$ there exists $L_\mu$ such that
$|F(\w_1,\w_2,x_1)-F(\w_1,\w_2,x_2)|\le L_\mu\,|x_1-x_2|$ for all $\w_1\in\W_1$, $\w_2\in\W_2$
and $x_1,x_2\in\bar\mB_\mu$. We fix $\mu>0$ and define
$\theta_\mu(h):=\sup_{s\in[0,h],\,\w_1\in\W_1,\,\w_2\in\W_2,\,x\in\bar\mB_\mu}|
F(\w_1,\w_2{\cdot}s,x)-F(\w_1,\w_2,x)|$.
Then, $\theta_\mu\colon\R_+\to\R_+$ is a modulus of continuity and
\[
\begin{split}
 &|f(\tau,t_1,x_1)-f(\tau,t_2,x_2)|\le |F(\bar\w_1{\cdot}\tau,\bar\w_2{\cdot}t_1,x_1)-
 F(\bar\w_1{\cdot}\tau,\bar\w_2{\cdot}t_1{\cdot}(t_2-t_1),x_1)|\\
 &\quad\qquad +
 |F(\bar\w_1{\cdot}\tau,\bar\w_2{\cdot}t_2,x_1)
 -F(\bar\w_1{\cdot}\tau,\bar\w_2{\cdot}t_2,x_2)|\le\theta_\mu(|t_1-t_2|)+L_\mu|x_1-x_2|
\end{split}
\]
for all $\tau\in\R$, $t_1,t_2\in\R$, and $x_1,x_2\in\bar\mB_\mu$.

Some more restrictive conditions guarantee \ref{f3} and \ref{f4}.
Let us assume that $(\W_1,\sigma_1)$ is uniquely ergodic (see Section \ref{sec:2}),
and let $m_1$ be the unique $\sigma_1$-invariant measure on $\W_1$.
The unique ergodicity ensures the existence of
\begin{equation}\label{eq:4unif}
\lim_{T\to\infty}\frac{1}{T}\int_0^T F(\w_1{\cdot}s,\w_2,x)\,ds=:\hat F(\w_2,x)=
\int_{\W_1}F(\w_1,\w_2,x)\,dm_1(\w_1)
\end{equation}
for all $\w_1\in\W_1$, $\w_2\in\W_2$ and $x\in\R^n$, as well as the uniformity of
the limit with respect to $\w_1\in\W_1$: see, e.g., \cite[Proposition 4.1.13]{kaha}.
Let us deduce that
the limit \eqref{eq:4unif} is uniform also with respect to $(\w_1,\w_2,x)\in
\W_1\times\W_2\times\bar\mB_\mu$ for any $\mu>0$ fixed. We assume for contradiction the
existence of $\ep>0$ and sequences $(\w_1^k)$, $(\w_2^k)$, $(x^k)$ and $(T_k)\uparrow\infty$
such that $|\hat F(\w_2^k,x^k)-(1/T_k)\int_0^{T_k} F((\w_1^k){\cdot}s,\w_2^k,x^k)\,ds|\ge\ep$,
and, without restriction, the existence of
$\w_2^*:=\lim_{k\to\infty}\w_2^k$ and $x^*:=\lim_{k\to\infty}x^k$.
We look for $k_0$ such that $|F(\w_1,\w_2^k,x^k)-F(\w_1,\w_2^*,x^*)|
\le\ep/3$ for all $\w_1\in\W_1$ if $k\ge k_0$, and deduce from \eqref{eq:4unif} that
$|\hat F(\w_2^k,x^k)-\hat F(\w_2^*,x^*)|\le\ep/3$ if $k\ge k_0$.
So,
\[
\begin{split}
 &\left|\hat F(\w_2^*,x^*)-\frac{1}{T_k}\int_0^{T_k}F((\w_1^k){\cdot}s,\w_2^*,x^*)\,ds\right|\\
 &\qquad \ge \left|\hat F(\w_2^k,x^k)-\frac{1}{T_k}\int_0^{T_k}F((\w_1^k){\cdot}s,\w_2^k,x^k)\,ds\right|
 -\left|\hat F(\w_2^*,x^*)-\hat F(\w_2^k,x^k)\right|\\
 &\qquad -\left|\frac{1}{T_k}\int_0^{T_k}\big(F((\w_1^k){\cdot}s,\w_2^k,x^k)-
 F((\w_1^k){\cdot}s,\w_2^*,x^*)\big)\,ds\right|\ge \ep-\frac{\ep}{3}-\frac{\ep}{3}=\frac{\ep}{3}\;,
\end{split}
\]
for all $k\ge k_0$, and this contradicts the uniformity of the limit \eqref{eq:4unif}.
It follows easily that
\[
 \lim_{T\to\infty} \frac{1}{T}\int_\tau^{T+\tau}\!\!\!\!f(s,t,x)\,ds=
 \!\lim_{T\to\infty}\frac{1}{T}\int_0^T\!\!
 F((\bar\w_1{\cdot}\tau){\cdot}s,\bar\w_2{\cdot}t,x)\,ds=
 \hat F(\bar\w_2{\cdot}t,x)
\]
and that the convergence is uniform on $(\tau,t,x)\in\R_+\!\times\R\times\bar\mB_\mu$ for any $\mu>0$.
Note that this means that the average \eqref{def:3promedio} of $f$ is given by
$\hf(t,x)=\hat F(\bar\w_2{\cdot}t,x)$ and that \ref{f3} and \ref{f4} are satisfied.

Regarding the remaining hypotheses on Theorem \ref{teor:4tracking}, they
concern the equation $dz/dt=\hf(t,z)$.
A priori, nothing guarantees that the hull construction described in Section \ref{subsec:32}
leads to the family $dz/dt=\hat F(\w_2{\cdot}t,z)$ for $\w_2\in\W_2$,
but this is not an inconvenient: the proof of Theorem \ref{teor:4tracking} can
be repeated if we assume the existence of a local attractor
$\mA$ for the local skewproduct flow induced by the family on $\W_2\times\R^n$
which projects onto the whole base, and we also assume that $(\bar\w_2,\hat x)$ belongs
to the basin of attraction of $\mA$.

Altogether, we have established a set on conditions on $\W_1$,
$F$ and its mean $\hat F$ ensuring
that the conclusions of Theorem \ref{teor:4tracking} apply to the equation
\eqref{eq:4notaF}. In addition, these conditions are independent of the choice
of $(\bar\w_1,\bar\w_2)$: assuming them and the existence of a local attractor
$\mA$ for the skewproduct flow induced by $dz/dt=\hat F(\w_2{\cdot}t,z)$
which projects onto the whole base, each one of the equations of the
family \eqref{eq:4Ffam} for which $(\w_2,\hat x)$ is in the attraction set of $\mA$
satisfies all the hypotheses of Theorem \ref{teor:4tracking}.

We close this remark by pointing out that the hull of the constructed
map $\hf$ satisfies the minimality condition of Proposition \ref{prop:4density} if
the flow $(\W_2,\sigma_2)$ is minimal.
\end{nota}
%
%
%
%
%
\subsection{The continuity of the map $\w\mapsto\mA_\w$}\label{subsec:41}
The information provided by Theorem \ref{teor:4tracking} is more precise in
the case of continuity of the fiber map
\begin{equation}\label{def:4section}
 \W_\mA=\Wf\to\mP_c(\R^n)\;,\quad\w\mapsto\mA_\w\,,
\end{equation}
for the Hausdorff metric of $\mP_c(\R^n)$,
which is not always the case: see Example \ref{ex:5fibranocont2}.
To establish conditions characterizing this continuity is the main goal of this section,
accomplished in Theorem \ref{th:4cont}.

So, throughout Section \ref{subsec:41},
we assume \ref{f1}, \ref{f2} and \ref{f3} as well as
the existence of a local attractor $\mA$ for the
flow \eqref{def:3pi}, with projection $\W_\mA=\Wf$. The number $\mu>0$, fixed, is given by
Definition \ref{def:3atractor}. Recall that, if $\mC\subseteq\WRn$, then
$\mC_\w:=\{x\in\R^n\,|\;(\w,x)\in\mC\}$. We also
call $v_t(\w,\mD):=\{v(t;\w,x)\,|\:x\in\mD\}$ for $\mD\subseteq\R^n$.

The mentioned conditions will be formulated in terms of the
uniform forward attractive character of $\mA$. This is closely related
to the contents of \cite[Section 3.2]{klra}. We begin by
analyzing the local pullback attractiveness properties of $\mA$.
\begin{prop}\label{prop:4siempre}
\begin{itemize}
\item[\rm(i)] The condition
$\lim_{t\to\infty}\dist_\WRn\!\big(\pi_t(\bar\mB_\mu(\mA)),\,\mA\big)=0$ of
Definition {\rm \ref{def:3atractor}} is equivalent to
\begin{equation}\label{3.rel:equiv-alt}
  \quad\lim_{t\to\infty}\sup_{\w\in\Wf}\big(\dist_{\R^n}\!
  \big(v_t\big(\w,(\bar\mB_\mu(\mA))_\w\big),\,
  \mA_{\wt}^{[\delta]}\big)\big)=0\;\text{ for all $\delta>0$}\,.
\end{equation}
\item[\rm(ii)] For any $\bar\w\in\Wf$,
\[
 \lim_{t\to\infty}\dist_{\R^n}\!\big(v_t\big(\bar\w{\cdot}(-t),
 (\bar\mB_\mu(\mA))_{\bar\w{\cdot}(-t)}\big),\,\mA_{\bar\w}\big)=0\,.
\]
\end{itemize}
\end{prop}
\begin{proof}
(i) Assume that $\lim_{t\to\infty}\dist_\WRn\!\big(\pi_t(\bar\mB_\mu(\mA)),\,\mA\big)=0$.
Given $\delta>0$, we take $\sigma\in(0,\delta]$ and $t_\sigma>0$ such that,
if $t\ge t_\sigma$, then $\dist_\WRn\!\big(\pi(t;\w,x),\,\mA\big)\le\sigma$ for all
$(\w,x)\in\bar\mB_\mu(\mA)$, which yields
$\dist_{\R^n}\!\big(v(t;\w,x),\,\mA_{\wt}^{[\delta]}\big)\le
\dist_{\R^n}\!\big(v(t;\w,x),\,\mA_{\wt}^{[\sigma]}\big)\le\sigma$
for all $(\w,x)\in\bar\mB_\mu(\mA)$: see Lemma \ref{lema:3inflado}(i).
Relation \eqref{3.rel:equiv-alt} follows easily from here. Conversely,
let us assume that \eqref{3.rel:equiv-alt} holds. Given $\sigma>0$, we take
$t_\sigma>0$ such that, if $t\ge t_\sigma$, then
$\dist_{\R^n}\!\big(v(t;\w,x),\,\mA_{\wt}^{[\sigma]}\big)\le\sigma$
for all $(\w,x)\in\bar\mB_\mu(\mA)$. Applying again Lemma \ref{lema:3inflado}(i),
we conclude that $\dist_\WRn\!\big(\pi(t;\w,x),\,\mA\big)\le\sigma$ for
all $(\w,x)\in\bar\mB_\mu(\mA)$ if $t\ge t_\sigma$, and this ensures that
$\lim_{t\to\infty}\dist_\WRn\!\big(\pi_t(\bar\mB_\mu(\mA)),\,\mA\big)=0$.
\smallskip

(ii)
We take $\sigma>0$. Lemma \ref{lema:3inflado}(ii) ensures the existence of $\delta>0$ such that
$\dist_{\R^n}\!\big(\mA_{\bar\w}^{[\delta]},\,\mA_{\bar\w}\big)<\sigma/2$.
It follows from (i) the existence of $t_\sigma>0$ such that
$\dist_{\R^n}\!\big(v_t\big(\bar\w{\cdot}(-t),(\bar\mB_\mu(\mA))_{\bar\w{\cdot}(-t)}\big),\,
\mA_{\bar\w}^{[\delta]}\big)\le\sigma/2$ if $t\ge t_\sigma$.
Since $\mA_{\bar\w}^{[\delta]}$ is a compact subset of $\R^n$,
(ii) follows from the triangular inequality for $\dist_{\R^n}$.
\end{proof}
Proposition \ref{prop:4siempre}(ii) shows that
$\mA$ is a {\em local pullback attractor\/}
for the flow \eqref{def:3pi}, which one can define
by adapting \cite[Definition 3.11]{klra} of global pullback attractor.
But the lack of uniformity with respect to $\bar\w\in\Wf$ in Lemma \ref{lema:3inflado}(ii)
precludes the lack of uniformity in the limit in Proposition
\ref{prop:4siempre}(ii); i.e., $\mA$ is not a local {\em uniform\/} pullback attractor,
and hence it is not a local uniform forward attractor $\mA$.
This uniformity (as local pullback or forward attractor)
is, in fact, equivalent to the continuity of the fiber map:
\begin{teor}\label{th:4cont}
The fiber map \eqref{def:4section} is continuous if and only if
\begin{equation}\label{eq:3caract}
  \lim_{t\to\infty}\sup_{\w\in\Wf}\big(\dist_{\R^n}
  \!\big(v_t\big(\w,(\bar\mB_\mu(\mA))_\w\big),\, \mA_{\wt}\big)\big)=0\,.
\end{equation}
\end{teor}
\begin{proof}
To prove the {\em only if\/} assertion, we reason as in the proof of
Proposition \ref{prop:4siempre}(ii),
using now the information provided by Lemma \ref{lema:3inflado}(iii). This leads us to
\[
 \lim_{t\to\infty}\sup_{\w\in\Wf}\dist_{\R^n}\!\big(v_t\big(\w{\cdot}(-t),
 (\bar\mB_\mu(\mA))_{\w{\cdot}(-t)}\big),\,\mA_\w\big)=0\,,
\]
and hence \eqref{eq:3caract} follows from
$\sup_{\w\in\Wf}\dist_{\R^n}\!\big(v_t\big(\w{\cdot}(-t),
(\bar\mB_\mu(\mA))_{\w{\cdot}(-t)}\big),\,\mA_\w\big)=
\sup_{\w\in\Wf}\dist_{\R^n}\!\big(v_t\big(\w,
 (\bar\mB_\mu(\mA))_{\w}\big),\,\mA_{\wt}\big)$.

The converse proof is more complex. First, we check that,
if $\di_\Wf(\w_1,\w_2)\le\mu/2$, then
\begin{itemize}
\item[(a)] $\mA_{\w_1}\subseteq(\bar\mB_{\mu/2}(\mA))_{\w_2}
\subseteq(\bar\mB_\mu(\mA))_{\w_1}\,.$
\end{itemize}
If $(\w_1,x)\in\mA$ (i.e., if $x\in\mA_{\w_1}$)
and $\di_\Wf(\w_1,\w_2)\le\mu/2$, then $\dist_{\WRn}\big((\w_2,x),\,\mA\big)\le
\di_{\WRn}\big((\w_2,x),\,(\w_1,x)\big)+\dist_\WRn\big((\w_1,x),\,\mA\big)=
\di_\Wf(\w_2,\w_1)+0\le \mu/2$; that is, $(\w_2,x)\in\bar\mB_{\mu/2}(\mA)$,
which shows the first inclusion. And, if $(\w_2,x)\in\bar\mB_{\mu/2}(\mA)$ and
$\di_\Wf(\w_1,\w_2)\le \mu/2$, then
\[
\dist_\WRn\big((\w_1,x),\,\mA\big)\le
\di_\Wf(\w_1,\w_2)+\dist_\WRn\big((\w_2,x),\,\mA\big)
\le \mu/2+\mu/2=\mu\,,
\]
which proves the second one.

Our current hypotheses ensure that
$\lim_{t\to\infty}d_t(\w)=0$ uniformly on $\Wf$ for
$d_t(\w):=\dist_{\R^n}\!\big(v_t\big(\w,(\bar\mB_\mu(\mA))_\w\big),\,\mA_{\wt}\big)$.
Let us fix $\bar\w\in\Wf$ and $\sigma>0$. We choose $t_\sigma>0$
such that $d_{t_\sigma}(\w)\le\sigma/3$ for all $\w\in\Wf$.
Now we will check the existence of $\delta_1=\delta_1(\bar\w,\sigma)\in(0,\mu/2]$ and
$\delta_2=\delta_2(\delta_1)>0$ such that, if $\di_\Wf(\bar\w,\w)\le \delta_2$, then
\begin{itemize}
\item[(b)] $\di_\Wf(\bar\w{\cdot}(-t_\sigma),\,\w{\cdot}(-t_\sigma))\le\delta_1\le \mu/2$,
\item[(c)] $|v(t_\sigma;\bar\w{\cdot}(-t_\sigma),x)-v(t_\sigma;\w{\cdot}(-t_\sigma),x)|\le \sigma/3$
for all $x\in\bigcup_{\w\in\Wf}(\bar\mB_{\mu/2}(\mA))_\w$,
\item[(d)] $\dist_{\R^n}\!\big(v_t(\w{\cdot}(-t_\sigma),
    (\bar\mB_{\mu/2}(\mA))_{\bar\w{\cdot}(-t_\sigma)}),\,\mA_\w\big)\le\sigma/3$.
\end{itemize}
Since the map $(\w,x)\mapsto v(t_\sigma;\w,x)$ is uniformly continuous
on the compact subsets of $\WRn$, there exists $\delta_1\in(0,\mu/2]$ such that
$|v(t_\sigma;\w_1,x)-v(t_\sigma;\w_2,x)|\le\sigma/3$ for all $\di_\Wf(\w_1,\w_2)\le\delta_1$
and $x\in\bigcup_{\w\in\Wf}(\bar\mB_{\mu/2}(\mA))_\w$. So, the continuity at
$\bar\w$ of the map $\w\mapsto\w{\cdot}(-t_\sigma)$ ensures the existence of
$\delta_2>0$ such that properties (b) and (c) hold if $\di_\Wf(\bar\w,\w)\le \delta_2$,
what we assume for $\w$ in what follows.
Since $\di_\Wf(\bar\w{\cdot}(-t_\sigma),\,\w{\cdot}(-t_\sigma))\le\delta_1\le \mu/2$, (a) yields
$\mA_{\w{\cdot}(-t_\sigma)}\subseteq(\bar\mB_{\mu/2}(\mA))_{\bar\w{\cdot}(-t_\sigma)}
\subseteq(\bar\mB_\mu(\mA))_{\w{\cdot}(-t_\sigma)}$. The second inclusion
combined with the choice of $t_\sigma$ shows that
\[
\begin{split}
 &\dist_{\R^n}\!\big(v_{t_\sigma}\big(\w{\cdot}(-t_\sigma),
 (\bar\mB_{\mu/2}(\mA))_{\bar\w{\cdot}(-t_\sigma)}\big),\,\mA_\w\big)\\
 &\qquad\qquad \le \dist_{\R^n}\!\big(v_{t_\sigma}\big(\w{\cdot}(-t_\sigma),
 (\bar\mB_\mu(\mA))_{\w{\cdot}(-t_\sigma)}\big),\,\mA_\w\big)\le d_{t_\sigma}(\w{\cdot}(-t_\sigma))
 \le\sigma/3\,,
\end{split}
\]
which proves (d).

To complete the proof, we take $\w\in\Wf$ with $\di_\Wf(\bar\w,\w)\le\delta_2$. Then,
\[
\begin{split}
 &\di^H_{\R^n}\big(\mA_\w,\,\mA_{\bar\w}\big)
 \le\di^H_{\R^n}\!\big(\mA_\w,v_{t_\sigma}\big(\w{\cdot}(-t_\sigma),\,
 (\bar\mB_{\mu/2}(\mA))_{\bar\w{\cdot}(-t_\sigma)}\big)\big)\\
 &\qquad\qquad\quad +\di^H_{\R^n}\!\big(v_{t_\sigma}
 \big(\w{\cdot}(-t_\sigma),(\bar\mB_{\mu/2}(\mA))_{\bar\w{\cdot}(-t_\sigma)}\big),
 v_{t_\sigma}\big(\bar\w{\cdot}(-t_\sigma),\,
 (\bar\mB_{\mu/2}(\mA))_{\bar\w{\cdot}(-t_\sigma)}\big)\big)\\
 &\qquad\qquad\quad +\di^H_{\R^n}\!\big(v_{t_\sigma}\big(\bar\w{\cdot}(-t_\sigma),
 (\bar\mB_{\mu/2}(\mA))_{\bar\w{\cdot}(-t_\sigma)}\big),\,\mA_{\bar\w}\big)\,.
\end{split}
\]
Note that
$v_{t_\sigma}\big(\w{\cdot}(-t_\sigma),(\bar\mB_{\mu/2}(\mA))_{\bar\w{\cdot}(-t_\sigma)}\big)
\supseteq v_{t_\sigma}\big(\w{\cdot}(-t_\sigma),\mA_{\w{\cdot}(-t_\sigma)})=\mA_\w$
(as deduced from the inclusion $(\bar\mB_{\mu/2}(\mA))_{\bar\w{\cdot}(-t_\sigma)}
\supseteq \mA_{\w{\cdot}(-t_\sigma)}$, previously obtained). So, the
definition of $\di^H_{\R^n}$ and (d) yield
\[
\begin{split}
 &\di^H_{\R^n}\!\big(\mA_\w,v_{t_\sigma}\big(\w{\cdot}(-t_\sigma),\,
 (\bar\mB_{\mu/2}(\mA))_{\bar\w{\cdot}(-t_\sigma)}\big)\big)\\
 &\qquad\qquad\qquad =
 \dist_{\R^n}\!\big(v_{t_\sigma}\big(\w{\cdot}(-t_\sigma),
 (\bar\mB_{\mu/2}(\mA))_{\bar\w{\cdot}(-t_\sigma)}\big),\,\mA_\w\big)\le\sigma/3\,.
\end{split}
\]
It is easy to deduce from property (c) that
\[
 \di^H_{\R^n}\!\big(v_{t_\sigma}\big(\w{\cdot}(-t_\sigma),
 (\bar\mB_{\mu/2}(\mA))_{\bar\w{\cdot}(-t_\sigma)}\big),\,
 v_{t_\sigma}\big(\bar\w{\cdot}(-t_\sigma),
 (\bar\mB_{\mu/2}(\mA))_{\bar\w{\cdot}(-t_\sigma)}\big)\big)\le\sigma/3\,.
\]
Finally, $v_{t_\sigma}\big(\bar\w{\cdot}(-t_\sigma),
(\bar\mB_{\mu/2}(\mA))_{\bar\w{\cdot}(-t_\sigma)}\big)\supseteq
v_{t_\sigma}\big(\bar\w{\cdot}(-t_\sigma),\mA_{\bar\w{\cdot}(-t_\sigma)})=
\mA_{\bar\w}$, and hence, as in the proof of (d),
\[
\begin{split}
 &\di^H_{\R^n}\!\big(v_{t_\sigma}\big(\bar\w{\cdot}(-t_\sigma),
 (\bar\mB_{\mu/2}(\mA))_{\bar\w{\cdot}(-t_\sigma)}\big),\,\mA_{\bar\w}\big)\\
 &\qquad\qquad\qquad =
 \dist_{\R^n}\!\big(v_{t_\sigma}\big(\bar\w{\cdot}(-t_\sigma),
 (\bar\mB_{\mu/2}(\mA))_{\bar\w{\cdot}(-t_\sigma)}\big),\,\mA_{\bar\w}\big)\le\sigma/3\,.
\end{split}
\]
Altogether, we conclude that $\di^H_{\R^n}\big(\mA_\w,\mA_{\bar\w}\big)\le
\sigma$ if $\di_\Wf(\w,\bar\w)\le\delta_2$, which completes the proof.
\end{proof}
We complete this section with a result proved in \cite[Theorem 3.3]{noos4},
which explains a condition under which the fiber map is continuous if
$\Wf$ is minimal.
\begin{defi}\label{def:4estable}
A compact $\pi$-invariant set $\mK\subset\WRn$ is {\em uniformly stable with
respect to itself} if, for any $\sigma>0$, there exists $\delta_\sigma>0$
such that, if $(\w,x),\,(\w,y)\in\mK$ satisfy
$|x-y|<\delta_\sigma$, then $|v(t;\w,x)-v(t;\w,y)|<\sigma$ for all $t\ge 0$.
The set $\mK$ is {\em uniformly asymptotically stable with respect to itself}
if it is uniformly stable with
respect to itself and, besides, there exists $\delta_0>0$ such that, if
$(\w,x),\,(\w,y)\in\mK$ satisfy
$|x-y|<\delta_0$, then $\lim_{t\to\infty}|v(t;\w,x)-v(t;\w,y)|=0$ uniformly
with respect to $(\w,x)\in\mK$.
\end{defi}
\begin{teor}\label{teor:4minimal}
Assume that $\Wf$ is minimal. If $\mA$ is uniformly stable with respect to itself,
then the fiber map \eqref{def:4section} is continuous.
\end{teor}
\subsection{The shape of the local attractor $\mA$}
We continue working under conditions \ref{f1}, \ref{f2} and \ref{f3} and
the assumption of existence of a local attractor $\mA$ for the
flow \eqref{def:3pi} with projection $\W_\mA=\Wf$. The more information we have about
the shape of this attractor, the more we can take advantage of the tracking
information provided by Theorem \ref{teor:4tracking}. The objective of this section
is to establish conditions under which this shape is relatively simple,
which we do in Theorem \ref{teor:4uae}.
\begin{defi}
A compact $\pi$-invariant set $\mK\subset\WRn$ is a {\em finite-cover of $\Wf$} if
there exists $m\in\N$ such that $1\le\text{\rm card\,}\mK_\w\le m$ for all $\w\in\Wf$.
A finite-cover $\mK\subset\WRn$ of $\Wf$ is an {\em $m$-cover of $\Wf$}, with $m\in\N$,
if $\text{\rm card\,}\mK_\w=m$ for all $\w\in\Wf$ and there exists
$\delta>0$ such that $|x_1-x_2|\ge\delta$ for all $\w\in\Wf$ and $x_1,x_2\in\mK_\w$
with $x_1\ne x_2$.
\end{defi}
Recall the existence of a $\sigma_\hf$-invariant residual subset of continuity points for the fiber map
$\Wf\to\mP_c(\R^n),\,\w\mapsto\mA_\w$: see Remark \ref{rm:3upper}. This is the
set $\W_\mA^{\text{\rm cont}}$ of the following statement.
\begin{teor} \label{teor:4uae}
Assume that $\mA$ is uniformly asymptotically stable with respect to itself, and
let $\delta_0$ be the constant of Definition {\rm\ref{def:4estable}}. Then,
\begin{itemize}
\item[\rm(i)] $\mA$ is a finite-cover of $\Wf$: there exists $m\in\N$
such that $1\le\text{\rm card\,}\mA_\w\le m$ for all $\w\in\Wf$.
\item[\rm(ii)] The map $\Wf\to\{1,\ldots,m\},\,\w\mapsto\text{\rm card\,}\mA_\w$
is constant along the $\sigma_\hf$-orbits.
\end{itemize}
For each $\w\in\Wf$, let $\delta_{\min}(\w)$ be $\infty$ if
$\text{\rm card\,}\mA_\w=1$ and the minimum distance
between different elements of $\mA_\w$ otherwise. Then,
\begin{itemize}
\item[\rm(iii)] for every $\w\in\Wf$ there exists
$t_\w<0$ such that $\delta_{\min}(\wt)\ge\delta_0$ for all $t\le t_\w$.
\item[\rm(iv)] If there exists a point $\w_0$ with dense $\sigma_\hf$-orbit, then
$\text{\rm card\,}\mA_{\bar\w}\le\text{\rm card\,}\mA_{\w_0}$ for
all $\bar\w\in\W_\mA^{\text{\rm cont}}$.
\item[\rm(v)] If there exists a point $\w_0$ with \upalfa-limit set
equal to $\Wf$, then $\text{\rm card\,}\mA_\w\ge\text{\rm card\,}\mA_{\w_0}$
for all $\w\in\Wf$; and $\text{\rm card\,}\mA_{\bar\w}=\text{\rm card\,}\mA_{\w_0}$ and
$\delta_{\min}(\bar\w)\ge\delta_0$ for all $\bar\w\in\W_\mA^{\text{\rm cont}}$.
\item[\rm(vi)] If $\Wf$ is minimal, then $\mA$ is an $m$-cover of $\Wf$,
with $\delta_{\min}(\w)\ge\delta_0$ for all $\w\in\Wf$.
\item[\rm(vii)] If there exists a point $\w_0\in\W_\mA^{\text{\rm cont}}$ with \upomeg-limit set
equal to $\Wf$, then: $\text{\rm card\,}\mA_\w\ge\text{\rm card\,}\mA_{\w_0}$
for all $\w\in\Wf$; and $\text{\rm card\,}\mA_{\bar\w}=\text{\rm card\,}\mA_{\w_0}$
and $\delta_{\min}(\bar\w)\ge\delta_0$
for all $\bar\w\in\W_\mA^{\text{\rm cont}}$.
\end{itemize}
\end{teor}
\begin{proof}
(i) Let us define $\mA_{\R^n}:=\{x\in\R^n\,|\;\exists\,(\w,x)\in\mA\}$ and denote
$\mB_\rho(x):=\{y\in\R^n\,|\;|x-y|<\rho\}$. Since $\mA_{\R^n}$ is compact and contained in
$\text{{\Large$\cup$}}_{x\in\mA_{\R^n}}\mB_{\delta_0/2}(x)$,
there exists a finite subcover, say $\mA_{\R^n}\subset\mB_{\delta_0/2}(y_1)\,
\cup\,\cdots\,\cup\,\mB_{\delta_0/2}(y_m)$.
Let us see that $0<\text{\rm card\,}\mA_\w\le m$ for all $\w\in\W_\mA$,
which combined with $\W_\mA=\Wf$ proves (i).
We assume for contradiction the existence of $\bar\w\in\Wf$ with
$\text{\rm card\,}\mA_{\bar\w}\ge m+1$, take different points
$x_1,\ldots,x_{m+1}\in\mA_{\bar\w}$, and define
$\sigma_0:=\min_{1\le i<j\le m+1}|x_i-x_j|$.
Definition \ref{def:4estable} provides $t_0=t_0(\sigma_0)$ such that
$|v(t_0;\w,z_1)-v(t_0;\w,z_2)|<\sigma_0$ if $(\w,z_1),\,(\w,z_2)\in\mA$ and
$|z_1-z_2|<\delta_0$. Let us call
$\mC_i:=\mB_{\delta_0/2}(y_i)\cap\mA_{\bar\w{\cdot}(-t_0)}$. It is easy to check that
$\mA_{\bar\w}=v(t_0;\bar\w{\cdot}(-t_0),\mC_1)\,\cup\,\cdots\,\cup\,
v(t_0;\bar\w{\cdot}(-t_0),\mC_m)$. Hence, at least one of the sets, say
$v(t_0;\bar\w{\cdot}(-t_0),\mC_j)$, contains two elements of $\mA_{\bar\w}$,
say $x_{i_1}$ and $x_{i_2}$. Then, there exist $z_1,z_2\in\mC_j$ such that
$|x_{i_1}-x_{i_2}|=|v(t_0;\bar\w{\cdot}(-t_0),z_1)-v(t_0;\bar\w{\cdot}(-t_0),z_2)|
<\sigma_0\le|x_{i_1}-x_{i_2}|$, since $|z_1-z_2|<\delta_0$.
This is the sought-for contradiction.
\smallskip

(ii) The map $t\mapsto\text{card\,}\mA_{\wt}$ is constant for all $\w\in\Wf$,
since $\mA_{\wt}=v(t;\w,\mA_\w)$.
\smallskip

(iii) We fix any $\w\in\Wf$, and observe that there is nothing to prove if
$\text{card\,}\mA_\w=1$. Let us assume that this is not the case,
write $\mA_\w=\{x_\w^1,\ldots, x_\w^l\}$ with $2\le l\le m$,
and note that $\mA_{\wt}=\{v(t;\w,x_\w^1),\ldots, v(t;\w,x_\w^l)\}$
for all $t\in\R$. Definition \ref{def:4estable} provides $t_0=t_0(\w)$
such that, if $s\in\R$ and $(\ws,x_1),(\ws,x_2)\in\mA$ satisfy $|x_1-x_2|<\delta_0$,
then $|v(t;\ws,x_1)-v(t;\ws,x_2)|<\delta_{\min}(\w)$ for all $t\ge t_0$. So,
$\min_{1\le i<j\le l}|v(-t;\w,x^i_\w)-v(-t;\w,x^j_\w)|\ge\delta_0$
for $t\ge t_0$:
if, on the contrary, $|v(-t;\w,x^i_\w)-v(-t;\w,x^j_\w)|<\delta_0$, then,
for $t\ge t_0$,
\[
 |x^i_\w-x^j_\w|=|v(t;\w{\cdot}(-t),v(-t;\w,x^i_\w))-
 v(t;\w{\cdot}(-t),v(-t;\w,x^j_\w))|<\delta_{\min}(\w)\,,
\]
and so $i=j$. This means that $t_\w=-t_0$ fulfills (ii).\smallskip

(iv) Let us fix $\w_0$ as in the statement and $\bar\w\in\W_\mA^{\text{\rm cont}}$. The density of the
$\sigma_\hf$-orbit of $\w_0$ and the continuity at $\bar\w$ of the fiber map allow
us to find $t\in\R$ such that $\dist_{\R^n}\!\big(\mA_{\bar\w},\,\mA_{\w_0{\cdot}t}\big)
\le\di^H_{\R^n}\big(\mA_{\bar\w},\,\mA_{\w_0{\cdot}t}\big)<\delta_{\min}(\bar\w)/2$,
which ensures that $\text{card\,}\mA_{\bar\w}\le \text{card\,}\mA_{\w_0{\cdot}t}$: otherwise
we would find two points in $\mA_{\bar\w}$ at distance less that $\delta_{\min}(\bar\w)$,
which is impossible. So,
(ii) yields $\text{card\,}\mA_{\bar\w}\le\text{card\,}\mA_{\w_0}$.
\smallskip

(v) Let us fix $\w_0$ as in the statement and any other $\w\in\Wf$.
Since $\w$ belongs to the \upalfa-limit set of $\w_0$, the
upper semicontinuity ensured by Proposition \ref{prop:3semicont}(i) allows us
to find $t\le t_{\w_0}$, with $t_{\w_0}$ provided by (iii), such that
$\dist_{\R^n}\!\big(\mA_{\w_0{\cdot}t},\,\mA_\w\big)<\delta_0/2
\le \delta_{\min}(\w_0{\cdot}t)/2$.
As in (iv), this is only possible if
$\text{card\,}\mA_{\w_0{\cdot}t}\le\text{card\,}\mA_\w$,
and hence (ii) yields $\text{card\,}\mA_{\w_0}\le\text{card\,}\mA_\w$.

Now we take $\bar\w\in\W_\mA^{\text{\rm cont}}$ and observe that the property already proved
combined with (iv) ensure that $\text{\rm card\,}\mA_{\bar\w}=\text{\rm card\,}\mA_{\w_0}$.
We write $\bar\w=\lim_{k\to\infty}\w_0{\cdot}t_k$ for $(t_k)\downarrow-\infty$.
Then, $\lim_{k\to\infty}\dist_{\R^n}\!\big(\mA_{\bar\w},\,\mA_{\w_0{\cdot}t_k}\big)\le
\lim_{k\to\infty}\di^H_{\R^n}\big(\mA_{\bar\w},\,\mA_{\w_0{\cdot}t_k}\big)=0$.
It is easy to deduce that, for all $\sigma>0$, there exists $k_\sigma>0$
such that $\delta_{\min}(\w_0{\cdot}t_k)\le \delta_{\min}(\bar\w)+\sigma$ for all $k\ge k_\sigma$,
and hence we deduce from (iii) that $\delta_0\le\delta_{\min}(\bar\w)+\sigma$ for all $\sigma>0$.
That is, $\delta_0\le\delta_{\min}(\bar\w)$.
\smallskip

(vi) According to Theorem \ref{teor:4minimal}, $\W_\mA^{\text{\rm cont}}=\Wf$, and the minimality ensures that
$\Wf$ is the \upalfa-limit set of any of its elements. Hence, (vi) follows from (v).
\smallskip

(vii) We write $\w_0=\lim_{k\to\infty}\w_0{\cdot}t_k$ for a
suitable sequence $(t_k)\uparrow\infty$. Note that
$\text{card\,}\mA_{\w_0{\cdot}t_k}=\text{card\,}\mA_{\w_0}$ for all $k\in\N$,
and that $\lim_{k\to\infty}\di^H_{\R^n}(\mA_{\w_0{\cdot}t_k},\mA_{\w_0})=0$ since
$\w_0\in\W_\mA^{\text{\rm cont}}$. So, each point $x_{\w_0}^i$ of $\mA_{\w_0}$ can be
written as $\lim_{k\to\infty}x_{\w_0{\cdot}t_k}^i$
with $x_{\w_0{\cdot}t_k}^i\in\mA_{\w_0{\cdot}t_k}$, which is only possible if
$\delta_{\min}(\w_0{\cdot}t_k)>\rho:=\delta_{\min}(\w_0)/2$ for large enough $k$.
Since $(t_k)\uparrow\infty$, we conclude that
$\delta_{\min}(\w_0{\cdot}t)\ge\delta_0$ for all $t\in\R$: if
$(\w_0{\cdot}s,x^1_s),\,(\w_0{\cdot}s,x^2_s)\in\mA$ satisfy $0<|x_s^1-x_s^2|<\delta_0$,
then $|v(t_k-s;\w_0{\cdot}s,x^1_s)-v(t_k-s;\w_0{\cdot}s,x^2_s)|<\rho<
\delta_{\min}(\w_0{\cdot}t_k)$ for $k$ large enough, impossible.
\par
Now we write any $\w\in\Wf$ as $\lim_{k\to\infty}\w_0{\cdot}s_k$ with $(s_k)\uparrow\infty$,
deduce from Proposition \ref{prop:3semicont}(i) that
$\lim_{k\to\infty}\dist_{\R^n}\!\big(\mA_{\w_0{\cdot}s_k},\,\mA_\w\big)=0$, take
$k$ large enough to ensure that $\dist_{\R^n}\!\big(\mA_{\w_0{\cdot}s_k},\,\mA_\w\big)<
\delta_0/2\le\delta_{\min}(\w_0{\cdot}s_k)/2$, and deduce as in (iv) that
$\text{card\,}\mA_{\w_0}=\text{card\,}\mA_{\w_0{\cdot}s_k}\le\text{card\,}\mA_\w$.
We do this for $\w=\bar\w\in\W_\mA^{\text{\rm cont}}$ and apply (iv) to conclude that
$\text{card\,}\mA_{\bar\w}\le\text{card\,}\mA_{\w_0}$, and hence that they are equal.
Finally, assume for contradiction that $\delta_{\min}(\bar\w)<\delta_0$. Since
$\bar\w$ is a continuity point for the fiber map,
for large enough $k_0$ there exist two points $x^1_{k_0},\,x^2_{k_0}
\in\mA_{\w_0{\cdot}s_{k_0}}$ at a distance less than $\delta_0$,
which means that $\lim_{k\to\infty}|v(s_k-s_{k_0};\w_0{\cdot}s_{k_0},x^1_{k_0})-
v(s_k-s_{k_0};\w_0{\cdot}s_{k_0},x^2_{k_0})|=0$. This contradicts
$\text{card\,}\mA_{\w_0{\cdot}{s_{k_0}}}=\text{card\,}\mA_{\bar\w}$, so the proof is complete.
%
\end{proof}
Finally, observe that the $\sigma_\hf$-invariance of the map
$\Wf\to\N,\,\w\mapsto\text{\rm card\,}\mA_\w$ proved in Theorem \ref{teor:4uae}(ii)
ensures that it is constant almost always for every $\sigma_\hf$-ergodic measure on $\Wf$.
\section{Scope and optimality of Theorem \ref{teor:4tracking}} \label{sec:5}
The main goal of this section is to present two nontrivial examples showing the scope
and optimality of Theorem \ref{teor:4tracking}. All these examples deal with
almost periodic maps $c\colon\R\to\R$.
The properties described in Proposition \ref{prop:5ejemplos} and
Lemma \ref{lema:5ejemplos} are required to understand the behavior of the
solutions in the particular cases to be analyzed.
\begin{prop}\label{prop:5ejemplos}
Let $c\colon\R\to\R$ be an almost periodic map.
Let $\W_c$ be its hull, and let $\mc\colon\W_c\to\R$ be defined by $\mc(\w):=\w(0)$.
Then,
\begin{itemize}
\item[\rm(i)] for all $\tau\in\R$, there exists
$\hat c=\lim_{T\to\infty}(1/T)\int_\tau^{\tau+T}c(s)\,ds$, and
the convergence is uniform with respect to $\tau\in\R$.
\end{itemize}
Assume that $c$ satisfies
\begin{equation}\label{eq:5noED}
 \hat c=0
 \qquad\text{and}\qquad \sup_{t\in\R}\left|\int_0^t c(s)\,ds\right|=\infty\,.
\end{equation}
Then,
\begin{itemize}
\item[\rm(ii)] there exists $\bar\w\in\W_c$ such that
 $\sup_{t\in\R}\int_0^t\mc(\bar\w{\cdot}s)\,ds<\infty$.
\item[\rm(iii)] For every $\w$ in a residual subset $\W_c^{\text{\rm osc}}\subset\W_c$,
there exist four sequences $(s_k^\pm)\uparrow\infty$ and $(\bar s_k^\pm)\downarrow-\infty$
such that
$\lim_{k\to\infty}\int_0^{s_k^\pm}\mc(\w{\cdot}s)\,ds=\pm\infty$ and
$\lim_{k\to\infty}\int_0^{\bar s_k^\pm}\mc(\w{\cdot}s)\,ds=\pm\infty$.
\end{itemize}
\end{prop}
\begin{proof}
Assertion (i) is a well-known fact for almost periodic maps, which is explained
in Remark \ref{rm:4F}, since the flow on $\W_c$ is minimal and almost periodic
(see, e.g., \cite[Theorem 3.1 of Part I]{shyi4}), and hence
uniquely ergodic (see Section \ref{sec:2}). A more direct proof can be found in, e.g.,
\cite[Corollary 3.2]{fink}.
The solutions of the linear family $dx/dt=\mc(\wt)\,x$ for $\w\in\W_c$ are given by
$v_l(t;\w,x^*)=x^*\exp\int_0^t\mc(\ws)\,ds$.
Condition $\hat c=0$ ensures the lack of exponential dichotomy of the family (as easily
deduced by contradiction from the definition of exponential dichotomy: see, e.g.,
\cite[Definition 1.66]{jonnf}) and hence the existence of $\bar\w\in\W_c$
for which all the solutions of the linear equation
are bounded (see \cite[Theorem 10.2]{selg}), which yields (ii).
The oscillation property (iii) is proved in \cite[Theorem A.2]{jnot}.
\end{proof}
The almost periodic maps satisfying \eqref{eq:5noED}
are rather counterintuitive nonperiodic functions with unbounded integral
and zero average. Their existence is well known: see Remark \ref{rm:5.meager}.
We will give a particular two-parameter family of these maps in Example \ref{ex:5fibracont} and
use them also in Example \ref{ex:5fibranocont2}.
\begin{lema}\label{lema:5ejemplos}
Let $h\colon\R\to\R$ be the odd continuous map given by $0$ on $[0,1]$
and by $-(x-1)^2$ on $(1,\infty)$. Let $a,\,b\colon\R\to\R$ be
almost periodic maps.
\begin{itemize}
\item[\rm(i)]
If the average of $a$ is $\hat a=0$, then the map $f(\tau,t,x):=a(\tau)\,x+b(t)\,x+h(x)$
satisfies \ref{f1}, \ref{f2}, \ref{f3} and \ref{f4}, with $\hf(t,x)=b(t)\,x+h(x)$.
\end{itemize}
Let $\W_b$ be the hull of $b$, define $\mb\colon\W_b\to\R,\,\w\mapsto\w(0)$,
and let $\pi$ be the skewproduct flow \eqref{def:3pi} induced by the family
$dz/dt=\mb(\wt)\,z+h(z)$. Then,
\begin{itemize}
\item[\rm(ii)] all the forward $\pi$-orbits are bounded, and there exists
a global attractor $\bar\mA\subset\W_b\times\R$
which projects onto $\W_b$, is composed by all the globally $\pi$-bounded orbits, and
takes the shape $\bar\mA=\text{{\Large$\cup$}}_{\w\in\W_b}\big(\{\w\}\times[-\muk(\w),\muk(\w)]\big)$
for an upper semicontinuous map $\muk\colon\W_b\to\R_+$ with $\pi$-invariant graph.
\end{itemize}
Assume also that $b$ satisfies \eqref{eq:5noED}.
Then,
\begin{itemize}
\item[\rm(iii)] $\W_b^{\text{\rm cont}}=\{\w\in\W_b\,|\;\sup_{t\le 0}\int_0^t\mb(\ws)\,ds=\infty\}$
is the (residual $\sigma_\hf$-invariant) set of continuity points of $\muk$
and of the fiber map $\W_b\to\mP_c(\R),\,\w\mapsto\bar\mA_\w$, and
$\muk(\w)=0$ if and only if $\w\in\W_b^{\text{\rm cont}}$.
\item[\rm(iv)] $\muk(\w)>0$ (or, equivalently, the fiber map $\W_b\to\mP_c(\R),\,\w\mapsto\bar\mA_\w$
is not continuous) at all the points $\w$ of a dense $\sigma_\hf$-invariant
subset of $\W_b$ of the first Baire category.
\item[\rm(v)] Let $y(t;0,x^*)$ be the solution of $dy/dt=b(t)\,y+h(y)$ with value
$x^*$ at $t=0$. Then, $\liminf_{t\to\infty}y(t;0,x^*)=0$. If, in addition,
$\sup_{t\ge 0}\int_0^t b(s)\,ds=\infty$, then
there exists a sequence $(t_k)\uparrow\infty$ such that
$y(t_k;0,x^*)>1$ for all $k\in\N$ if $x^*>0$, and $y(t_k;0,x^*)<-1$ for all $k\in\N$ if $x^*<0$.
\end{itemize}
\end{lema}
\begin{proof}
(i) Recall that an almost periodic map is bounded and uniformly continuous.
Property \ref{f1} follows easily from this boundedness; property \ref{f2} follows from
the uniform continuity of $b$ and the locally Lipschitz character of $h$; and $\hat a=0$
combined with Proposition \ref{prop:5ejemplos}(i) ensures \ref{f3} and \ref{f4} for
$\hf(t,x)=b(t)\,x+h(x)$.
\smallskip

(ii) The existence and composition of the global attractor $\bar\mA$ follow from the
strong coercivity property ``$\lim_{x\to\pm\infty}(\mb(\w)\,z+h(z))/z=-\infty$
uniformly on $\W$'': see Remark \ref{rm:3existen}. The shape of $\bar\mA$ and the semicontinuity of
$\muk$ are consequences of the compactness of $\bar\mA$ and of
$\pi(t;\w,-x^*)=-\pi(t;\w,x^*)$.
\smallskip

(iii) These assertions follow from \cite[Theorem 3.5]{cano} and Proposition \ref{prop:5ejemplos}(iii).
\smallskip

(iv) Proposition \ref{prop:5ejemplos}(ii) ensures the existence of
$\bar\w\in\W_b-\W_b^{\text{\rm cont}}$. Since $\W_b^{\text{\rm cont}}$
is residual and $\sigma_\hf$-invariant,
$\W_b-\W_b^{\text{\rm cont}}$ is of the first Baire category and $\sigma_\hf$-invariant.
In particular, it contains the $\sigma_\hf$-orbit $\{\bwt\,|\;t\in\R\}$,
which is dense since $\W_b$ is minimal (see Section \ref{sec:2}).
\smallskip

(v) We fix $x^*\ne 0$, so that $y(t;0,x^*)\ne 0$ whenever it is defined.
The property $\liminf_{t\to\infty}|y(t;0,x^*)|=0$ can be deduced from
Proposition \ref{prop:4density}(i), but we give a simple direct proof.
Let $v_l(t;\w,x^*)$ and $v(t;\w,x^*)$ be the solutions of $dy/dt=\mb(\wt)\,y$ and
$dy/dt=\mb(\wt)\,y+h(y)$ with value $x^*$ at $t=0$, so that $v(t;b,x^*)=y(t;0,x^*)$.
In particular, $v_l(t;\w,x^*)=
x^*\exp\int_0^t \mb(\ws)\,ds$. According to Proposition \ref{prop:5ejemplos}(iii),
if $\w\in\W_b^{\text{\rm osc}}$, then $\lim_{k\to\infty}\int_0^{s_k^-}\mb(\ws)\,ds=-\infty$
for a sequence $(s_k^-)\uparrow\infty$, and hence
$\liminf_{t\to\infty}|v_l(t;\w,x^*)|=0$.
Since $\pm h(x)\le 0$ for $\pm x>0$, an easy comparison argument shows that
$|v(t;\w,x^*)|\le |v_l(t;\w,x^*)|$ for all $t\ge 0$. Hence,
$\liminf_{t\to\infty}|v(t;\w,x^*)|=0$ for all $\w\in\W_b^{\text{\rm osc}}$.
This proves the result if $b\in \W_b^{\text{\rm osc}}$, since $y(t;0,x^*)=v(t;b,x^*)$.
For contradiction, we assume now that $b\notin\W_b^{\text{\rm osc}}$ and that
$\liminf_{t\to\infty}|v(t;b,x^*)|=\rho>0$.
We take $\bar\w\in\W_b^{\text{\rm osc}}$, write it as $\bar\w=\lim_{k\to\infty}b{\cdot}r_k$
for a sequence $(r_k)\uparrow\infty$ (which is possible since $\W_b$ is minimal), and
assume without restriction the existence of $\bar x:=\lim_{k\to\infty}v(r_k;b,x^*)$
(since $\{v(t;b,x^*)\,|\;t\ge 0\}$ is bounded, as ensured in (ii)). Then,
$(\bar\w{\cdot}t,v(t;\bar\w,\bar x))=
\lim_{k\to\infty}(b{\cdot}(r_k+t),v(r_k+t;b,x^*))$
for all $t\ge 0$, and hence $|v(t;\bar\w,\bar x)|\ge\rho$ for all $t\ge 0$,
which contradicts the previously proved property.
Using again $y(t;0,x^*)=v(t;b,x^*)$ we conclude that
$\liminf_{t\to\infty}|y(t;0,x^*)|=0$, as asserted.

Let us now assume that $\sup_{t\ge 0}\int_0^t b(s)\,ds=\infty$ and that $x^*>0$,
and check the existence of $(t_k)\uparrow\infty$ such that $y(t_k;0,x^*)>1$. If, on the contrary,
there exists $\bar t$ such that $y(t;0,x^*)\in(0,1]$ for all $t\ge\bar t$, then
$y(t;0,x^*)=y(\bar t;0,x^*)\,\exp \int_{\bar t}^t b(s)\,ds$ for all
$t\ge\bar t$. So, $\int_{\bar t}^t b(s)\,ds$ is upperly bounded as $t\to\infty$, which
is not the case. For $x^*<0$, the assertion follows from $y(t_k;0,-x^*)=-y(t_k;0,x^*)$.
\end{proof}
Now, we are in conditions to describe the examples. Both of them are focused on the solutions
$\tilde x_\ep(t;0,\hat x)$ of \eqref{eq:4init}; i.e., on the solutions
written with respect to the slow time $t$.
\begin{exa}\label{ex:5fibracont}
This first case is a sample of the simplest situation of continuity of the
fiber map for a local (global, in fact) attractor $\mA$, and
shows the necessity of considering its fibers $\mA_{\hf{\cdot}t}$
as the ``tracking objects'' for $\tilde x_\ep(t;0,\hat x)$
instead of considering just the solution $\tilde z(t;0,\hat x)$ of \eqref{eq:4promt}.
So, it proves the optimality of the thesis of Theorem \ref{teor:4tracking}(ii).

We take an almost periodic map $a$ satisfying
\begin{equation}\label{eq:5noED2}
 \hat a=0
 \qquad\text{and}\qquad \sup_{t\le 0} \int_0^t a(s)\,ds=\sup_{t\ge 0} \int_0^t a(s)\,ds=\infty
\end{equation}
(which in particular ensures \eqref{eq:5noED}) and the map
$h\colon\R\to\R$ of Lemma \ref{lema:5ejemplos}, and define
$f(\tau,t,x):=a(\tau)\,x+h(x)$.  Lemma \ref{lema:5ejemplos}(i) ensures that $f$ satisfies
\ref{f1}, \ref{f2}, \ref{f3} and \ref{f4}, with $\hf(t,x)=h(x)$.
So, equations \eqref{eq:4init} and \eqref{eq:4promt} are
$dx/dt=a(t/\ep)\,x+h(x)$ and $dz/dt=h(z)$, the hull $\Wf$  reduces to the singleton
$\{\hf\}$, and it is easy to check that $\mA:=\Wf\times[-1,1]$ is the global attractor
of the corresponding skewproduct flow. Obviously, the fiber map \eqref{def:4section},
which just sends $\hf$ to $[-1,1]$, is continuous, and hence the conclusions
of Theorem \ref{teor:4tracking}(ii) apply: given any $\hat x\in\R$ and $\sigma>0$,
there exist $t_0>0$ and $\ep_0>0$ such that, if $0<\ep<\ep_0$, then
$\dist_\R\big(\tilde x_{\ep}(t;0,\hat x),\,[-1,1])\le \sigma$
for all $t\ge t_0$. Since $\tilde x(t;0,0)\equiv 0$, any non-zero solution preserves
sign, which allows us to be more
precise: $\dist_{\R^n}\!\big(\tilde x_{\ep}(t;0,\hat x),\,[0,1])\le \sigma$ if
$\hat x>0$ and $\dist_{\R^n}\!\big(\tilde x_{\ep}(t;0,\hat x),\,[-1,0])\le \sigma$
if $\hat x<0$.

Let us check the optimality of this result. We fix $\hat x\in[-1,1]$
and $\ep>0$, and consider the respective solutions $\tilde x_\ep(t;0,\hat x)$
and $\tilde z(t;0,\hat x)=\hat x$ of $dx/dt=a(t/\ep)\,x+h(x)$ and $dz/dt=h(z)$.
It is immediate to check that $a_\ep(t):=a(t/\ep)$ is almost periodic and
satisfies \eqref{eq:5noED2}. Since $\sup_{t\ge 0} \int_0^t a_\ep(s)\,ds=\infty$,
Lemma \ref{lema:5ejemplos}(v) shows that
$\liminf_{t\to\infty}\tilde x_\ep(t;0,\hat x)=0$ and that
there exists a sequence $(t_k^\ep)\uparrow\infty$
such that $\tilde x_\ep(t_k^\ep;0,\hat x)>1$ for all $k\in\N$ if
$\hat x>0$ and $\tilde x_\ep(t_k^\ep;0,\hat x)<-1$ for all $k\in\N$ if
$\hat x<0$. In particular, if $\hat x>0$,
the range of (always positive) values taken by $\tilde x_\ep(t;0,\hat x)$
covers and exceeds the interval $(0,1]$, oscillating
between values approaching 0 and values strictly above 1 (with distance to 1 decreasing
as $\ep$ decreases): this range of values covers $(0,1]$, which is
a significant portion of the fiber of $\mA$.
And the situation is analogous for $\hat x<0$. It is interesting to insist in the fact that,
if $\hat x\in[-1,1]-\{0\}$, it is not true that
$\lim_{\ep\to 0^+}|\tilde x_\ep(t;0,\hat x)-\tilde z(t;0,\hat x)|=
\lim_{\ep\to 0^+}|\tilde x_\ep(t;0,\hat x)-\hat x|=0$ uniformly
at any positive half-line, although, of course, the averaging principle
ensures that the limit stands pointwise for all $t\ge 0$.
\begin{figure}[ht]
 \includegraphics[width=0.95\textwidth]{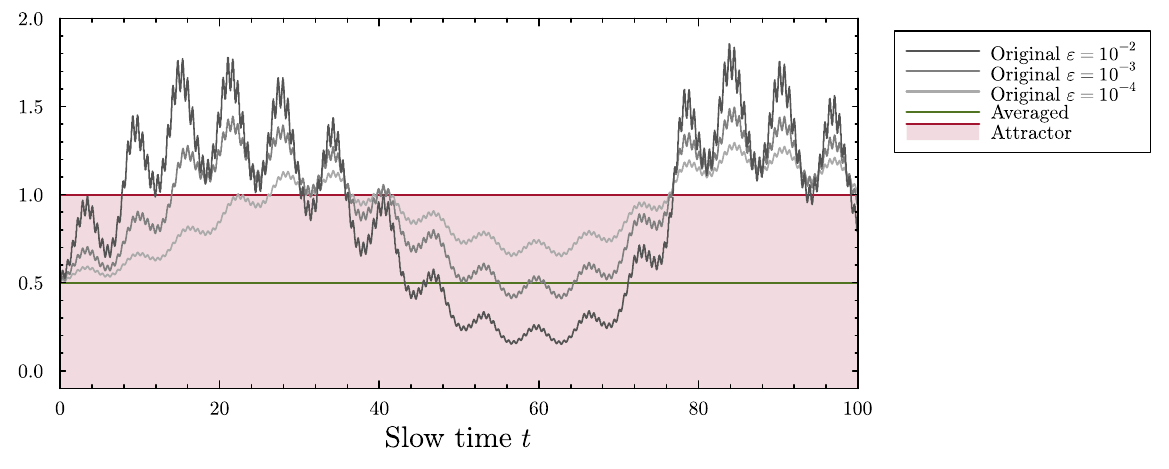}
 \caption{This figure corresponds to Example \ref{ex:5fibracont}.
 In increasingly lighter shades of gray, the solutions $\tilde x_\ep(t;0,1/2)$ of
 $dx/dt=a_{0.5,0.1}(t/\ep)\,x+h(x)$ for the three values $10^{-2},\,10^{-3}$ and $10^{-4}$
 of $\ep$, with $a_{0.5,0.1}$ and $h$ defined in \eqref{def:5aalphabeta} and Lemma
 \ref{lema:5ejemplos}; in green, the solution $1/2$ of $dz/dt=h(z)$; in light pink,
 over each $t$, part of the constant fiber $\mA_{\hf{\cdot}t}=[-1,1]$,
 where $\mA$ is the global attractor for $dz/dt=h(z)$; and in darker pink the
 upper bounded solution of $dz/dt=h(z)$, which determines the upper
 bound of the attractor fibers.}
 \label{fig:5exfibracont}
\end{figure}

In order to partly depict this behavior, we take a map described in \cite{orta} which
is almost periodic and satisfies \eqref{eq:5noED2}: the function
\begin{equation}\label{def:5aalphabeta}
 a_{\alpha,\beta}(t):=\sum_{k=0}^\infty \alpha^k \sin\!\pren{\beta^k t}\,,\qquad
 \text{with $\;0<\beta<\alpha<1$}\,
\end{equation}
is almost periodic and, if
\begin{equation}\label{def:5Aalfabeta}
 A_{\alpha,\beta}(t):=\int_0^ta_{\alpha,\beta}(s)\,ds=2\sum_{k=0}^\infty
 \pren{\frac{\alpha}{\beta}}^k\!\sin^2\!\pren{\frac{\beta^k t}{2}}\,,
\end{equation}
then $A_{\alpha,\beta}(t)\ge m_{\alpha,\beta}|t|^{1-\ln(\alpha)/\ln(\beta)}$
for a constant $m_{\alpha,\beta}>0$ if $|t|>\pi$ (see \cite[Appendix A.2]{orta}).
So, $\lim_{t\to\pm\infty}A_{\alpha,\beta}(t)=\infty$.
These properties mean that $a_{\alpha,\beta}$ satisfies the last two equalities in
\eqref{eq:5noED2}. In order to check also that its average $\hat a_{\alpha,\beta}$ is $0$,
we start again from \eqref{def:5Aalfabeta}. For $|t|\ge 1/\beta$, we define
$k(t)$ as the unique positive integer number in $(-\log(|t|)/\log(\beta)-1,-\log(|t|)/\log(\beta)]$,
observe that $\lim_{|t|\to\infty}k(t)=\infty$
and use $\sin(\theta)\le 1$ and $\sin^2(\theta)\le\theta^2$ to get
\[
 A_{\alpha,\beta}(t)\le 2\sum_{k=0}^{k(t)-1}\pren{\frac{\alpha}{\beta}}^k+
 \frac{t^2}{2} \sum_{k=k(t)}^\infty (\alpha\beta)^k=2\:\frac{\pren{\alpha/\beta}^{k(t)}-1}
 {\pren{\alpha/\beta}-1}+\frac{t^2}{2}\,\frac{\pren{\alpha\beta}^{k(t)}}{1-\alpha\beta}\,.
\]
It is not hard to deduce from here that $A_{\alpha,\beta}(t)\le M_{\alpha,\beta}|t|^{1-\ln(\alpha)/\ln(\beta)}$
for a constant $M_{\alpha,\beta}>0$ if $|t|\ge 1/\beta$. Since
$\ln(\alpha)/\ln(\beta) \in (0,1)$,
$\hat a_{\alpha,\beta}=\lim_{t\to\infty}A_{\alpha,\beta}(t)/t=0$, as asserted.

In Figure~\ref{fig:5exfibracont}, the solutions $\tilde x_\ep(t;0,1/2)$ of $dx/dt=a(t/\ep)\,x+h(x)$
have been integrated numerically for the approximation to the map $a:=a_{0.5,0.1}$ given by the
first 51 terms of \eqref{def:5aalphabeta} and for three values of $\ep$: $10^{-2},\,10^{-3}$ and $10^{-4}$.
The plot window provides a visual clue of the proved existence of
the sequence of points $(t_k^\ep)\uparrow\infty$ with $\tilde x_\ep(t_k^\ep;0,1/2)>1$,
indicating also that the maximum value of each solution decreases as $\ep$ decreases:
the maximum approaches $1$ as $\ep\downarrow 0$, as ensured by Theorem \ref{teor:4tracking}(ii).
However, the space limitation does not allow a good representation of the (also proved) existence
of a sequence $(s_k^\ep)\uparrow\infty$ with $\lim_{k\to\infty}\tilde x_\ep(s_k^\ep;0,1/2)=0$,
which combined with the previous property shows that $\tilde x_\ep$ takes
any value in $(0,1]$ of the fiber of the attractor infinitely many times as time increases.
Finally, it can also be observed that $\tilde x_\ep(t;0,1/2)$
bears little resemblance to $\tilde z(t;0,1/2)=1/2$ as time increases.
Of course, as already mentioned, the averaging principle stands
at finite intervals $[0,c]$: a small enough $\ep_c>0$ ensures that the
distance between $\tilde x_{\ep_c}(t;0,1/2)$ and $1/2$ is as small as desired
for $t\in[0,c]$.
\end{exa}

\begin{exa}\label{ex:5fibranocont2}
This second example shows the optimality of the thesis of Theorem \ref{teor:4tracking}(i):
in the case of a noncontinuous fiber map for the local attractor $\mA$ (which is global in this case),
the tracking objects are not $\mA_{\hf{\cdot}t}$ but $\mA_{\hf{\cdot}t}^{[\delta]}$ for $\delta>0$
as small as desired. It differs from Example \ref{ex:5fibracont} in this respect, and is similar
to it in that $\tilde x_\ep(t;0,\hat x)$ does not track $\tilde z(t;0,\hat x)$ uniformly on $[0,\infty)$
as $\ep\downarrow 0$.

Let us define $f(\tau,t,x):=a_{\bar\alpha,\beta}(\tau)\,x-a_{\alpha,\beta}(t)\,x+h(x)$,
where $h$ is defined in Lemma \ref{lema:5ejemplos} and
the parameters $\bar\alpha,\alpha,\beta$ giving rise to the
maps \eqref{def:5aalphabeta} satisfy $0<\beta<\alpha<\bar\alpha<1$.
Since, as proved
in Example \ref{ex:5fibracont}, $a_{\bar\alpha,\beta}$ satisfies
\eqref{eq:5noED2} (and hence \eqref{eq:5noED}), Lemma \ref{lema:5ejemplos}(i) ensures
$f$ satisfies \ref{f1}, \ref{f2}, \ref{f3} and \ref{f4}, with
$\hf(t,x)=-a_{\alpha,\beta}(t)\,x+h(x)$. Now, $\int_0^t(-a_{\alpha,\beta}(s))\,ds=
-A_{\alpha,\beta}(t)$, with $A_{\alpha,\beta}$ given by \eqref{def:5Aalfabeta}.
As seen in Example \ref{ex:5fibracont}, this
ensures that $\lim_{t\to\pm\infty}\int_0^t(-a_{\alpha,\beta}(s))\,ds=-\infty$.
It follows easily that all the solutions of $dz/dt=-a_{\alpha,\beta}(t)\,z$
tend to 0 as $t$ increases. Since $\pm h(z)\le 0$ for $\pm z>0$, basic comparison
arguments show that also $\tilde z(t;0,\hat x)$ (which solves $dz/dt=-a_{\alpha,\beta}(t)\,z
+h(z)$) tends to 0 as $t$ increases for all $\hat x$.

Note that the elements of the hull $\W_\hf$ are the maps
$(t,x)\mapsto-\ma_{\alpha,\beta}(\wt)\,x+h(x)$, where
$\w$ belongs to the hull $\bar\W$ of $a_{\alpha,\beta}$ and $\ma_{\alpha,\beta}(\w)=\w(0)$.
Lemma \ref{lema:5ejemplos}(ii) shows the existence of
global attractor $\mA$ for the skewproduct flow induced by
$dz/dt=-\ma_{\alpha,\beta}(\wt)\,z+h(z)$ and describes its shape.
Since $-\ma_{\alpha,\beta}$ also satisfies
\eqref{eq:5noED}, the property
$\sup_{t\le 0}\int_0^t(-a_{\alpha,\beta}(s))\,ds<\infty$ and Lemma \ref{lema:5ejemplos}(iii)
ensure that the fiber of $\mA$ over the point $\hf$ is a nondegenerate interval
$[-\muk(\hf),\muk(\hf)]$. Since $t\mapsto \muk(\hf{\cdot}t)$ is a nonzero solution of
$dz/dt=-\ma_{\alpha,\beta}(\wt)\,z+h(z)$,
$\mA_{\hf{\cdot}t}$ is a nondegenerate interval $[-\muk(\hf{\cdot}t),\muk(\hf{\cdot}t)]$
for all $t\ge 0$, and it tends to $\{0\}$ as $t$ increases.
In fact, Lemma \ref{lema:5ejemplos}(iv) guarantees that there exists
a dense subset of points of $\W_\hf$ for which the fiber of $\mA$ is a non-degenerate
interval, and that the fiber map \eqref{def:4section}
is not continuous. In turn, Lemma \ref{lema:5ejemplos}(iii)
ensures that the fibers of $\mA$ over
the points of a residual subset of $\W_\hf$ reduce to $\{0\}$.
Observe the complexity of $\mA$: for each $\delta>0$ and
each $t>0$, $\mA_{\hf{\cdot}t}^{[\delta]}$ contains fibers given by $\{0\}$ corresponding
to base points as close as desired to points with fibers given by non-degenerate intervals, without continuous
transition between them; and the minimality of $\Wf$ ensures that
there exists a sequence $(s_k)\uparrow\infty$ (depending on $\delta$)
such that $\mA_{\hf{\cdot}s_k}^{[\delta]}=[-\bar u,\bar u]$,
where $\bar u$ is the maximum of $\muk$ on $\W_\hf$.
\begin{figure}[ht]
 \includegraphics[width=0.95\textwidth]{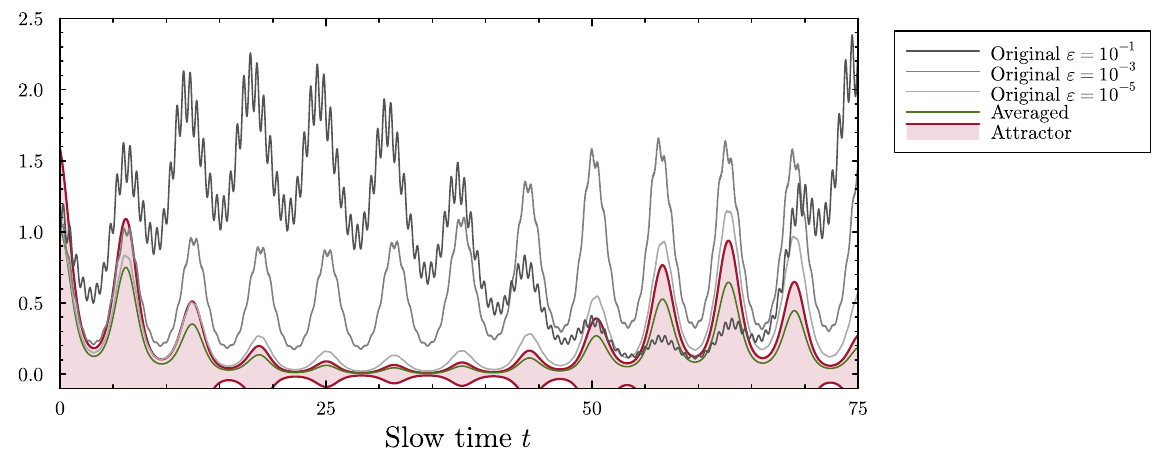}
 \caption{This figure corresponds to Example \ref{ex:5fibranocont2}.
 In increasingly lighter shades of gray, the solutions $\tilde x_\ep(t;0,1)$ of
 $dx/dt=(a_{0.6,0.1}(t/\ep)-a_{0.15,0.1}(t))\,x+h(x)$ for the values $10^{-1},\,10^{-3}$ and $10^{-5}$
 of $\ep$,
 with $a_{\alpha,\beta}$ defined in \eqref{def:5aalphabeta} and $h$ in Lemma
 \ref{lema:5ejemplos}; in light pink,
 over each $t$, the fiber $\mA_{\hf{\cdot}t}$ of the global attractor $\mA$ of
 $dz/dt=-a_{0.15,0.1}(t)\,z+h(z)$; and, in green and red, the solution
 $\tilde z(t;0,1)$ of the same equation and its upper and lower bounded solutions.}
 \label{fig:5exfibranocont2}
\end{figure}

Let us now fix $\ep>0$ and analyze equation \eqref{eq:4init}, i.e., $dx/dt=c_\ep(t)\,x+h(x)$
for $c_\ep(t):=a_{\bar\alpha,\beta}(t/\ep)-a_{\alpha,\beta}(t)$. It is clear that
$c_\ep$ is an almost periodic function with average $\hat c_\ep=0$.
We will check below that it also satisfies the second and third equalities in
\eqref{eq:5noED2}, and hence \eqref{eq:5noED}. So, according to Lemma \ref{lema:5ejemplos}(v),
if $\hat x>0$ then $\liminf_{t\to\infty}\tilde x_\ep(t;0,\hat x)=0$ and there exists
$(t_k)\uparrow\infty$ (depending on $\ep$)
such that $\tilde x_\ep(t_k;0,\hat x)>1$ for all $k\in\N$. This fact combined with
$\lim_{t\to\infty}[-\muk(\hf{\cdot}t),\muk(\hf{\cdot}t)]=\{0\}$ ensures that
$\tilde x_\ep(t;0,\hat x)$ does not track $\mA_{\hf{\cdot}t}$: the lack of continuity of the fiber map
\eqref{def:4section} causes the thesis of Theorem \ref{teor:4tracking}(ii) to fail.
Figure~\ref{fig:5exfibranocont2}
(done for $\bar\alpha=0.6$, $\alpha=0.15$ and $\beta=0.1$) illustrates this lack of tracking:
for each $t$, it depicts the
nondegenerate fiber $[-\muk(\hf{\cdot}t),\muk(\hf{\cdot}t)]$, whose limit as $t$ increases
is $\{0\}$, and which of course varies continuously with respect to $t$. It also depicts $\tilde x_\ep(t;0,1)$ for
three values of $\ep$: $10^{-1},\,10^{-3}$ and $10^{-5}$.
As Theorem \ref{teor:4tracking}(i) ensures, $\tilde x_\ep(t;0,\hat x)$ tracks
$\mA_{\hf{\cdot}t}^{[\delta]}$ for $\delta>0$ as small as desired, and
Figure~\ref{fig:5exfibranocont2} (as well as Figure~\ref{fig:5exfibranocont2tiempolargo}, described below)
can also be understood as a graphical sample of the complexity of
these inflated fibers.

Figure \ref{fig:5exfibranocont2} also depicts $\tilde z(t;0,1)$.
As seen above, for any $\hat x$, $\lim_{t\to\infty}\tilde z(t;0,\hat x)=0$. Lemma \ref{lema:5ejemplos}(ii)
ensures that $\tilde x_\ep(t;0,\hat x)$ is bounded for $t\ge 0$. The averaging result for
double nonautonomous equations established in \cite[Proposition 2.15 and Theorem 2.17]{nuor1}
shows that the maximum distance between $\tilde x_\ep(t;0,\hat x)$ and
$\tilde z(t;0,\hat x)$ on any fixed interval $[0,c]$ uniformly tends to $0$ as $\ep\downarrow 0$,
while the previously observed properties ensure that this is not the case on $[0,\infty)$.
Figure~\ref{fig:5exfibranocont2} shows the ``good behavior'' of the approach of
$\tilde x_\ep$ to $\tilde z$ in a finite interval of time, while
Figure~\ref{fig:5exfibranocont2tiempolargo},
which depicts the
graphs of $\tilde z(t;0,1)$ and $\tilde x_\ep(t;0,1)$ for $\ep=10^{-5}$, shows that
$\tilde x_\ep(t;0,1)$ does not track $\tilde z(t;0,1)$ uniformly on $[0,\infty)$.
It also illustrates that the set of values $t$ for which
$\tilde x_\ep(t;0,1)$ is close to $0$ (and hence close to $\mA_{\hf{\cdot}t}$ if $t$ is large enough)
is large, as Proposition \ref{prop:4density} establishes.

Note finally that $1$ belongs to $\mA_{\hf}$
(which can be numerically checked), and hence $\tilde z(t;0,1)\in\mA_{\hf{\cdot} t}$
for all $t\ge0$. Choosing a larger $\hat x$, say $\hat x=2$ would give rise to a similar behavior for the
solutions $\tilde x_\ep(t;0,\hat x)$ and to a map $\tilde z(t;0,\hat x)$ which is always
above the upper bound $\muk_{\hf{\cdot}t}$ of the attractor fiber,
and which eventually approaches $0$.
\begin{figure}[ht]
 \includegraphics[width=0.95\textwidth]{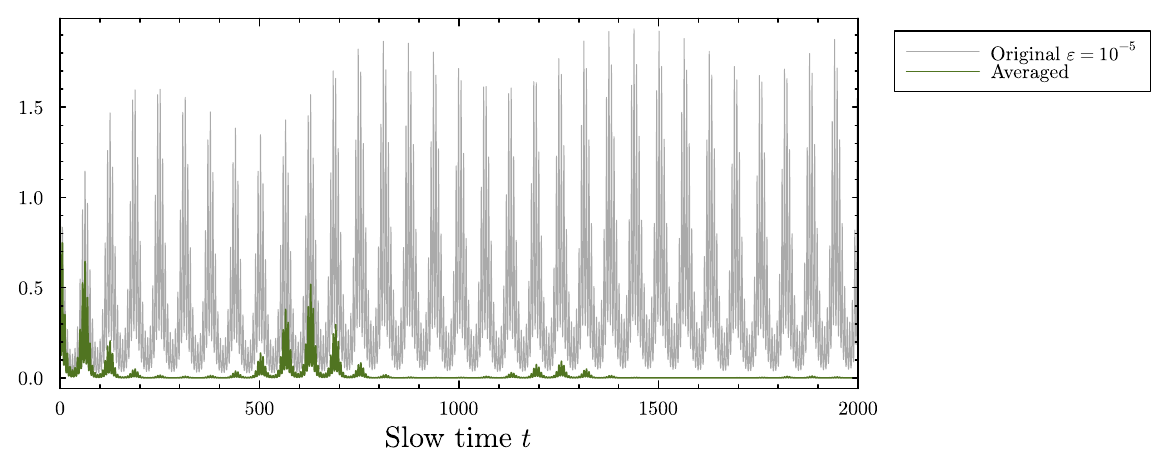}
 \caption{This figure also corresponds to Example \ref{ex:5fibranocont2}.
 In gray, the solution $\tilde x_\ep(t;0,1)$ of the equation
 $dx/dt=\big(a_{0.6,0.1}(t/\ep)-a_{0.15,0.1}(t)\big)\,x+h(x)$ for $\ep=10^{-5}$,
  with $a_{\alpha,\beta}$ defined in \eqref{def:5aalphabeta} and $h$ in Lemma
 \ref{lema:5ejemplos}; and, in green, the solution
 $\tilde z(t;0,1)$ of $dz/dt=-a_{0.15,0.1}(t)\,x+h(x)$.}
 \label{fig:5exfibranocont2tiempolargo}
\end{figure}

It remains to check the property on the primitive of $c_\ep$ that we used before.
As said after \eqref{def:5Aalfabeta},
$m_{\alpha,\beta}|t|^{1-\ln(\alpha)/\ln(\beta)}\le A_{\alpha,\beta}(t)\le
M_{\alpha,\beta}|t|^{1-\ln(\alpha)/\ln(\beta)}$ if $|t|$ is large enough, for
$M_{\alpha,\beta}>m_{\alpha,\beta}>0$.
So,
\[
 \ep\,A_{\bar\alpha,\beta}(t/\ep)-A_{\alpha,\beta}(t)\ge
 \ep^{\ln(\bar\alpha)/\ln(\beta)}\,m_{\bar\alpha,\beta}|t|^{1-\ln(\bar\alpha)/\ln(\beta)}-
 M_{\bar\alpha,\beta}|t|^{1-\ln(\alpha)/\ln(\beta)}
\]
if $|t|$ is large enough. Since $1-\ln(\bar\alpha)/\ln(\beta)>1-\ln(\alpha)/\ln(\beta)$,
$\lim_{|t|\to\infty}\int_0^t c_\ep(s)\,ds=\infty$, and
this means that $c_\ep$ satisfies the last conditions in \eqref{eq:5noED2},
as asserted.
\end{exa}

\begin{nota}\label{rm:5.meager}
Let $(\W,\sigma)$ be a uniquely ergodic flow on a compact metric space which is
almost periodic (and hence minimal) but not periodic, let $m$ be the unique
$\sigma$-invariant measure, and let $C_0(\W,\R)$ be the set of continuous maps
$\mb\colon\W\to\R$ with $\int_\W\mb(\w)\,dm=0$. Let $h$ be the map defined in Lemma \ref{lema:5ejemplos},
and take $\mb\in C_0(\W,\R)$. The same arguments leading to the proof of Lemma \ref{lema:5ejemplos}(ii)
show the existence of a global attractor $\bar \mA_{\mb}$ for the skewproduct flow $\pi_{\mb}$
induced on $\W\times\R$ by the family $x'=\mb(\wt)\,x+h(x)$.
This set can be identified with the global attractor of the flow defined from the
equation \eqref{eq:4promt} given by the map $\hf(t,z):=\mb(\wt)\,z+h(z)$,
where $\w$ is any fixed point of $\W$.
Combining topological and measure-theoretical arguments,
it is possible to prove that
\begin{itemize}[leftmargin=15pt]
\item[-] the subset of maps $\mb$ for which the fiber map is continuous
contains all the elements of $C_0(\W,\R)$ which have a bounded primitive (i.e.,
a map $\pb\in C(\W,\R)$ such that $\int_0^t\mb(\ws)\,ds=\pb(\wt)-\pb(\w)$),
which is a dense proper subset of $C_0(\W,\R)$;
\item[-] the elements $C_0(\W,\R)$ without a bounded primitive contains a proper
subset $\mN$ of maps $\mb$ for which the fiber map of $\bar\mA_{\mb}$ is continuous at
a residual subset with full measure $m$, which due to the ergodic uniqueness means with
complete measure; and this subset $\mN$ is residual in $C_0(\W,\R)$.
\end{itemize}
In this way, we have described a framework on which there is ``abundance" of maps $f$
satisfying the conditions of Proposition \ref{prop:4density}.
\end{nota}
\section{Tracking results for a single process}\label{sec:6}
Throughout this section, we will  work with
$f\colon\R_+\!\times\R_+\!\times\R^n\to\R^n$ under conditions
\ref{f1}, \ref{f2}, \ref{f3} and \ref{f4}.
As in Section \ref{sec:4}, $x_\ep(\tau;\tau^*,x^*)$ and
$z_\ep(\tau;\tau^*,x^*)$ represent the maximal (now for $\tau\ge\tau^*$) solutions
of \eqref{eq:4initau} and \eqref{eq:4promtau} with value $x^*$ at $\tau=\tau^*$,
and $\tilde x_\ep(t;t^*,x^*)$ and $\tilde z(t;t^*,x^*)$ play the same role for
\eqref{eq:4init} and \eqref{eq:4promt}.

In this section, the hull construction is not required: the hypotheses and theses
are formulated in terms of the solutions of \eqref{eq:4init} and \eqref{eq:4promt},
and of some dynamical elements of the process induced by  \eqref{eq:4promt}.
The two final examples will show connections and differences between both formulations.

The main hypothesis of the first theorem is the existence of a uniform local attractor for the
process determined by the averaged equation. The result proves that the corresponding
fibers are tracked by the solutions of
\eqref{eq:4init} starting close enough if $\ep$ is sufficiently small and
$t$ is greater than a certain fixed time.

Given $\mC\subset\R^n$ and $t^*,t\in\R_+$ such that $\tilde z(t;t^*,x^*)$ exists
for all $x^*\in\mC$, we
represent by $\tilde z(t;t^*,\mC):=\{\tilde z(t;t^*,x^*)\,|\;x^*\in\mC\}$.
\begin{defi}\label{def:6attractor}
A family $\{\mK_t\,|\;t\ge 0\}$ of nonempty compact subsets of $\R^n$ is a
{\em uniform local attractor for \eqref{eq:4promt}} if
\begin{itemize}[leftmargin=25pt]
\item[-] the family is {\em invariant}, i.e.,
$\tilde z(t;t^*,\mK_{t^*})=\mK_t$ for all $t^*\ge 0$ and $t\ge 0$\,;
\item[-] there exists $\mu>0$ such that
for all $t^*\ge 0$ and $x^*\in\bar\mB_\mu(\mK_{t^*})$,
$\tilde z(t;t^*,x^*)$ exists for all $t\ge t_*$, and
\begin{equation}\label{eq:6equiatr}
 \lim_{t\to\infty}\dist_{\R^n}\!\big(\tilde z(t+t^*;t^*,\bar\mB_\mu(\mK_{t^*})),\,\mK_{t+t^*}\big)=0
\end{equation}
uniformly for $t^*\ge 0$.
\end{itemize}
\end{defi}
\begin{teor}\label{teor:6trackAtractor}
Let $f\colon\R_+\!\times\R_+\!\times\R^n\to\R^n$ satisfy \ref{f1}, \ref{f2},
\ref{f3} and \ref{f4}. Assume also the existence of a
uniform local attractor $\{\mK_t\,|\;t\ge 0\}$
for $dz/dt=\hf(t,z)$ such that there exists $k>0$ with
$\text{{\Large$\cup$}}_{t\ge 0}\,\mK_t\subset\mB_k$.
Take $\hat x\in\R^n$ with $\dist_{\R_n}\big(\hat x,\,\mK_0\big)\le\mu$,
where $\mu$ is the constant of Definition {\rm\ref{def:6attractor}}.
Then, for all $\sigma>0$, there exist $t_\sigma>0$ and $\ep_\sigma>0$ such that,
if $0<\ep<\ep_\sigma$, then $\tilde x_\ep(t;0,\hat x)$ exists for all $t>0$ and
satisfies
\begin{equation}\label{eq:6tesisAtractor}
 \dist_{\R^n}\!\big(\tilde x_\ep(t;0,\hat x),\,\mK_t\big)\le\sigma
 \qquad \text{for all $t\ge t_\sigma$}\,.
\end{equation}
Or, equivalently, $x_\ep(\tau;0,\hat x)$ exists for all $\tau>0$ and
satisfies
\[
 \dist_{\R^n}\!\big(x_\ep(\tau;0,\hat x),\,\mK_{\ep\tau}\big)\le\sigma
 \qquad \text{for all $\tau\ge t_\sigma/\ep$}\,.
\]
\end{teor}
\begin{proof}
The equivalence of the statement is an easy consequence of \eqref{eq:4relacion}.
Let $\mu>0$ be the constant of Definition \ref{def:6attractor}, and fix $\sigma\in(0,\mu]$.
The limit \eqref{eq:6equiatr} provides $t_\sigma>0$ such that
$\dist_{\R^n}\!\big(\tilde z(t+t^*;t^*,\bar\mB_\mu(\mK_{t^*})),\,\mK_{t+t^*})<\sigma/2$
for all $t^*\ge 0$ if $t\ge t_\sigma$.
We will find $\ep_\sigma>0$, and check by an
induction argument that, if $\ep\in(0,\ep_\sigma)$, then
\[
\tag{$\mP_m^*$}
 \dist_{\R^n}\!\big(\tilde x_\ep(t;0,\hat x),\,\mK_t\big)\le\sigma
 \quad \text{if $\;t\in[t_\sigma,\,(m+1)\,t_\sigma]$}
\]
for all $m\ge 1$, which clearly proves \eqref{eq:6tesisAtractor}.

We assume without restriction that $|\tilde z(t;0,\hat x)|\le k$ for
all $t\in[0,t_\sigma]$, where $k>0$ is the constant of the statement, and call $r:=k+3\mu/2$.
As in the proof of Theorem \ref{teor:4tracking}, we look for a
$C^1$-map $h\colon\R^n\to[0,1]$ with value $1$ on $\bar\mB_r$ and $0$ outside $\mB_{2r}$,
define $g\colon\R_+\!\times\R_+\!\times\R^n$ by $g(\tau,t,x):=h(x)\,f(\tau,t,x)$,
check that $g$ satisfies the hypotheses of Theorem \ref{teor:4acotada}
with $\hat g(t,x)=h(x)\,\hat f(t,x)$, and observe that
the solutions $\bar x_\ep(t;t^*,x^*)$ and $\bar z(t;t^*,x^*)$
with value $x^*\in\bar\mB_r$ at $t^*$ of
\begin{equation}\label{eq:6porht-2}
 \frac{dx}{dt}=g(t/\ep,t,x)
 \qquad\text{and}\qquad\frac{dz}{dt}=\hat g(t,z)\,,
\end{equation}
which are defined on $\R_+$, coincide with $\tilde x_\ep(t;t^*,x^*)$ and
$\tilde z(t;t^*,x^*)$ as long as they remain in $\bar\mB_r$.
Let $D_{2t_\sigma,r}(\ep)$ be defined by \eqref{def:4Deltat} for these
equations. We apply Theorem \ref{teor:4acotada} to choose $\ep_\sigma>0$
small enough to ensure that $D_{2t_\sigma,r}(\ep)<\sigma/2$
whenever $0<\ep<\ep_{\sigma}$. From now on, we fix $\ep\in(0,\ep_\sigma)$.

Let us check $(\mP_1^*)$ for $\ep$. Since $\hat x\in\bar\mB_\mu(\mK_0)$,
$\dist_{\R^n}\!\big(\tilde z(t;0,\hat x),\,\mK_t\big)<\sigma/2$
if $t\ge t_\sigma$. In particular, $|\tilde z(t;0,\hat x)|
< k+\sigma/2<r$ if $t\ge 0$, so $\tilde z(t;0,\hat x)=
\bar z(t;0,\hat x)$ if $t\ge 0$. On the other hand,
the choice of $\ep_\sigma$ ensures that
$|\bar x_\ep(t;0,\hat x)-\bar z(t;0,\hat x)|<\sigma/2$
if $t\in[0,\,2t_\sigma]$, which yields $|\bar x_\ep(t;0,\hat x)|<
|\bar z(t;0,\hat x)|+\sigma/2\le k+\mu<r$ and hence that
$\bar x_\ep(t;0,\hat x)=\tilde x_\ep(t;0,\hat x)$
if $t\in[0,\,2t_\sigma]$. Altogether, we get
$\dist_{\R^n}\!\big(\tilde x_\ep(t;0,\hat x),\,\mK_t\big)\le
|\bar x_\ep(t;0,\hat x)-\bar z(t;0,\hat x)|+
\dist_{\R^n}\!\big(\tilde z(t;0,\hat x),\,\mK_t\big)<\sigma/2+\sigma/2$
if $t\in[t_\sigma,\,2t_\sigma]$, and $(\mP_1^*)$ is proved.

Now, we assume $(\mP_m^*)$ for an $m\ge 1$ in order to check $(\mP_{m+1}^*)$.
Property $(\mP_m^*)$ ensures that
$\dist_{\R^n}\!\big(\tilde x_\ep(t;0,\hat x),\,\mK_t\big)\le \mu$
for all $t\in[t_\sigma, (m+1)\,t_\sigma]$. This implies, on the one hand, that
$\tilde x_\ep(mt_\sigma;0,\hat x))\in\bar\mB_\mu(\mK_{mt_\sigma})$; and, on the
other hand, that $|\tilde x_\ep(t;0,\hat x)|\le k+\mu<r$
for $t\in [t_\sigma, (m+1)\,t_\sigma]$. In the proof of $(\mP_1^*)$, we have
also checked the last property for $t\in[0,t_\sigma]$, and hence
$\tilde x_\ep(t;0,\hat x)=\bar x_\ep(t;0,\hat x)$ for $t\in[0,\,(m+1)\,t_\sigma]$,
which in turn yields
\begin{equation}\label{eq:6x}
 \bar x_\ep(t;0,\hat x)=\bar x_\ep(t;mt_\sigma,\tilde x_\ep(mt_\sigma;0,\hat x))=
 \tilde x_\ep(t;mt_\sigma,\tilde x_\ep(mt_\sigma;0,\hat x))
\end{equation}
for $t\in[mt_\sigma,\,(m+1)t_\sigma]$.
These properties, the choice of $\ep_\sigma$ and $(\mP_m^*)$ ensure that
\[
\begin{split}
 &\dist_{\R^n}\!(\bar z(t;mt_\sigma,\tilde x_\ep(mt_\sigma;0,\hat x)),\,\mK_t\big)\\
 &\qquad\qquad \le|\bar z(t;mt_\sigma,\tilde x_\ep(mt_\sigma;0,\hat x))
 -\bar x_\ep(t;mt_\sigma,\tilde x_\ep(mt_\sigma;0,\hat x))|\\
 &\qquad\qquad \quad +\dist_{\R^n}\!\big(\tilde x_\ep(t;0,\hat x),\,\mK_t\big)\le \sigma/2+\sigma\le 3\mu/2
\end{split}
\]
if $t\in[mt_\sigma,\,(m+1)t_\sigma]$, which implies that
$|\bar z(t;mt_\sigma,\tilde x_\ep(mt_\sigma;0,\hat x))|\le k+3\mu/2=r$
if $t\in[mt_\sigma,\,(m+1)t_\sigma]$; and, since
$\tilde x_\ep(mt_\sigma;0,\hat x))\in\bar\mB_\mu(\mK_{mt_\sigma})$,
the choice of $t_\sigma$ yields
\begin{equation}\label{des:6fin}
 \dist_{\R^n}\!\big(\tilde z(t;mt_\sigma,\tilde x_\ep(mt_\sigma;0,\hat x)),\,\mK_t\big)
 <\sigma/2
\end{equation}
if $t\ge (m+1)\,t_\sigma$, which implies that
$|\tilde z(t;mt_\sigma,\tilde x_\ep(mt_\sigma;0,\hat x))|\le k+\mu/2<r$
if $t\ge (m+1)\,t_\sigma $. Altogether,
$\tilde z(t;mt_\sigma,\tilde x_\ep(mt_\sigma;0,\hat x))=
\bar z(t;mt_\sigma,\tilde x_\ep(mt_\sigma;0,\hat x))$ for
$t\ge mt_\sigma$. Combining this property, \eqref{des:6fin}, and the
choice of $\ep_\sigma$, the triangular inequality yields
\begin{equation}\label{des:6fin2}
\dist_{\R^n}\!\big(\bar x_\ep(t;mt_\sigma,\tilde x_\ep(mt_\sigma;0,\hat x)),\,\mK_t\big)\le \sigma\,,
\end{equation}
for $t\in[(m+1)\,t_\sigma,\,(m+2)\,t_\sigma]$, and \eqref{eq:6x} and \eqref{des:6fin2} yield
$\tilde x_\ep(t;mt_\sigma,\tilde x_\ep(mt_\sigma;0,\hat x))=
\bar x_\ep(t;mt_\sigma,\tilde x_\ep(mt_\sigma;0,\hat x))$ also
for $t\in[(m+1)\,t_\sigma,\,(m+2)\,t_\sigma]$.
This, \eqref{des:6fin2} and $(\mP_m^*)$ prove $(\mP_{m+1}^*)$, and complete
the proof of the theorem.
\end{proof}

Our second result is a consequence of the first one and constitutes an extension to our ``doubly''
nonautonomous setting of the results of Eckhaus and S\'{a}nchez-Palencia for $dx/d\tau=\ep\,g(\tau,x)$:
see \cite[Theorem 5.5.1]{savm}.
It shows that a uniformly asymptotically stable solution of the averaged equation \eqref{eq:4promt}
is tracked uniformly on the whole positive half-line by the solution with the same initial data of
\eqref{eq:4init} if $\ep$ is small.

\begin{defi}\label{def:6UAE}
Let $t^*\ge 0$. A solution $\tilde z(t;t^*,x^*)$ of \eqref{eq:4promt} is
{\em uniformly asymptotically stable}
if it is defined on $[t^*,\infty)$ and there exists
$\mu>0$ such that, for all $\sigma\le\mu$, there exists $s_\sigma\ge0$ such that,
whenever $s^*\ge t^*$ and $|\tilde z(s^*;t^*,x^*)-z^*|\le\mu$, the solution
$\tilde z(t;s^*,z^*)$ is defined for all $t\ge s^*$ and satisfies
$|\tilde z(t;t^*,x^*)-\tilde z(t;s^*,z^*)|\le\sigma$ for all $t\ge s^*+s_\sigma$.
%
\end{defi}
\begin{coro}\label{coro:6trackSol}
Let $f\colon\R_+\!\times\R_+\!\times\R^n\to\R^n$ satisfy \ref{f1}, \ref{f2},
\ref{f3} and \ref{f4}. Take $\hat x\in\R^n$ and assume that
$\tilde z(t;0,\hat x)$ is a uniformly asymptotically stable solution of \eqref{eq:4promt}
and that $\{\tilde z(t;0,\hat x)\,|\:t\ge 0\}$ is bounded. Then, for all
$\sigma>0$, there exists $\ep_\sigma>0$ such that, if $0<\ep<\ep_\sigma$, then
$\tilde x_\ep(t;0,\hat x)$ exists for all $t>0$ and satisfies
\begin{equation}\label{eq:6tesis}
 |\tilde x_\ep(t;0,\hat x)-\tilde z(t;0,\hat x)|\le\sigma
 \qquad \text{for all $t\ge 0$}\,.
\end{equation}
Or, equivalently, $x_\ep(\tau;0,\hat x)$ exists for all $\tau>0$ and satisfies
\[
 |x_\ep(\tau;0,\hat x)-z_\ep(\tau;0,\hat x)|\le\sigma
 \qquad \text{for all $\tau\ge 0$}\,.
\]
\end{coro}
\begin{proof}
Again, \eqref{eq:4relacion} shows the equivalent of the statement, so that it suffices
to prove \eqref{eq:6tesis}. It is easy to deduce from the uniform asymptotic stability of
the solution $\tilde z(t;0,\hat x)$ of
$dz/dt=\hf(t,z)$ that the sets $\mK_t:=\{\tilde z(t;0,\hat x)\}$
give rise to a uniform local attractor. So, given $\sigma>0$,
Theorem \ref{teor:6trackAtractor} provides $t_\sigma>0$ and $\bar\ep_\sigma>0$ such that,
if $0<\ep<\bar\ep_\sigma$, then $\tilde x_\ep(t;0,\hat x)$ exists for all $t>0$ and
satisfies $|\tilde x_\ep(t;0,\hat x)-\tilde z(t;0,\hat x)|\le\sigma$ for all $t\ge t_\sigma$.
In these conditions, \cite[Proposition 2.15 and Theorem 2.17]{nuor1} provide
$\ep_\sigma\in(0,\bar\ep_\sigma]$ such
that, if $0<\ep<\ep_\sigma$, then the inequality also holds in $[0,t_\sigma]$.
\end{proof}

As in the case of a local attractor for the skewproduct flow $\pi$
(see Remark \ref{rm:3existen}), the existence of a uniform local attractor
or a uniformly asymptotically stable solution for $dz/dt=\hf(t,z)$
is not {\em a priori\/} guaranteed, but there are many cases on which these objects
actually occur. We complete this paper by describing two examples of it.
\begin{exa}\label{ex:6parAR}
Example \ref{ex:4parAR} can be also understood as a simple sample of
the tracking behavior described in Corollary \ref{coro:6trackSol}.
By reviewing the definition of attractive hyperbolic solution, it is easy to conclude
that, if $\hat x>r_{1/2}(0)$, then $\tilde z(t;0,\hat x)$
is a uniformly asymptotically stable solution of the equation $dz/dt=1/2+p_2(t)-z^2$.
So, Corollary \ref{coro:6trackSol} states
that the solution $\tilde x_\ep(t;0,\hat x)$ ``stays as close as desired'' to $\tilde z(t;0,\hat x)$ on
$[0,\infty)$ if $\ep>0$ is small enough. This
behavior is reflected in Figure \ref{fig:4exparAR}.
We can also use this example as a simple illustration of Theorem \ref{teor:6trackAtractor}:
as said in the proof of Corollary \ref{coro:6trackSol}, $\big\{\{a_{1/2}(t)\}\,|\;t\ge 0\big\}$
is a uniform local attractor.
Theorem \ref{teor:6trackAtractor} states that $\tilde x_\ep(t;0,\hat x)$
is as close as desired to $a_{1/2}(t)$ for $t\ge t_\ep$ (but, clearly, not for all
$t\ge 0$) if $\hat x$ is in the uniform domain of attraction of the hyperbolic solution and $\ep$ is small
enough,
being the distance smaller and the time $t_\ep$ larger as $\ep$ decreases. This behavior
can also be observed in Figure \ref{fig:4exparAR}.
\end{exa}

\begin{exa}\label{ex:6noflujo}
The following example of application of Theorem \ref{teor:6trackAtractor} and Corollary \ref{coro:6trackSol}
is at the same time a case to which Theorem \ref{teor:4tracking} cannot be applied.

Let us define $c(t)$ as the piecewise linear map with value $1$ on
$[1,\infty)$ and $-1$ on $(-\infty,0]$. Then, its hull $\W_c$ is given by
$\{-\bar 1\}\cup\{c_s\,|\;s\in\R\}\cup\{+\bar 1\}$, where $c_s(t):=c(s+t)$ and $\pm \bar 1$
are the constant maps with values $\pm 1$: see Remark \ref{nota:3union}. Let $a\colon\R\to\R$
be an almost periodic map with average $\hat a=0$,
and let us consider the equations $dx/dt=a(t/\ep)+c(t)-x^2$ and $dz/dt=c(t)-z^2$. That is,
equations \eqref{eq:4init} and \eqref{eq:4promt} for
$f(\tau,t,x):=a(\tau)+c(t)-x^2$ and (hence) $\hf(t,z)=c(t)-z^2$.
Reasoning as in Lemma \ref{lema:5ejemplos}(i), we conclude that
$f$ satisfies conditions \ref{f1}, \ref{f2}, \ref{f3} and \ref{f4}.
Observe that the hull of $\hat f$
can be identified with $\W_c$ and that the family $dz/dt=\mc(\wt)-z^2$ for $\w\in\W_c$,
with $\mc(\w):=\w(0)$, defines the local skewproduct \eqref{def:3pi} on $\W_c\times\R$.
It is easy to check that the left border of the maximal interval of definition of
any solution $\tilde z(t;0,\hat x)$ of $dz/dt=c(t)-z^2$ is above $-\pi$.
So, this equation has no globally defined solutions, and this precludes the
existence of a local attractor $\mA$ projecting onto $\W_c$: see
Proposition~\ref{prop:3semicont}(ii).
\begin{figure}[ht]
 \includegraphics[width=0.95\textwidth]{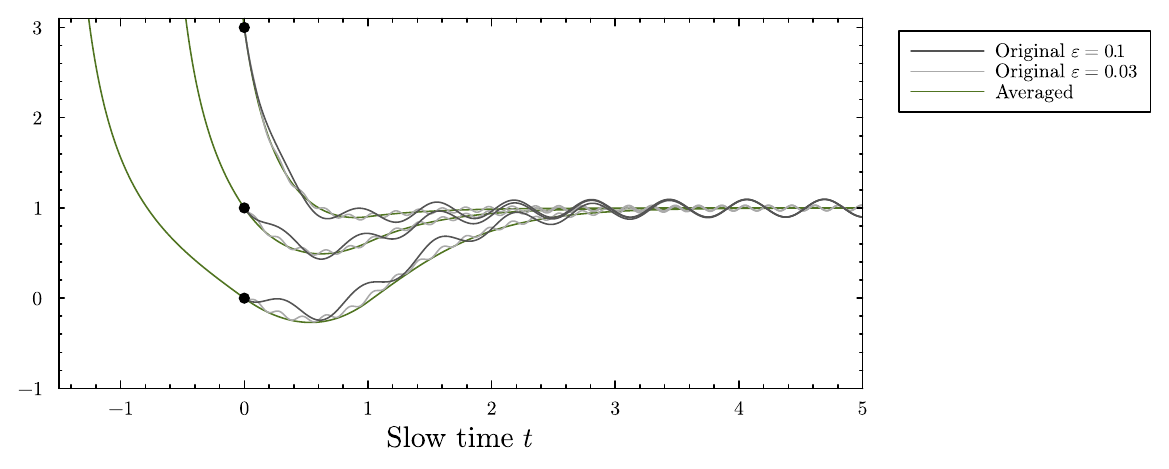}
 \caption{This figure corresponds to Example \ref{ex:6noflujo}; more precisely, to
 equations $dx/dt=\sin(t/\ep)+c(t)-x^2$ and $dz/dt=c(t)-z^2$.
 In green, $\tilde z(t;0,\hat x)$ for three different values of $\hat x$ above $-1$; and, for each one
 of these values, in darker and lighter shades of gray, $\tilde x_\ep(t;0,\hat x)$ for
 $\ep=0.1$ and $\ep=0.03$. }
 \label{fig:6noflujo}
\end{figure}

On the other hand, $1$ is an attractive hyperbolic critical point of $dz/dt=1-z^2$,
$\tilde z(t;1,1)=1$ for all $t\ge 0$, and
$\lim_{t\to\infty}(\tilde z(t;1,\hat x)-\tilde z(t;1,1))=
\lim_{t\to\infty}(\tilde z(t;1,\hat x)-1)=0$ for all $\hat x>-1$.
It is not hard to deduce from here that $\tilde z(t;1,\hat x)$
is a uniformly asymptotically stable solution of $dz/dt=c(t)-z^2$
for all $\hat x>-1$, and hence so is $\tilde z(t;0,\hat x)$ whenever
$\tilde z(1;0,\hat x)>-1$.

Figure \ref{fig:6noflujo} depicts the tracking behavior described in Corollary \ref{coro:6trackSol}
for $a(\tau):=\sin(\tau)$, for three different choices of $\hat x$ above $-1$, and for
$\ep=0.1$ and $\ep=0.03$. It shows that, after a certain period of time, the graphics are
indistinguishable despite the different initial value $\hat x$, which is due to the
strong attractiveness property of $1$ for $t\ge 1$.

\end{exa}

\bibliographystyle{siam}
\bibliography{TrackingBib}

@PREAMBLE{
 "\providecommand{\noopsort}[1]{}"
 # "\providecommand{\singleletter}[1]{#1}%"
}

@book{arnol,
  author    = {Arnold, V. I.},
  title     = {Geometrical Methods in the Theory of Ordinary Differential Equations},
  edition   = {second},
  publisher = {Springer},
  year      = {1988}
}

@article{arst,
  author  = {Artstein, Z.},
  title   = {Averaging of time-varying differential equations revisited},
  journal = {J. Differential Equations},
  volume  = {243},
  number  = {2},
  year    = {2007},
  pages   = {146--167}
}

@article{aspw,
  author  = {Ashwin, P. and Perryman, C. and Wieczorek, S.},
  title   = {Parameter shifts for nonautonomous systems in low dimension: bifurcation and rate-induced tipping},
  journal = {Nonlinearity},
  volume  = {30},
  number  = {6},
  year    = {2017},
  pages   = {2185--2210}
}

@book{aufr,
  author    = {Aubin, J.-P. and Frankowska, H.},
  title     = {Set-Valued Analysis},
  publisher = {Birkh{\"a}user},
  year      = {1990}
}

@book{bege,
  author    = {Berglund, N. and Gentz, B.},
  title     = {Noise-induced phenomena in slow-fast dynamical systems. A sample-paths approach},
  series    = {Probab. Appl.},
  publisher = {Springer},
  year      = {2006}
}

@book{bomi,
  author    = {Bogoliubov, N. N. and Mitropolsky, Y. A.},
  title     = {Asymptotic methods in the theory of non-linear oscillations},
  publisher = {Gordon and Breach},
  year      = {1961}
}

@article{cano,
  author  = {Campos, J. and N{\'u}{\~n}ez, C. and Obaya, R.},
  title   = {Uniform stability and chaotic dynamics in nonhomogeneous linear dissipative scalar ordinary differential equations},
  journal = {J. Differential Equations},
  volume  = {361},
  year    = {2023},
  pages   = {248--287}
}

@book{calr,
  author    = {Carvalho, A. and Langa, J. and Robinson, J.},
  title     = {Attractors for infinite-dimensional non-autonomous dynamical systems},
  series    = {Appl. Math. Sci.},
  volume    = {182},
  publisher = {Springer},
  year      = {2013}
}

@book{dyfi,
  author    = {Dymnikov, V. P. and Filatov, A. N.},
  title     = {Mathematics of Climate Modeling},
  series    = {Modeling and Simulation in Science, Engineering and Technology},
  publisher = {Birkh{\"a}user},
  year      = {1997}
}

@phdthesis{duen,
  author = {Due{\~n}as, J.},
  title  = {D-concave nonautonomous differential equations and applications to critical transitions},
  school = {Universidad de Valladolid},
  year   = {2024}
}

@article{duno4,
  author  = {Due{\~n}as, J. and N{\'u}{\~n}ez, C. and Obaya, R.},
  title   = {Critical transitions for asymptotically concave or d-concave nonautonomous differential equations with applications in ecology},
  journal = {J. Nonlinear Sci.},
  year    = {2024},
  pages   = {34:105}
}

@article{dno5,
  author  = {Due{\~n}as, J. and N{\'u}{\~n}ez, C. and Obaya, R.},
  title   = {Saddle--node bifurcations for concave in measure and d-concave in measure skewproduct flows with applications to population dynamics and circuits},
  journal = {Commun. Nonlinear Sci. Numer. Simul.},
  volume  = {142},
  year    = {2025},
  pages   = {108577}
}

@article{feni,
  author  = {Fenichel, N.},
  title   = {Geometric singular perturbation theory for ordinary differential equations},
  journal = {J. Differential Equations},
  volume  = {31},
  number  = {1},
  year    = {1979},
  pages   = {53--98}
}

@book{fink,
  author    = {Fink, A. M.},
  title     = {Almost periodic differential equations},
  series    = {Lecture Notes in Math.},
  volume    = {377},
  publisher = {Springer},
  year      = {1974}
}

@book{hale,
  author    = {Hale, J. K.},
  title     = {Ordinary differential equations},
  publisher = {Wiley},
  year      = {1969}
}

@book{jonnf,
  author    = {Johnson, R. and Obaya, R. and Novo, S. and N{\'u}{\~n}ez, C. and Fabbri, R.},
  title     = {Nonautonomous linear Hamiltonian systems: oscillation, spectral theory and control},
  series    = {Developments in Mathematics},
  volume    = {36},
  publisher = {Springer},
  year      = {2016}
}

@article{jnot,
  author  = {Jorba, A. and N{\'u}{\~n}ez, C. and Obaya, R. and Tatjer, J. C.},
  title   = {Old and new results on strange nonchaotic attractors},
  journal = {Internat. J. Bifur. Chaos},
  volume  = {17},
  number  = {11},
  year    = {2007},
  pages   = {3895--3928}
}

@book{kaha,
  author    = {Katok, A. and Hasselblatt, B.},
  title     = {Introduction to the modern theory of dynamical systems},
  publisher = {Cambridge University Press},
  year      = {1995}
}

@book{klra,
  author    = {Kloeden, P. E. and Rasmussen, M.},
  title     = {Nonautonomous dynamical systems},
  series    = {Mathematical Surveys and Monographs},
  volume    = {176},
  publisher = {AMS},
  year      = {2011}
}

@book{kueh,
  author    = {Kuehn, C.},
  title     = {Multiple time scale dynamics},
  series    = {Appl. Math. Sci.},
  volume    = {191},
  publisher = {Springer},
  year      = {2015}
}

@article{liyt,
  author  = {Li, H. and You, Y. and Tu, J.},
  title   = {Random attractors and averaging for non-autonomous stochastic wave equations with nonlinear damping},
  journal = {J. Differential Equations},
  volume  = {258},
  year    = {2015},
  pages   = {148--190}
}

@book{lomu,
  author    = {Lochak, P. and Meunier, C.},
  title     = {Multiphase averaging for classical systems},
  series    = {Appl. Math. Sci.},
  volume    = {72},
  publisher = {Springer},
  year      = {1988}
}

@article{lnor,
  author  = {Longo, I. P. and N{\'u}{\~n}ez, C. and Obaya, R. and Rasmussen, M.},
  title   = {Rate-induced tipping and saddle-node bifurcation for quadratic differential equations with nonautonomous asymptotic dynamics},
  journal = {SIAM J. Appl. Dyn. Syst.},
  volume  = {20},
  number  = {1},
  year    = {2021},
  pages   = {500--540}
}

@article{loos,
  author  = {Longo, I. P. and Obaya, R. and Sanz, A. M.},
  title   = {Tracking nonautonomous attractors in singularly perturbed systems of ODEs with dependence on the fast time},
  journal = {J. Differential Equations},
  volume  = {414},
  year    = {2025},
  pages   = {609--644}
}

@article{miln,
  author  = {Milnor, J.},
  title   = {On the concept of attractor},
  journal = {Comm. Math. Phys.},
  volume  = {99},
  year    = {1985},
  pages   = {177--195}
}

@book{mitr,
  author    = {Mitropolsky, Yu. A.},
  title     = {Problems of the asymptotic theory of nonstationary vibrations},
  publisher = {Israel Program for Scientific Translations},
  year      = {1965}
}

@article{near,
  author  = {Newman, Julian and Ashwin, Peter and Rasmussen, Martin},
  title   = {Return-time and pullback-convergence properties of statistical attractors},
  journal = {Nonlinearity},
  volume  = {38},
  year    = {2025},
  pages   = {045022},
}

@article{noos4,
  author  = {Novo, S. and Obaya, R. and Sanz, A. M.},
  title   = {Stability and extensibility results for abstract skew-product semiflows},
  journal = {J. Differential Equations},
  volume  = {235},
  number  = {2},
  year    = {2007},
  pages   = {623--646}
}

@article{nuor1,
  author  = {N{\'u}{\~n}ez, C. and Obaya, R. and Rodr{\'\i}guez, J.},
  title   = {A note on the averaging principle for ordinary differential equations depending on the slow time},
  journal = {Appl. Math. Lett.},
  volume  = {178},
  year    = {2026},
  pages   = {109910}
}

@article{olar,
  author  = {Olja\u{c}a, Lea and Ashwin, Peter and Rasmussen, Martin},
  title   = {Measure and statistical attractors for nonautonomous dynamical systems},
  journal = {J. Dyn. Differ. Equ.},
  volume  = {36},
  number  = {3},
  year    = {2024},
  pages   = {2375--2411}
}

@article{orta,
  author  = {Ortega, R. and Tarallo, M.},
  title   = {Almost periodic linear differential equations with non-separated solutions},
  journal = {J. Funct. Anal.},
  volume  = {237},
  year    = {2006},
  pages   = {402--426}
}

@book{rudi2,
  author    = {Rudin, W.},
  title     = {Fourier analysis on groups},
  publisher = {Wiley},
  year      = {1990}
}

@book{save,
  author    = {Sanders, J. A. and Verhulst, F.},
  title     = {Averaging methods in nonlinear dynamical systems},
  publisher = {Springer},
  year      = {1985}
}

@book{savm,
  author    = {Sanders, J. A. and Verhulst, F. and Murdock, J. A.},
  title     = {Averaging methods in nonlinear dynamical systems},
  publisher = {Springer},
  year      = {2007}
}

@book{sell2,
  author    = {Sell, G. R.},
  title     = {Topological dynamics and ordinary differential equations},
  publisher = {Van Nostrand Reinhold},
  year      = {1971}
}

@article{selg,
  author  = {Selgrade, J. F.},
  title   = {Isolated invariant sets for flows on vector bundles},
  journal = {Trans. Amer. Math. Soc.},
  volume  = {203},
  year    = {1975},
  pages   = {359--390}
}

@book{shyi4,
  author    = {Shen, W. and Yi, Y.},
  title     = {Almost automorphic and almost periodic dynamics in skew-product semiflows},
  series    = {Mem. Amer. Math. Soc.},
  volume    = {647},
  publisher = {AMS},
  year      = {1998}
}

@article{tikh,
  author  = {Tikhonov, A. N.},
  title   = {Systems of differential equations containing a small parameter multiplying the derivative},
  journal = {Mat. Sb.},
  volume  = {31},
  year    = {1952},
  pages   = {575--586}
}

@book{verl,
  author    = {Verhulst, F.},
  title     = {Methods and applications of singular perturbations},
  series    = {Texts in Appl. Math.},
  volume    = {50},
  publisher = {Springer},
  year      = {2005}
}

@article{vich,
  author  = {Vishik, M. I. and Chepyzhov, V. V.},
  title   = {Averaging of trajectory attractors of evolution equations with rapidly oscillating terms},
  journal = {Sb. Math.},
  volume  = {192},
  year    = {2001},
  pages   = {11--47}
}

@book{walt,
  author    = {Walters, P.},
  title     = {An introduction to ergodic theory},
  publisher = {Springer},
  year      = {1982}
}

@article{walk,
  author  = {Wang, Y. and Li, D. and Kloeden, P. E.},
  title   = {On the asymptotical behavior of nonautonomous dynamical systems},
  journal = {Nonlinear Anal.},
  volume  = {59},
  number  = {1--2},
  year    = {2004},
  pages   = {35--53}
}
\end{document}